\documentclass[letterpaper,12pt,oneside]{report}
\usepackage[T1]{fontenc}
\usepackage{amsmath,amsfonts,amsthm,subcaption,graphicx,mathrsfs,amssymb,booktabs,latexsym,url,xspace,tikz-cd}
\usepackage[english]{babel}
\usepackage[latin1]{inputenc}
\usepackage{bm}
\usepackage{geometry}
\usepackage{xcolor}
\usepackage{setspace}
\usepackage{fancyhdr}
\usepackage{titlesec}
\usepackage[nottoc]{tocbibind}
\usepackage{etoolbox}

\IfFileExists{algpseudocode.sty}{%
  \usepackage{algpseudocode}
  \usepackage{algorithm}
  
}{%
  \usepackage{float}
  \newfloat{algorithm}{tbp}{loa}[chapter]
  \floatname{algorithm}{Algorithm}
  \newenvironment{algorithmic}
    {\begin{list}{}{\setlength{\leftmargin}{1.5em}\setlength{\itemsep}{0.2em}\setlength{\parsep}{0pt}}}
    {\end{list}}
  \newcommand{\State}{\item[]}
  \newcommand{\Require}{\item[\textbf{Inputs:}]}
  \newcommand{\Return}{\item[]\textbf{return} }
  \newcommand{\For}[1]{\item[]\textbf{for }##1\textbf{ do}\begin{list}{}{\setlength{\leftmargin}{1.5em}\setlength{\itemsep}{0.2em}\setlength{\parsep}{0pt}}}
  \newcommand{\EndFor}{\end{list}\item[]\textbf{end for}}
  \newcommand{\While}[1]{\item[]\textbf{while }##1\textbf{ do}\begin{list}{}{\setlength{\leftmargin}{1.5em}\setlength{\itemsep}{0.2em}\setlength{\parsep}{0pt}}}
  \newcommand{\EndWhile}{\end{list}\item[]\textbf{end while}}
}

\fancypagestyle{plain}{%
  \fancyhf{}
  \fancyfoot[C]{\thepage}

}
\titleformat{\chapter}[display]
  {\normalfont\bfseries\centering}
  {\thechapter}
  {0.75em}
  {\Large\MakeUppercase}
\titlespacing*{\chapter}{0pt}{0pt}{2em}
\titleformat{\section}
  {\normalfont\large\bfseries}
  {\thesection}
  {1em}
  {}

\addto\captionsenglish{%
}

\AtBeginEnvironment{figure}{\singlespacing}
\AtBeginEnvironment{table}{\singlespacing}
\AtBeginEnvironment{algorithm}{\singlespacing}

\theoremstyle{definition}
\newtheorem{theorem}{Theorem}
\newtheorem{lemma}[theorem]{Lemma}
\newtheorem{corollary}[theorem]{Corollary}

\newtheorem{definition}[theorem]{Definition}

\newtheorem{observation}[theorem]{Observation}

\newcommand{\thesistitle}{Recurrent Barycenter Encoder: Prediction, Simulation and Applications to Time Series Analysis}
\newcommand{\thesisauthor}{Javier Ignacio Almarza}
\newcommand{\thesisdepartment}{Department of Mathematics}
\newcommand{\thesisuniversity}{New York University}
\newcommand{\thesisdate}{August, 2026}
\newcommand{\thesisadvisor}{Esteban Tabak}

\begin{document}
\pagenumbering{roman}
\thispagestyle{empty}
\singlespacing
\begin{center}
  \vspace*{0.25in}
  {\large\bfseries \thesistitle\par}

  \vspace{0.75in}
  by

  \vspace{0.35in}
  {\large \thesisauthor\par}

  \vfill

  A dissertation submitted in partial fulfillment

  of the requirements for the degree of

  \vspace{0.35in}
  Doctor of Philosophy

  \vspace{0.35in}
  \thesisdepartment

  \vspace{0.2in}
  \thesisuniversity

  \vspace{0.35in}
  \thesisdate
\end{center}

\vfill

\begin{flushright}
  \begin{minipage}{2.75in}
    \centering
    \rule{\linewidth}{0.5pt}\\[0.5ex]
    \thesisadvisor
  \end{minipage}
\end{flushright}

\clearpage


\chapter*{ACKNOWLEDGEMENTS}
\addcontentsline{toc}{chapter}{ACKNOWLEDGEMENTS}
\doublespacing

\noindent\textit{I would like to thank my advisor, Prof. Esteban Tabak, whose feedback, constant support and original ideas made this thesis possible, and from whom I learned a stimulating way of approaching problems in applied mathematics. I also thank Prof. Petter Kolm for his careful reading of this work and for several helpful remarks that helped me improve it. I want to thank all the mathematicians, and especially among them Prof. Pablo Shmerkin and Prof. Rom\'{a}n Sasyk, who taught me to do and, most importantly, to appreciate mathematics.\\
I am thankful to my parents and family for all their support.\\
 And last but not least I want to express my immense gratitude to Sofi, my wife, without whose love I would have never mustered the strength to resume my PhD and bring this work to completion. You mean everything to me.}

\clearpage

\chapter*{ABSTRACT}
\addcontentsline{toc}{chapter}{ABSTRACT}
\doublespacing

\hspace*{\parindent}This dissertation has two parts. In the first part we theoretically develop a framework that extends the optimal transport barycenter problem from independent data with observed covariates to time-dependent processes whose relevant covariate is their own history. Given a stationary, square-integrable stochastic process, our framework shows that under reasonable regularity assumptions there exists an invertible transport map that extracts from it an iid component whose distribution is the barycenter of the conditional distributions of the input process on its own past. This extracted iid sequence thus captures the innovation or noise component of the process that is not explained by the process' own past.

In the second part of this thesis we implement an effective algorithm to solve this problem and test it against synthetic data generated with some well-known models in time series analysis for which the solution is known in closed-form. We then show how this method can be applied to build a predictive sampler for such processes, and assess the accuracy of this predictive sampler both with the aforementioned synthetic data and with real-world time series such as the SP500 and Bitcoin 4-hour returns, and the daily records of the North Atlantic Oscillation index from the NOAA Climate Prediction Center. Finally, we demonstrate an application of the sampler to path simulation with a well-known time series model.

\clearpage
\doublespacing

\tableofcontents

\clearpage
\listoffigures

\clearpage
\listoftables

\clearpage
\chapter*{LIST OF APPENDICES}
\addcontentsline{toc}{chapter}{LIST OF APPENDICES}
\singlespacing

\noindent APPENDIX A\quad DEFERRED PROOFS\dotfill\pageref{app:proofs}

\vspace{1.25em}

\noindent APPENDIX B\quad IMPLEMENTATION DETAILS\dotfill\pageref{appimplementation}

\clearpage
\pagenumbering{arabic}
\doublespacing

\chapter{Introduction}

\section{Background}

The fundamental aim in statistical learning is to extract underlying structural information from datasets, enabling critical tasks such as prediction and inference from unseen data and generation of new samples that exhibit the same structure or obey the same stochastic law as the original dataset. This goal of capturing latent dependence structures applies equally to time-independent or cross-sectional data and to time-dependent sequential processes. In the first case, the common assumption is that the observed dataset $x_1,\dots, x_N\in\mathbb{R}^n$ is an iid sample taken from an unknown underlying distribution $P_X$ on $\mathbb{R}^n$. In the time-dependent case, it is assumed that the data is a realization $x_1,\dots,x_T$ of some stochastic process following a law $P_X$.

Often the variable $x\in\mathbb{R}^n$ is observed together with another variable $z\in\mathbb{R}^d$ and the structure that one aims to extract is the conditional law $P(x\in\cdot\mid z)$. In recent years, a `barycentric'\footnote{The reason for this name will become clear in the next section.} methodology based on optimal transport has been developed to extract this conditional structure given samples $(x_1,z_1),\dots,(x_N,z_N)$ drawn from $P(x,z)$, by finding a random variable $y^*$ that solves the problem
\begin{equation}\label{eqf}\min_y \mathbb{E}_P[\|x-y\|^2]\quad\text{ s.t. } y{\perp\!\!\!\perp} z
\end{equation}
where $y$ may be any suitable random variable\footnote{Strictly speaking, the problem as written is ill-posed since there may be no random variable $y$ minimizing the transport cost subject to the independence constraint. We will later see how to properly define this problem.} and ${\perp\!\!\!\perp}$ stands for independence, yielding successful applications to inference, simulation and latent factor discovery\cite{Yang22}.

Our goal in this work is to extend this methodology, which has been so far limited to iid data, to the analysis of data that exhibits time-dependence. In other words, we want to extract some dependence structure from time-indexed sample paths $(x_t)_{t\in\mathbb{N}}$ that obey a given stochastic law. This law could be known (for example, the law of a well-known stochastic process like a linear Gaussian or a GARCH process) or unknown, like the law underlying the returns of some set of financial assets.

To illustrate what kind of structure extraction we are aiming at, we may start by assuming there is a stochastic process $(x_t,z_t)_{t\in\mathbb{T}}$ ($\mathbb{T}=\mathbb{N}\text{ or }\mathbb{Z}$) with law $P$ ($x_t\in\mathbb{R}^{d_x},z_t\in\mathbb{R}^{d_z}$) underlying our dataset. In the most general setting, the variable $z$ will represent a source of variability for $x$. It may be known and observable, or it may be some latent or hidden state that must be inferred, or a combination of observable and hidden variables, $z_t=(z_t^o,z_t^h)$. For example, in financial time series, the observable $z^o$ could be exogenous factors that can explain individual returns, such as macroeconomic variables (inflation, GDP) or market or factor returns, while unobservable variables $z^h$ could be volatility or some latent ``market sentiment''. When the data is time-dependent, $z_t$ may also consist of past observations of $x_t$, i.e. $z_t=(x_{t-1},\dots,x_{t-l})$ that explain its current variability. Indeed, our general focus in this work will be on using the methodology to extract this kind of dependence, since cross-sectional dependence (i.e. analyzing $P(x_t\mid z_t)$, when $z_t$ is an exogenous variable) has already been the object of most applications of the barycentric method so far. Nonetheless, both approaches may naturally be combined.

Extracting the structural dependence of a stochastic process on its own past entails expressing it at time $t$ as a function of $(x_s)_{s<t}$ plus an innovation\footnote{We will use `innovation' to mean not only a strictly additive residual but rather any `noise' variable $\epsilon_t$, independent of the history $x_{< t}$, from which, when combined with this history, the next observation $x_{t}$ can be recovered, using any measurable transformation.} $\epsilon_t$ which captures the variability in $x_t$ not explained by its own past.

\begin{equation}\label{eq111}x_t=F((x_s)_{s<t},\epsilon_t)\end{equation}

A distinction must be made here regarding the innovation term. If we only required, as is often done in time series analysis, that the $\epsilon_t$ were zero-mean, constant variance and uncorrelated (in other words, if $\epsilon_t$ were a \textit{weak} white noise), then it is a classical result that the existence of orthogonal complements in Hilbert spaces can be used to prove that such a decomposition always exists in linear form. 
\begin{theorem}[Wold's theorem]
Let $(x_t)_{t\in\mathbb{Z}}$ be a weakly stationary\footnote{This condition means that $\mathbb{E}[x_t]$ is constant in $t$, $\mathbb{E}[x_t^2]<\infty$ and $\operatorname{Cov}(x_t,x_{t-l})$ depends only on the lag $l$.}, zero-mean process. Then $x_t$ admits a decomposition as a sum of two uncorrelated components
$$x_t=d_t+\sum_{j=0}^\infty c_j\epsilon_{t-j}$$ where:
\begin{itemize}
\item a) $\sum_jc_j^2<\infty$,
\item b) the first component $d_t$ (the `deterministic part') belongs to  $\bigcap_{s \in \mathbb{Z}} \mathcal{H}_s$, where $\mathcal{H}_s = \overline{\text{span}}\{x_l : l \le s\}$, and
\item c) the $\epsilon_s$'s in the second component are a weak white noise, i.e. for $c_{s-t}\neq 0$, $\mathbb{E}[\epsilon_t]=0$, $\mathbb{E}[\epsilon_s^2]=1$ and $\mathbb{E}[\epsilon_{s'}\epsilon_s]=0$ for $s\neq s'$
\end{itemize}
\end{theorem}

If, however, we want a decomposition where each extracted noise term is independent of the past then the first obstruction is that the function $F$ in (\ref{eq111}), which must exist in a measurable sense for standard Borel spaces, is nonunique and generally intractable, providing no useful procedure for innovation extraction. A nice decomposition like the linear one guaranteed to exist by Wold's theorem will generally not hold. It is precisely this case that is of interest to us: we want the extracted innovation to be not only decorrelated from the observed past, but also free of any higher-order dependence. This is so because many processes of interest, such as zero-mean GARCH return processes, are already free of any autocorrelation, i.e. they are already a weak white noise and therefore have a trivial Wold decomposition, yet their higher order dependence is strong \footnote{GARCH processes specifically show high volatility clusters, i.e. variance tends to be high if it was recently high.}.

A classical theoretical solution to extracting strictly independent innovations is the Rosenblatt transformation \cite{Rosenblatt52}, which maps the sequence into iid uniform random variables by evaluating each observation through its conditional cumulative distribution function\footnote{Here we are assuming that $F_t$ (equivalently, the conditional distribution) is continuous. Later we will see this assumption guarantees a Monge solution in the Optimal Transport framework for one-dimensional $x_t$ (in higher dimensions, Monge solutions require absolute continuity).}, i.e.
$$u_t = F_t(x_t),\quad\text{ where }F_t(a):=P(x_t\leq a \mid x_{s},s<t)$$

However, while this transformation guarantees exact independence, it provides no practical extraction procedure for complex or unknown data-generating processes, since it requires either full knowledge of the conditional distributions or at least an accurate estimation thereof. Yet estimating, or, rather, sampling, these conditional distributions is a hard problem, and it is indeed one of the main ultimate goals of the barycentric method. 

\section{The barycentric formulation}

In the barycentric approach for this problem, the independence requirement is captured by the constraint in (\ref{eqf}): the condition $y{\perp\!\!\!\perp} z$ is much stronger than $y\perp z$. It requires that the joint distribution of $(y, z)$ factorizes into the product of their respective marginal distributions. It is a nonlinear, nonconvex condition that operates over the entire space of probability measures instead of merely enforcing an orthogonality condition in a Hilbert space. As a consequence, any approach that relies on the linear structure of Hilbert spaces (such as Wold's) will be insufficient. The problem must be recast as a nonlinear optimization over the space of probability measures, a setting which naturally leads to optimal transport theory.

An important property of the optimal transport barycentric problem (OTBP) is that the distribution $\mu^*$ of the random variable $y$ that solves (\ref{eqf}) is the {\em barycenter} of the conditional distributions $P_z(x):=P(x\mid z)$ under the 2-Wasserstein distance, that is, it solves the problem (see Lemma 1 in \cite{Tabak25})

\begin{equation}\label{eqwasser}\min_{\mu \in \mathcal{P}_2(\mathbb{R}^{d_x})} \mathbb{E}_{z}\left[ W_2^2(P_z,\mu)\right]
\end{equation}
where $\mathcal{P}_2$ stands for probability measures of finite variance. Intuitively, the barycenter $y$ is meant to capture the variability in $x$ not explained by the covariate $z$.

In our time-dependent setting, we would like $z=z_t$ to capture all relevant autoregressive information of the observed past of $x$ before time $t$. The process formed by the $y_t$ extracted at each time $t$ would then form the desired sequence of iid innovations that, together with $z_t$, determine $x_t$. Finding such a $z_t$ is not, however, a trivial task.

For the sake of completeness we state now the standard existence and uniqueness results for the barycentric problem (\ref{eqwasser}).

\begin{theorem}\label{teogeneralotb}
Let $(x,z)$ be jointly distributed, with
$x$ square-integrable, and let $P_z=P(x\in\cdot\mid z)$. Assume that,
for almost every $z$, $P_z$ is atomless if $d_x=1$, and absolutely
continuous with respect to Lebesgue measure if $d_x\geq2$. Then
(\ref{eqf}) and (\ref{eqwasser}) admit unique solutions, have the same
value, and the distribution of the solution $y^*$ to (\ref{eqf}) is the
barycenter $\mu^*$ solving (\ref{eqwasser}). Moreover,
$$y^*=T_z(x)$$
where $T_z$ is the unique quadratic optimal transport map from $P_z$ to
$\mu^*$. For almost every $z$, $T_z$ is invertible up to null sets and $(T_z^{-1})_\#\mu^*=P_z$.
\end{theorem}

Existence and equivalence follow from
\cite[Theorems 1 and 2]{Yang22} and
\cite[Theorem 2]{LeGouicLoubes17}. In dimension one, uniqueness of the barycenter and its quantile representation follow from \cite[Proposition 4.1]{BigotEtAl17}, while existence and uniqueness of the monotone optimal transport map follow from \cite[Theorems 2.5 and 2.9]{Santambrogio15}. In dimensions $d_x\geq2$, uniqueness and absolute continuity of $\mu^*$ follow from \cite[Theorem 5.1]{Ma26}. Invertibility follows in the $d_x\geq 2$ case by applying \cite[Theorem 2.12]{Villani03} in both directions. In the $d_x=1$ case it follows from applying again \cite[Theorems 2.5 and 2.9]{Santambrogio15} in both directions after noticing that $\mu^*$ must also be atomless as a consequence of the quantile representation.

A methodology based on implementing effective solutions to (\ref{eqf}) (or, equivalently, (\ref{eqwasser})) has many applications, especially when the conditions guarantee that the Monge transport map $T$ on $(x,z)$ exists and is invertible for fixed $z$. The inverse 
$$T_z^{-1}:\operatorname{Im}(T_z)\to\mathcal{X}$$
 then provides a route for conditional sampling and conditional density estimation that is sample-efficient and is not limited to a predetermined parametric density form, since $(T_z^{-1})_\#\mu^*$ equals the conditional distribution $P_z$.

As mentioned above, so far in almost all applications of this optimal transport barycentric method the datasets are not meant to represent time series or stochastic processes, and are therefore iid for all practical purposes. If the samples $(x_i,z_i)$ are iid then the barycentric variables $y_i=T(x_i,z_i)$ are iid too. This iid property of the barycentric variables is crucial for the effectiveness of applications like conditional sampling, for which we first sample uniformly from the barycentric variables $y_i=T(x_i,z_i)$ obtained from the dataset $(x_i,z_i)$ and then we transport them backwards with the inverse map $T_{z}^{-1}$ to generate new datapoints $x_i'=T_{z}^{-1}(y_i)$ that are also iid (with law  $P(\cdot\mid z)$). Since our goal is to apply this method to datasets that come from time-dependent processes which are not iid, some important modifications will be needed.

\section{Outline of the thesis}

The rest of the thesis is organized in two parts, corresponding to Chapters 2 and 3, with Chapter 4 left for conclusions. Chapter 2 theoretically develops the population level optimal transport framework for time-dependent processes, while Chapter 3 implements this framework numerically and tests the implementation with both synthetic and real-world data.
Within Chapter 2, we first examine Markov processes, for which lagged observations provide natural covariates and the resulting barycentric variables are easily shown to form an iid innovation sequence. After illustrating the construction for some well-known processes used in time series analysis such as AR and ARCH, we then consider non-Markov processes such as GARCH and linear Gaussian state space models.

The recurrent adversarial formulation necessary to address this more general setting is then developed, introducing a recurrent variable whose role is to encode the information from the observed past that is relevant to the conditional distribution of $x_t$. We then establish general existence and uniqueness results and prove that, under general regularity assumptions, the barycentric variables that solve the problem are iid. We also study and give closed-form formulas for the linear Gaussian state space model and the GARCH$(1,1)$ model.

In order to numerically implement these ideas, in Chapter 3 we approximate the independence constraint through adversarial correlation tests and examine the distinction between the original infsup problem and the more stable supinf optimization used by the algorithm. A duality criterion for equality of their values is established and verified for AR, GARCH and linear Gaussian state space models. A simple example of a model for which this equality fails, AR--GARCH, is introduced and a solver with several test-function channels is then presented to address these and more general models. Experiment results are then presented showing that the solver accurately recovers the barycentric solution.

The solver's learned transport map and barycentric variables are then used to construct a predictive sampler, the Recurrent Barycentric Encoder (RBE). We evaluate this sampler on synthetic AR, GARCH, Kalman (i.e. linear Gaussian SSM) and AR--GARCH data, at individual histories and through systematic out-of-sample tests. We also examine prediction at increasing horizons. The method is then applied to historical Bitcoin returns, SP500 returns and the North Atlantic Oscillation index. Its predictive performance is assessed using CRPS against different alternative models, and the distribution of the barycenter variables is evaluated using Kolmogorov-Smirnov and Ljung-Box tests. Finally, we discuss path simulation and present an experiment to evaluate the OTBP's potential to simulate an ARCH$(1)$ process.

\chapter{An optimal transport framework for time-dependent processes}

\section{The Markov case}\label{subs12}

To see how the barycentric methodology can be applied to time-dependent processes let us begin with the simplest scenario and assume that the process $(x_t)_{t\in\mathbb{N}}$ is both stationary and Markov, so that all dependence on the past is captured by the last value of the process. Then (\ref{eq111})
becomes
\begin{equation}\label{eqmain}x_{t}=F(x_{t-1}, \epsilon_t)
\end{equation}
for some deterministic function $F$ and iid noise $\epsilon_t$.

The barycentric framework requires a covariate $z$ which explains part of the variability in $x$ and with respect to which the barycentric variable $y$ must be independent. We will see that for general stochastic processes finding such a covariate $z_t$ that captures the dependence of $x_t$ on its past will not be trivial. Under the Markov hypothesis, however,  there is a natural candidate for a covariate $z$ that captures the dependence of $x_t$ on its own past, namely $x_{t-1}$. Notice that in the case of a Markov process of step $l$ this remains true with natural candidate $z_t=(x_{t-1},\dots,x_{t-l})$.

Thus, given a sample path $x_1,\dots,x_T$ of a Markov process of step 1, we can apply the OTB method to $x_t$ with covariate $z=x_{t-1}$ and obtain barycentric variables $y_t$ which are solutions to

$$\arg\min_{y_t} \mathbb{E}[\|x_t-y_t\|^2]\ \text{ s.t. }\ y_t{\perp\!\!\!\perp} x_{t-1}$$
where the minimization problem is independent of $t$ because of stationarity. As explained above, the minimization problem seeks to extract, through $y_t$, the variability in $x_t$ not explained by its immediate past, where `explained by' is understood in the strong sense of probabilistic dependence. Since the process is Markov, this amounts to any variability attributable to the whole past.

A crucial property for our prediction and simulation purposes is that the barycentric variables $Y=(y_t)_{t=1..T}$ obtained by solving (\ref{eqf}) for all $t$ are iid. Indeed, assuming these have been obtained from a training sample path $x_1,\dots,x_T$, if we later observe a new realization $x_1',\dots,x_s'$ of the same process up to some time $s$ then the previously obtained $Y$ can be used to sample the conditional predictive distribution $P(x_{s+1}\in\cdot|x_{l},l\leq s)=P(x_{s+1}\in\cdot|x_{s})$. Indeed, the push-forward $T_{x_s}^{-1}(\mu^*)$ of $\mu^*$, the distribution of $y_t$ for all $t$, equals $P(x_{s+1}\in\cdot|x_{s})$, by Theorem \ref{teogeneralotb}. For a practical implementation, $\mu^*$ can be approximated by $\widehat{\mu}$, the empirical distribution of $Y$, or by fitting a density estimate to it. In either case, the crucial property is that the barycentric variables $y_t$ in $Y$ are iid, since this makes the time indices exchangeable. For the Markov case, this iid property is easy to prove, as we will see now.

\begin{lemma}\label{lemma1} Let $(x_t)_{t\in \mathbb{N}}$ be a stationary Markov process. Assume the barycentric process $y_t=\arg \min_{y_t} \mathbb{E}[\|x_t-y_t\|^2]\quad\text{ s.t. }\quad y_t{\perp\!\!\!\perp} x_{t-1}$ is a Monge solution, i.e. it is given by a transport map $y_t=T(x_t,x_{t-1})$. Then the $y_t$ are iid.
\end{lemma}
\begin{proof}
    The ``identically'' follows from the fact that $x_t$ is stationary and that $y_t$ is a function of $x_t$ and $x_{t-1}$. For independence we have
    \begin{align}
        \mathbb{P}(y_{t+1}\in A)&=\mathbb{P}(y_{t+1}\in A\mid x_{t}) \quad (\hbox{since } \ y_{t+1}{\perp\!\!\!\perp} x_{t}) \nonumber \\
        &=\mathbb{P}(T(x_{t+1},x_t)\in A\mid x_{t})=\mathbb{P}(T(x_{t+1},x_t)\in A\mid x_{s},s\leq t) \quad \text{ (Markov)} \nonumber\\
        &= \mathbb{P}(y_{t+1}\in A\mid x_{s},s\leq t) \nonumber
    \end{align}
    Hence $y_{t+1}{\perp\!\!\! \perp}\{x_s,s\leq t\}$ and therefore $y_{t+1}{\perp\!\!\! \perp} (T(x_s,x_{s-1}))_{s\leq t}=(y_s)_{s\leq t}$
\end{proof}

\begin{observation}\label{obss}
The proof works the same with Markov processes of step $k$ as long as we take the last $k$ observations as covariates, i.e. $z=x_{t-1},\dots,x_{t-k}$, yielding $y_t=T(x_t,x_{t-1},\dots,x_{t-k})$. 
\end{observation}

The proof also works the same if we add other covariates.

\begin{lemma}\label{lemmaX} Let $(x_t)_{t\in \mathbb{N}}$ satisfy a stochastic dynamics of the form
$$x_{t}=F(x_{t-1}, \epsilon_t, z^o_t), $$
where $z^o_t$ is a set of observed covariates that are jointly stationary with $x_t$ and $\epsilon_t$ is iid and
$$\epsilon_t\perp\!\!\!\perp\sigma(x_s,z_u^o:s<t,u\leq t)$$ Assume the barycentric process
$$y_t=\arg \min_{y_t} \mathbb{E}[\|x_t-y_t\|^2]
\quad\text{ s.t. }\quad
y_t{\perp\!\!\!\perp} (x_{t-1}, z^o_t)$$
is a Monge solution, i.e. it is given by a transport map $y_t=T(x_t,x_{t-1}, z^o_t)$. Then the $y_t$ are iid.
\end{lemma}
\begin{proof}
    The ``identically'' follows from joint stationarity of $(x_t,z_t^o)$ and the time-homogeneity of the transport map $T$.
    For independence we have
    \begin{align}
        \mathbb{P}(y_{t+1}\in A)&=\mathbb{P}(y_{t+1}\in A\mid x_{t}, z^o_{t+1}) \quad (\hbox{since } \ y_{t+1}{\perp\!\!\!\perp} x_{t}, z^o_{t+1}) \nonumber \\
        &=\mathbb{P}(T(x_{t+1},x_t, z^o_{t+1})\in A\mid x_{t}, z^o_{t+1})=\mathbb{P}(T(x_{t+1},x_t, z^o_{t+1})\in A\mid x_{s}, z^o_{s+1},s\leq t) \nonumber\\
        &= \mathbb{P}(y_{t+1}\in A\mid x_{s}, z^o_{s+1},s\leq t) \nonumber
    \end{align}
    Hence $y_{t+1}{\perp\!\!\! \perp}\{x_s,z^o_u,s\leq t,u\leq t+1\}$ and therefore $y_{t+1}{\perp\!\!\! \perp} (T(x_s,x_{s-1}, z^o_s))_{s\leq t}=(y_s)_{s\leq t}$
\end{proof}

Observation \ref{obss} grants that, when the process is Markov, the barycentric method can remove all the variability attributed to the past simply by taking $z$ equal to the vector of the last $k$ lagged values of the process. As we will later see, finding such a $z$ when the process is not Markov is not an easy task.

\subsection{Path simulation}

In addition to prediction we may also want to simulate new data, that is, generate paths $x_1',\dots,x_T'$ that follow the same law as the original $x_t$. The simplicity of the Markov case makes it easy to see how the iid barycentric variables $y_t$ can be used to simulate alternative paths.

Indeed, starting with a fixed $x_1'$, we can sample from $\widehat \mu\approx \mu^*$ by choosing $i\in \{1,\dots,T\}$ uniformly at random\footnote{An alternative to this uniform sampling is to fit some density estimate $\hat{\mu}_{\text{dens}}$ for the $y_t$ and then sample from this density $\hat{\mu}_{\text{dens}}$.} and apply the inverse of the transport map to define $x_2'=T_{x_1'}^{-1}(y_i)$. Sampling the iid barycentric variables uniformly and applying the inverse transport map $T_{x_1'}^{-1}$ will generate a sample $x_2'$ that follows the conditional law $P(\cdot \mid x_1')$. Thus, we get a sample path that follows the original Markov law by iteratively applying this procedure, i.e. $x_t'=T_{x_{t-1}'}^{-1}(y_{i_t})$ where $i_t\in \{1,\cdots,T\}$ is chosen uniformly at random. In Section \ref{pathsimulation} we implement this idea and run a series of tests to assess its accuracy.

For a variant of this where cross-sectional data is added, consider the situation where the initial sample involves two observed variables $(x_1,z_1),\dots,(x_T,z_T)$ such that the process is jointly Markov. We may want to sample paths $(x_1',z_1),\dots,(x_T',z_T)$ that follow the same law but where the $z_t$ is kept the same as in the original dataset.

Applying the OTB algorithm as before to the conditional distribution\footnote{Notice that in this scenario the covariate $z$ corresponds to $x_{t-1},z_t,z_{t-1}$, i.e. we are extracting both cross-sectional and past dependence.} $P(x_t\mid x_{t-1},z_t,z_{t-1})$ we get the barycentric variables $y_t=T(x_t,x_{t-1},z_t,z_{t-1})$, which are independent of covariates $(x_{t-1},z_t,z_{t-1})$.

If we assume that
$$P(x_{t+1}|x_{\leq t} ,z_{\leq T})=P(x_{t+1}|x_t,z_t,z_{t+1})$$

then, starting from some $(x_1',z_1)$, we choose $i\in \{1,\dots,T\}$ uniformly at random and apply the inverse of the transport map to define $(x_2')=T_{(x_1',z_2,z_1)}^{-1}(y_i)$. As before, applying this procedure iteratively we get a path $x_1',\dots,x_T'$ sampled under the law $P(x_1,\dots,x_T\mid z_1,\dots,z_T)$.

\section{Some well-known processes}\label{wellknown}

Before proceeding to the general non-Markovian scenario, let's make these ideas more concrete by exploring what the barycentric variables look like when applying this method to stochastic processes $x_t$ that follow one of a few well-known models in time series analysis, such as AR$(p)$ and ARCH$(q)$.

For a series that follows the simplest stationary AR$(1)$ process

\begin{equation}\label{eq1}
x_{t}=\alpha+\beta x_{t-1}+\epsilon_{t}
\end{equation}
with $\epsilon_{t}$ an iid white noise and $|\beta|<1$ (to ensure stationarity), the OTB method for $x_t$ with covariate $z=x_{t-1}$ yields, by Corollary \ref{1dcor2}, barycentric variables
$$y^*_t=T(x_t,x_{t-1})=\alpha+\beta \mathbb{E}[x_t]+x_t-(\alpha+\beta x_{t-1})=\alpha+\beta \mathbb{E}[x_t]+\epsilon_t$$
with transport cost
$$V=\mathbb{E}[|x_t-y^*_t|^2]=\beta^2Var(x)$$

AR$(1)$ models naturally extend to AR$(p)$ models, which express $x_t$  as a linear function of the last $p$ lags $x_{t-1},\dots,x_{t-p}$, plus a white noise term. For these models, the observation after Lemma \ref{lemma1} means that the OTB should be done with the last $p$ lags as covariates, i.e. $y_t {\perp\!\!\! \perp} z_t=x_{t-1},\dots,x_{t-p}$, since the model is Markov of step $p$.

The simplicity of AR models makes them popular, yet their simplicity entails some limitations for modeling. One of these limitations is that the distribution of $x_{t}$ conditional on the past has constant variance $\sigma_\epsilon^2$, i.e. variance is independent of past history. ARCH and GARCH-type models were developed to relax this, which is especially useful since in financial time series the variance of returns, as a measure of volatility, exhibits persistence, i.e. a spike in volatility at time $t$ makes the expected volatility at $t+1$ higher.

The general form of an ARCH or GARCH process $(\eta_t)_{t\in\mathbb{T}}$ ($\mathbb{T}=\mathbb{N}\text{ or }\mathbb{Z}$) is
$$\eta_t=\sigma_t \epsilon_t$$
where $\epsilon_t$ is an iid standardized white noise (usually Gaussian or Student-$t$) and $\sigma_t$ is a positive function of the past. For an ARCH$(q)$ process, this function is given by

$$\sigma_{t}^2=\omega+\sum_{i=1}^q\alpha_i \eta_{t-i}^2 \quad \text{ ($\omega,\alpha_i>0$)}$$   
with $\sum_i\alpha_i<1$ to ensure stationarity.

If we apply the OTB method to $\eta_t$ with covariates $(\eta_{t-1},\dots,\eta_{t-q})$ we get (from Corollary \ref{1dcor2} again)
\begin{align}y_t&=T(\eta_t,\eta_{t-1},\dots,\eta_{t-q})=\frac{\eta_t}{\sqrt{\omega+\sum_{i=1}^q\alpha_i\eta^2_{t-i}}}\mathbb{E}\left[\sqrt{\omega+\sum_{i=1}^q\alpha_i\eta^2_{t-i}}\right] =\nonumber \\
&=\epsilon_t\mathbb{E}\left[\sigma_t\right]
\end{align}
with transport cost $V=\mathbb{E}[\epsilon_t^2(\sigma_t-\mathbb{E}[\sigma_t])^2]=\mathbb{E}[\epsilon_t^2]\mathbb{E}[(\sigma_t-\mathbb{E}[\sigma_t])^2]=\mathbb{E}[(\sigma_t-\mathbb{E}[\sigma_t])^2]$

\

In the examples we have considered so far, the observed process is Markov. Both the AR$(p)$ process $x_t$ and the ARCH$(q)$ process $\eta_t$ are Markov (of step $p$ and $q$, respectively). However, most of the observed stochastic processes in the real world are not Markov. A simple related example of a model used in time series analysis where the observable process is not Markov of any step is the GARCH$(1,1)$ model, which is defined by\footnote{Notice an ARCH$(1)$ process corresponds to a GARCH$(1,1)$ with $\beta=0$.}.
$$x_t=\sigma_t \epsilon_t,\quad\text{ with $\epsilon_t$ an iid white noise and }\quad \sigma_{t+1}^2=\omega+\alpha x_t^2+\beta \sigma_t^2 \quad \text{ ($\omega,\alpha,\beta>0$)}$$
where $\alpha+\beta<1$ is assumed to ensure stationarity with finite variance. For this model, the squared process $x_t^2$ exhibits strong autocorrelation, and this is the main reason it is used to model volatility, since large squared returns tend to be followed by large squared returns, even when raw returns themselves are serially uncorrelated.

The difference from the previous examples is that now $\sigma_{t+1}$ is a function of $\sigma_t$, so that the process $x_t$ does not satisfy any relation $x_{t+1}=F(x_t,\dots,x_{t-l},\epsilon_t)$ for iid noise $\epsilon_t$. In other words, the process is not Markov of any order. Of course, the joint process $(x_t,\sigma_t)$ is Markov of order 1, but $\sigma_t$ is not observable.

In this scenario, we would like to know if (and how) the barycentric method could be used to sample the conditional law $P(x_t\mid x_{s},s<t)$. To start answering this question, in the next section we will develop a version of the barycentric method that will allow us to handle data that comes from processes exhibiting longer time-dependence.

As a final observation, notice that with a sufficiently large $q$ we can approximate a GARCH$(1,1)$ process fairly well with an ARCH$(q)$ process. To see this, expand $\sigma_t^2$ recursively from the definition of the GARCH$(1,1)$ process to get

\begin{align}\sigma_t^2&=\omega+\alpha x_{t-1}^2+\beta\sigma_{t-1}^2=\omega+\alpha x_{t-1}^2+\beta\omega+\beta\alpha x_{t-2}^2+\beta^2\sigma_{t-2}^2\nonumber \\
&=\omega\sum_{i=0}^{q-1}\beta^i+\alpha \sum_{i=1}^q\beta^{i-1}x_{t-i}^2+\beta^q\sigma_{t-q}^2\nonumber \\
&\approx \omega\frac{1-\beta^{q}}{1-\beta}+\alpha \sum_{i=1}^q\beta^{i-1}x_{t-i}^2\  \text{(for large $q$)}\nonumber
\end{align}

Thus, in principle, we could content ourselves with the fact that a GARCH process can be approximated exponentially well by an ARCH$(q)$ process with large $q$. We could then apply the barycentric method to generate alternative sample paths using $z_t=(x_{t-1},\dots,x_{t-q})$ as covariates.

This Markov approximation will hold for much more general processes than GARCH$(1,1)$. However, high dimensionality is bad and inefficient, both in general\footnote{This is the main reason why GARCH$(1,1)$ was introduced in the first place, since $q=9$ was needed to model the daily volatility of the SP500 well enough.} and specifically for our OTB algorithm, whose covariate dimension will grow with $q$. In the specific case of a GARCH$(1,1)$ process, a single, though hidden, dimension (i.e. volatility) should be enough to capture all relevant information from the past. Ideally, we would like a mechanism that captures such a sufficient statistic whenever it exists.

\section{An adversarial game for time-dependent processes}

The previous section ended with an example of an observed process that is not Markov but which becomes Markov when a hidden process becomes observable. We will next analyze what is arguably the most well-known example of a process with this characteristic, the classical linear Gaussian state space model (LG-SSM), which we will call the {\em Kalman system}. Before doing so, though, a few comments regarding our notation and setting are necessary.

For any collection $\mathcal{C}$ of random variables defined on the same probability space, we write $\sigma(\mathcal{C})$ for the $\sigma$-algebra generated by all random variables in $\mathcal{C}$. Moreover, since we will be working with a stationary process $x_t$ defined on a probability space $(\Omega,\mathcal{F},P)$ we will regard $\sigma$-algebras $\mathcal{A}\subseteq\mathcal{F}$ mod 0, i.e., $\mathcal{A}$ and $\mathcal{A}'$ are in the same equivalence class if they differ by sets of $P$-measure zero, and we will write $\mathcal{A}=\mathcal{A}'$ accordingly (and similarly for inclusions). If we say a random variable $X$ is $\mathcal{A}$-measurable it means the $\sigma$-algebra generated by it, $\sigma(X)$, is included in $\mathcal{A}$ mod 0.

We will also assume that $(\Omega,\mathcal{F},P)$ is rich enough to define in it the bilateral extension (i.e. to $t\in\mathbb{Z}$) of $x_t$. This can always be done by stationarity and Kolmogorov's extension theorem. We will often write $x_{<t}$ instead of $\{x_s:s<t\}$, so that $\sigma(x_{<t})$ will then denote the (equivalence class of the) $\sigma$-algebra generated by $\{x_s:\ s<t\}$.

Finally, when considering processes that involve a white noise $\epsilon_t$ we will assume that such white noise (in addition to being standardized and iid) has an atomless distribution (equivalently, its cdf is continuous).

\subsection{A motivating example: the linear state Gaussian model}

Consider $(x_t,z_t)$ in $\mathbb{R}^2$ a stationary process obeying

\begin{align}\label{eqHMlin}
x_{t}&=\beta z_{t}+\eta_{t} \nonumber\\
z_{t+1}&=\alpha z_t+\epsilon_t
\end{align}
where $\epsilon_t,\eta_t$ are iid Gaussian white noise with standard deviations $\tau_\epsilon$ and $\tau_\eta$ respectively.

We saw in the previous section that, if the joint process $(x_t,z_t)$ is observable, then since it is Markov and stationary, we can apply the barycentric algorithm to $(x_t,z_t)$ with covariates $(x_{t-1},z_{t-1})$  and recover the process from the barycentric variables $y_t=T((x_t,z_t),(x_{t-1},z_{t-1}))$.

Suppose now that $z_t$ is no longer observable. We would like $z_t$ to be the covariate in some version of the barycentric problem. Yet for the barycentric problem $z$ must be given, whereas for our current model $z_t$ is unknown. In fact $z_t$ cannot be known, not even if the whole observable process $\{x_t\}_{t\in\mathbb{N}}$ is known, so the covariate will have to be some different $h_t$ that captures the dependence of the observed $x_t$ on its own past. We would want this $h_t$ to capture as much of the dependence of the observed past as is necessary, and this optimization of the explanatory power of $h_t$ will be at the expense of the variability of $x_t$ captured by the barycenter $y_t$, so that our barycentric problem becomes an adversarial game. This is analogous to the adversarial game that arises in the case of iid data when the barycentric method is used for factor discovery, see \cite{Yang22}. The value $V$ of this adversarial game will be the maximal transport cost that the player that extracts $h_t$ from the past can guarantee:

\begin{equation}\label{eqlatent}V=\max_G\min_{y_t} \mathbb{E}[\|x_{t}-y_{t}\|^2]\quad \text{ s.t. } y_{t}{\perp\!\!\!\perp} h_{t}\quad \text{ with } h_{t}=G(x_{t-1},h_{t-1})\text{ stationary}
\end{equation}

As before, $y_t$ is in principle a random variable in some probability space that extends\footnote{This will be defined more rigorously in the next section, together with the appropriate function space for $G$, to show that the problem is well-posed as written. For the linear Gaussian example, though, an exact solution can be found analytically without worrying too much about these considerations.} $(\Omega,\mathcal{F},P)$ and $G:\mathcal{X}\times\mathcal{H}\to\mathcal{H}$, which models how the outer player extracts relevant information from $x_t$'s past, is taken from an appropriate function space that ensures a stationary solution $h_t$ exists. Modeling long-range dependence with a recurrence equation like the one involving $G$ and $h_t$ is standard in the machine learning literature (e.g. Recurrent Neural Networks like GRUs and LSTMs) and for us it has the additional mathematical benefit of allowing for a formulation that takes into account a single transition pair $(t,t+1)$. Stationarity then ensures that equation (\ref{eqlatent}) is independent of $t$.

Notice that the outer player cannot directly affect the squared distance it wants to maximize. It can only do so by restricting the possible $y_t$'s available to the inner player. Since $y{\perp\!\!\!\perp} w\rightarrow y{\perp\!\!\!\perp} z$ whenever $\sigma(z)\subseteq\sigma(w)$, the first player's payoff is monotone in $\sigma(h_t)$, i.e. it has an incentive to choose $G$ to try to capture as much relevant information about the observed past as possible and thus limit as much as possible the inner player's payoff.

Because the independence constraint $y_t \perp\!\!\!\perp h_t$ operates nonlinearly over the space of probability measures, finding an explicit analytical solution to the maxmin problem (\ref{eqlatent}) is generally intractable. However, to analyze the behavior of the linear Gaussian state  space model, we can temporarily consider a relaxed version of the problem, replacing the strict independence condition with a weaker $L^2$ orthogonality constraint, requiring only that $\mathbb{E}[y_t \mid h_t] = \mathbb{E}[y_t]$, i.e. that $y_t$ is orthogonal to any centered function of $h_t$,

\begin{equation}\label{eqrelax}V_{rel}=\max_G\min_{y_t} \mathbb{E}[\|x_{t}-y_{t}\|^2]\quad \text{ s.t. } \mathbb{E}[y_{t}\mid h_{t}]=\mathbb{E}[y_{t}]\quad \text{ with } h_{t}=G(x_{t-1},h_{t-1}) \text{ stationary}\end{equation}
 An optimal $y_t$ must have $\mathbb{E}[y_t]=\mathbb{E}[x_{t}]=0$ and the constraint on $y_t$ becomes $y_t\in \mathcal{V}(h_{t})^\perp$ where $\mathcal{V}(h_{t})$ is the subspace of square-integrable random variables that are $\sigma(h_{t})$-measurable. Hence, the inner min problem has a solution
 $$y_t^*=y_t^*(G)=T(x_{t},h_t)=x_{t}-\mathbb{E}[x_{t}\mid h_t]$$ 
 The outer max problem is then 
\begin{equation}\label{eqrelout}\max_G \mathbb{E}[|x_t-y_t^*|^2]=\max_G \mathbb{E}[|\mathbb{E}[x_{t}\mid h_t]|^2]=\max_G\mathbb{E}[|P_{\mathcal{V}(h_t)}x_{t}|^2]
\end{equation}

The outer player wants an $h_t$ that maximizes the $L^2$-norm of the projection of $x_{t}$ on $\sigma(h_{t})$-measurable functions. We also need $h_t$ to be causal, i.e. $\sigma(x_{<t})$-measurable, since only information from the past is available to the outer player. Indeed, the following inclusions must hold:
$$\sigma(h_t)\subseteq \sigma(h_0,x_0,\dots,x_{t-1})\subseteq\sigma(x_{< t})$$
In the next section, after introducing the relevant functional class under which the problem is well-posed, we will see that this causal property follows from the construction and uniqueness result in Lemma \ref{lemmaG}. For now we may just impose it.

Since $\sigma(h_t)\subseteq \sigma(x_{<t})$, the projection on the space of $\sigma(h_t)$-measurable functions is dominated by the projection on the space of $\sigma(x_{<t})$-measurable functions. Hence,
$$\|P_{\mathcal{V}(h_t)}(x_{t})\|_{L^2}^2=\mathbb{E}\left[|\mathbb{E}[x_{t}\mid h_t]|^2\right]\leq \mathbb{E}\left[|\mathbb{E}[x_{t}\mid x_{<t}]|^2\right]$$

$G$ would be optimal if the resulting process $h_t$ captures that part of the observed past that is relevant to predict $x_t$. We stress `relevant' because if we want $G$ to belong to a restricted class of functions (say, linear) then it would be unreasonable to expect that all information from the observed past can be recovered through $h_t$.

We also notice that, strictly speaking, $h_0$ or $\{x_t\}_{t<0}$ are not part of the `observed past'. We treat them as part of the observed past because we want the problem to be independent of $t$ and have stationary solutions so that it captures the long-term behavior of the system. The distribution for $h_0$ in a stationary solution carries information about $\{x_t\}_{t<0}$. If $h_0$ were initialized as a constant, the resulting process would generally not be stationary at finite times. In practice, though, implementations will be initialized with constants (both for $h_t$ and $x_t$) and solutions will be computed allowing for a disregarded initial burn-in period that allows the system to `warm up' or get close to the stationary solution.

\begin{observation}
(a) So far we have not used any knowledge of (\ref{eqHMlin}): the linearity of the equations, the parameters or the fact that the white noises are Gaussian. Thus (\ref{eqrelout}) holds for any general process $x_t$ whenever the relaxed constraint $\mathbb{E}[y_t\mid h_t]=\mathbb{E}[y_t] $ is imposed.

(b) There is a gauge invariance or indeterminacy in the strategy of the outer player: given an optimal $G^*$, any measurable bijection of the space $\mathcal{H}$ provides an alternative optimal $\tilde{G}$. This holds both for (\ref{eqlatent}) and for the relaxed problem (\ref{eqrelax}).
\end{observation}

Now, $\epsilon_t$ and $\eta_t$ are Gaussian so $x_t$ and $z_t$ are also Gaussian. From the theory of Kalman filters we know the expression for $\mathbb{E}[x_{t}\mid x_{<t}]$, it is a linear function of Gaussians and, hence, itself Gaussian. Moreover, from this Kalman solution we can see that $G$ may be taken linear. Indeed, the Kalman solution to the filtering problem for the latent state is given by
\begin{align}
    \mathbb{E}[z_{t}\mid x_{<t}]&=(1-K\beta )\alpha \mathbb{E}[z_{t-1}\mid x_{<t-1}]+ K \alpha x_{t-1}\nonumber    
\end{align}
with constant $K$ independent\footnote{Notice that the stationarity assumption is essential for this.} of $t$
\begin{align}
    K&=\frac{\beta\hat{S}_{t\mid t-1}}{\beta^2\hat{S}_{t\mid t-1}+\tau_\eta^2}\nonumber\\
    \hat{S}_{t\mid t-1}&= \mathbb{E}\left[\left(z_t-\mathbb{E}[z_t\mid x_{<t}]\right)^2\mid \mathcal{F}_{t-1}^{(x)}\right]\nonumber\\
    &=\frac{1}{2\beta^2}\left[(\alpha^2-1)\tau_\eta^2+\beta^2\tau_\epsilon^2+\sqrt{((\alpha^2-1)\tau_\eta^2+\beta^2\tau_\epsilon^2)^2+4\beta^2\tau_\eta^2\tau_\epsilon^2}\right]\nonumber
\end{align}

Hence, writing $\hat{x}_{t\mid t-1}:=\mathbb{E}[x_{t}\mid x_{<t}]$ we have
\begin{equation}\label{eqK}   \hat{x}_{t\mid t-1}=\beta\mathbb{E}[z_{t}\mid x_{<t}]=(1-K\beta )\alpha \hat{x}_{t-1\mid t-2}+ K \beta\alpha x_{t-1}
\end{equation}

Setting $G(x,h):=(1-K\beta)\alpha h+K\beta\alpha x$ then from (\ref{eqK}) we get $h_t=\hat{x}_{t\mid t-1}$ and clearly $\mathbb{E}[x_{t}\mid h_t]= \mathbb{E}\left[x_{t}\mid\mathbb{E}\left[x_{t}\mid x_{<t}\right]\right]=\mathbb{E}\left[x_{t}\mid x_{<t}\right]$, which in turn implies $y_t^*=x_t-\mathbb{E}[x_{t}\mid h_t]=x_t-\mathbb{E}\left[x_{t}\mid x_{<t}\right]$. So the maximizing $G$ may be taken linear.

The problem has been reduced to the point where the outer player need only find two constants defining a linear function. One may think this is a consequence of the relaxed formulation of (\ref{eqrelax}), but this is not the case.

\begin{lemma}\label{lemmakalman} (a) The optimal $y^*_t=x_t-\hat{x}_{t\mid {t-1}}$ in the solution of the relaxed problem (\ref{eqrelax}) for the linear Gaussian system is independent of $\sigma(x_{<t})$.

(b) The solutions to (\ref{eqrelax}) and (\ref{eqlatent}) for the linear Gaussian system coincide.    
\end{lemma}
\begin{proof}
Since $x_t$ and $h_t=\hat{x}_{t\mid {t-1}}$ are jointly Gaussian it follows that $y^*_t$ is Gaussian too. This implies that the relaxed condition $\mathbb{E}[y^*_t\mid h_t]=\mathbb{E}[y^*_t]$ is actually equivalent to independence, the original unrelaxed constraint in (\ref{eqlatent}). Moreover, the optimal $y^*_t=x_t-\hat{x}_{t\mid {t-1}}=x_t-\mathbb{E}[x_t\mid x_{<t}]$ is by construction orthogonal to all $x_i,x_{i+1},\dots,x_{t-1}$ with $i<t$ and because $(y^*_t,x_i,\dots,x_{t-1})$ is multivariate Gaussian, this orthogonality implies independence, that is, $y^*_t$ is independent of the observed past $\sigma(x_{<t})$, proving (a).

Since regardless of the choice of $G$ the resulting $h_t$ will always be a function of the observed past, then  $y^*_t=x_t-\hat{x}_{t\mid {t-1}}$ will always be independent of $h_t$ and the inner player can always choose it to guarantee a payoff of at most $\mathbb{E}[|\hat{x}_{t\mid {t-1}}|^2]$, the value $V_{rel}$ of (\ref{eqrelax}). The best the first player can do is then to choose any $G$ (eg, the linear $G(x,h):=(1-K\beta)\alpha h+K\beta\alpha x$ described earlier) that guarantees the second player does not do strictly better than $\mathbb{E}[|\hat{x}_{t\mid {t-1}}|^2]$, so this value is also the equilibrium solution $V$ for (\ref{eqlatent}), proving (b).
\end{proof}

Hence, in both the original and relaxed problems the space of admissible $G$ may be restricted to linear functions. This restriction greatly reduces the gauge freedom: a bijection of $\mathcal{H}$ that preserves the linearity of the recurrence $h_{t+1}=G(x_t,h_t)$ is itself a linear bijection, which for $\mathcal{H}=\mathbb{R}$ is just a nonzero constant $c$. If $G(x,h)=ah+bx$ is a solution of the outer problem then, writing $\tilde{h}_{t}=ch_t$, we have $\tilde{h}_{t+1}=a\tilde{h}_{t}+bcx_t$ and $\tilde{G}(x,h)=ah+cbx$ is another solution. Moreover, in the problem restricted to linear $G$ all solutions are of this form, as the following Lemma shows.

\begin{lemma}\label{lemmauniquekalman}
    Assume $\alpha\beta\tau_\epsilon\neq 0$ and consider problem (\ref{eqlatent}) with $G$ restricted to functions of the form $G(x,h)=ah+bx$ ($|a|<1$), then the solution set consists of all those $G$ such that $b\neq 0$ and $a=(1-K\beta)\alpha$.
\end{lemma}
\begin{proof}See Appendix \ref{app:proofs}.
\end{proof}

 Thus, if we know the parameters of the model ($\alpha,\beta,\tau_\eta,\tau_\epsilon$), we can instead fix $b=K\alpha$, since $b$ is free, and, in this way, recover $h_t=z_{t\mid t-1}=\mathbb{E}[z_{t}\mid x_{\leq t-1}]$ (instead of $h_t=\hat{x}_{t|t-1}$), the solution to the Kalman filtering problem.

Of course, for this filtering problem, whose solution has long been known to be a simple linear recursion with known parameters, applying the barycentric method is an inefficient overkill. However, one could ask if, in the case the parameters are unknown, our approach allows us to recover them. If we do not know the linear parameters but we know (or assume) the system has this linear Gaussian form, we first notice that the parameter tuple $(\alpha,\beta,\tau_\epsilon,\tau_\eta)$ overidentifies the observed process $x_t$, since the process with parameters $(\alpha,1,\beta\tau_\epsilon,\tau_\eta)$ has the same distribution for $x_t$\footnote{To see this define $\tilde{z}_t:=\beta z_t$ and multiply both sides of the second equation in (\ref{eqHMlin}) by $\beta$, obtaining the latent state equation $\tilde{z}_{t+1}=\alpha \tilde{z}_t+\beta \epsilon_t$.}. We may hence assume $\beta=1$, leaving three unknown parameters.

Our method of solving (\ref{eqlatent}) with $G$ restricted to linear functions will provide a system of three nonlinear equations for the three unknown parameters.
\begin{align}\label{threeeqs}
    a &= (1-K)\alpha \quad \text{ (guaranteed by Lemma \ref{lemmauniquekalman})}\nonumber \\
    V &= \mathbb{E}[|\hat{x}_{t\mid {t-1}}|^2]=\frac{\tau_\epsilon^2}{1-\alpha^2}-\frac{K\tau_\eta^2}{1-K}\nonumber \\
    \mathbb{E}[|y_t|^2]&=\frac{\tau_\eta^2}{1-K}
\end{align}
Here $V$ is the transport cost, which as we saw equals $\mathbb{E}[|\hat{x}_{t\mid {t-1}}|^2]$, the variance of the estimator. The formula for the variance of the estimator and of the Kalman innovations (i.e. the barycentric variables $y_t=x_t - \hat{x}_{t\mid {t-1}}$) can be found in any textbook exposition of the Kalman filter.

Finally, observe that the $y_t=x_{t}-\hat{x}_{t\mid {t-1}}$ are iid. This follows from the standard fact that the innovations of a Kalman filter are iid, but can also be seen as a consequence of Lemma \ref{lemmakalman}, that is, the $y_t$ are independent of the observed past $\sigma(x_{<t})$, but any $y_s$ with $s<t$ is a function of $x_s$ and $h_s$, which are $\sigma(x_{<t})$-measurable.

We mention this iid property because, as explained in the previous section, it is a necessary ingredient for our goal of simulating the predictive distribution by sampling the barycentric variable and applying the inverse of $T$ to it.

\subsection{The general case}

The Kalman system is simple enough that we were able to obtain a full analytic expression for the solution of our barycentric problem. For slightly more general $x_t$ this will no longer be the case. Nonetheless, the barycentric method is powerful enough to allow for a useful extraction of independent `idiosyncratic' noise from a process exhibiting long dependence on its own past.

For instance, we may want to apply our methods to stationary, square-integrable systems of the following general form, which subsumes all the examples (AR, GARCH and Kalman) seen so far

\begin{align}\label{eqHMnonlin}
x_{t+1}&=f (z_{t+1},\epsilon_{t+1}) \nonumber\\
z_{t+1}&=g (z_t, x_t, \eta_{t+1})
\end{align}
where $f$ and $g$ are (possibly nonlinear) differentiable functions, usually parametrized by $\beta$ and $\alpha$ (i.e. $f=f_\beta$ and $g=g_\alpha$).

The white noise processes $\epsilon_t$ and $\eta_t$ are iid, independent of each other and of the past, and they will usually be multiplicative or additive, but they are here included  as arguments within the same function as the immediate past of the process to encompass the most general case. Because the noise terms are iid it follows that the joint process $(x_t,z_t)$ is Markov. However, only $x_t$ is assumed to be observable and $x_t$ alone is generally not Markov. Moreover, there will generally be no closed-form solution for the barycenter like the one we found for the simple linear state model. Even for that system (\ref{eqHMlin}) the closed-form solution is lost once we drop the Gaussian assumption for the observation noise $\eta_t$.

Yet, our approach will be the same as in the Kalman case, that is, we will obtain from any square-integrable, stationary $x_t$ a barycentric process $y_t$ that is both iid and independent of $x_t$'s past. In this way, a solution to the adversarial game will return: (i) a decomposition of $x_t$'s variability in terms of its own past and this new idiosyncratic noise captured by $y_t$ and (ii) a recurrent function $G$ that computes the dependence of $x_t$ on its own past. The adversarial formulation will be the same as (\ref{eqlatent}) before\footnote{As in the Kalman case before, $h_0$ and $\{x_t\}_{t<0}$ are regarded as part of the `observed past' so that $h_t$ is stationary and, due to stationarity, the objective function does not depend on $t$.}
\begin{equation}\label{eqlatent2}\max_G\min_{y_t} \mathbb{E}[\|x_{t}-y_{t}\|^2]\quad \text{ s.t. } y_{t}{\perp\!\!\!\perp} h_{t}\quad \text{ with } h_{t}=G(x_{t-1},h_{t-1})\ \text{ stationary }
\end{equation}
where $x_t$ is a stationary, square-integrable process.

This time, however, we must prove the well-posedness of this problem. Among other things, we must determine what conditions a measurable $G$ must satisfy so that a stationary solution $h_t$ to the recurrence equation exists. Before handling that problem, we will focus on the well-posedness of the inner minimization problem, for which we may take $h_t$ as given, and which is just an instance of the optimal transport barycenter problem \cite{Pass}\cite{Tabak25}.

For this problem, the $y_t$ chosen by the inner player will in principle be a random process in an expanded probability space. To be more precise, the solution $\pi^*$ to the optimal transport barycentric problem must be an element of $\Pi((\mathcal{X}\times\mathcal{H})\times\mathcal{X})$, the set of probability measures on $(\mathcal{X}\times\mathcal{H})\times\mathcal{X}$, such that the marginal $\pi_{x_t,h_t}$ on $(\mathcal{X}\times\mathcal{H})$ coincides with the one induced by $P$, which we denote by $\rho$, and the marginal $\pi_{h_t,y_t}^*$ on $\mathcal{H}\times\mathcal{X}$ is a product measure (equivalently, $y$ and $h$ are independent). That is, given $h_t$, the inner problem is
$$\min_{\pi\in\Pi((\mathcal{X}\times\mathcal{H})\times\mathcal{X})}\mathbb{E}_\pi[|x_t-y_t|^2]\quad \text{s.t. }\ \pi_{x_t,h_t}=\rho,\ \pi_{h_t,y_t}=\pi_{h_t}\otimes \pi_{y_t}$$
where the existence of the minimizer is guaranteed by \cite[Theorem 1]{Yang22}.

The solution $\pi$ applies to a single $t$, but can be naturally extended to define a process $y=(y_t)_{t\in\mathbb{N}}$ in the same space as $(x_t,h_t)_{t\in\mathbb{N}}$, i.e., we define a joint probability $\pi_\infty$ on $\mathcal{X}^\mathbb{N}\times \mathcal{H}^\mathbb{N}\times\mathcal{X}^\mathbb{N}$ that makes the $y_t$'s relatively independent by taking the product of the conditional distributions $\pi_{y_t\mid x_t,h_t}(A):=\pi(y_t\in A\mid x_t,h_t)$, i.e. for cylinder sets $B=\prod_s B_s$:
$$\pi_\infty\left(y\in B\mid 
x,h\right)=\prod_{s\in\mathbb{N}} \pi_{y_s\mid x_s,h_s}(B_s)\quad\text{ and }\ \pi_\infty\left((x,h)\in C\right)=P\left((x,h)\in C\right)$$
 where $x=(x_t)_{t\in\mathbb{N}}$ and $h=(h_t)_{t\in\mathbb{N}}$. By construction, $\pi_\infty$ extends $P$ to the new variable and it can be easily verified that its $t$-marginals are $\pi$ for all $t$. Notice that in the case that $\pi$ is given by a Monge solution $\pi_{y_s\mid x_s,h_s}$ is just a delta measure, so that $\pi_\infty$ becomes the marginal of $P$ given by defining $y_t=T(x_t,h_t)$ within the original space $(\Omega,\mathcal{F},P)$.

We now turn to the well-posedness of the outer player's problem and the conditions under which $G$ admits a stationary solution to the recurrence equation.

\begin{definition}
    We will say a measurable $G:\mathcal{X}\times\mathcal{H}\to\mathcal{H}$ satisfies the set of conditions C if there exists a measurable $\Lambda:\mathcal{X}\to\mathbb{R}_{>0}$ such that:
    \begin{enumerate}
        \item C1. (Stochastic Lipschitz condition) For all $x,h,h'$, $$|G(x,h)-G(x,h')|\leq \Lambda(x)|h-h'|$$  
        \item C2. (Exponential forcing term contractivity) 
        $$a_n<\infty,\forall n\geq0,\text{ and }\limsup_{n\to\infty} a_n^{1/n}<1$$
        where $$a_n= E\left[ \left( \prod_{k=1}^n \Lambda(x_{k})^2 \right) |G(x_{0}, 0)|^2 \right]$$
        \item C3. (Contractive Lyapunov exponent)
        $$\limsup_{n\to\infty} E\left[ \prod_{k=1}^n \Lambda(x_{k})^2   \right]^{1/n}<1$$
        
    \end{enumerate}
\end{definition}

\begin{lemma}\label{lemmaG}
    Let $G:\mathcal{X}\times\mathcal{H}\to\mathcal{H}$ be a measurable function such that the condition set C is satisfied. Then there exists a unique stationary and square-integrable solution $(h_t)_{t\in\mathbb{N}}$ to the recurrence equation, $h_{t+1}=G(x_t,h_t)$
\end{lemma}
\begin{proof}
See Appendix \ref{app:proofs}.\end{proof}

We will write $\mathcal{G}$ for the set of measurable functions $G$ satisfying the set of conditions C. Observe that any $G$ that has finite $\mathbb{E}[G(x_t,0)^2]$ and is Lipschitz on the second argument with constant Lipschitz constant $\lambda<1$ is immediately in $\mathcal{G}$.

We have finished proving that (\ref{eqlatent2}) is well-defined\footnote{It remains to prove the outer supremum is attainable.}. We could now rewrite it as:
$$\max_{G\in \mathcal{G}}\ \min_{\substack{\pi\in\Pi((\mathcal{X}\times\mathcal{H})\times\mathcal{X})\\ \pi_{x_t,h_t}=\rho\\ \pi_{h_t,y_t}=\pi_{h_t}\otimes \pi_{y_t}}}\mathbb{E}_\pi[|x_t-y_t|^2]\quad \substack{\text{ where $\rho$ is the law of $(x_t,h_t)$}\\\text{ with $h_t$ the solution of the recurrence for $G$ }}$$
We will, however, stick to (\ref{eqlatent2}) for its simplicity.

As in the Kalman case, we want the recurrent $G$ to generate a covariate $h_t$ that captures all relevant information about the observed past up to $t-1$. The following result shows that $\mathcal{G}$ is rich enough to guarantee this with $\text{dim}\mathcal{H}=1$.

\begin{theorem}\label{teoopt}
    Let $\mathcal{H}=\mathbb R$. There exists a $G\in\mathcal{G}$ such that if $(h_t)_{t\in\mathbb{N}}$  is a stationary solution to the recurrence equation $h_{t+1}=G(x_t,h_t)$ then $\sigma(h_t)=\sigma(x_{<t})$. 
\end{theorem}
\begin{proof} See Appendix \ref{app:proofs}.
\end{proof}

\begin{observation}\label{obsegen} The function $G\in\mathcal{G}$ from Theorem \ref{teoopt} is optimal for the outer player.
\end{observation}
\begin{proof}
    By contradiction, if  $V:=\min_{y\text{ s.t. }y_t{\perp\!\!\perp }h_t} \mathbb{E}[\|x_t-y_t\|^2]$ we assume that there is another $\tilde{G}\in\mathcal{G}$ that the outer player can choose with a recurrent solution $\tilde{h}_t$ such that 
    $$\min_{y\text{ s.t. }y_t{\perp\!\!\!\perp }\tilde{h}_t} \mathbb{E}[\|x_t-y_t\|^2]> V$$
    Let $y^*$ be a minimizer for $(h_t)_{t\in \mathbb{N}}$, that is, $\mathbb{E}[\|x_t-y_t\|^2]=V$ and $y^*_t{\perp\!\!\!\perp }h_t$, where $(h_t)_{t\in\mathbb{N}}$ is the solution to the recurrence equation for $G$. By Theorem \ref{teoopt}, $\sigma(h_t)=\sigma(x_{<t})$, so $y_t^*$ is independent of $\sigma(x_{<t})$. Now $\tilde{h}_t$ is necessarily $\sigma(x_{<t})$-measurable because it satisfies the recurrence equation for $\tilde{G}$. It follows that $y^*_t{\perp\!\!\!\perp }\tilde{h}_t$. This implies that
    $$\min_{y\text{ s.t. }y_t{\perp\!\!\!\perp }\tilde{h}_t} \mathbb{E}[\|x_t-y_t\|^2]\leq V$$
    contradicting the strict inequality.
\end{proof}

It follows that our problem (\ref{eqlatent2}) can be restated as 
$$\min_{y_t} \mathbb{E}[\|x_{t}-y_{t}\|^2]\quad \text{ s.t. } y_{t}{\perp\!\!\!\perp} \mathcal{F}_{t-1}\quad \text{ with } \mathcal{F}_{t-1}=\sigma(x_s;\ s\leq t-1) $$

The advantage of our recurrent adversarial formulation is that in addition to the barycentric solution it yields $G$, which offers a way of encoding the relevant information from the past, just like the standard Kalman solution to linear Gaussian SSMs, or like recurrent neural networks (RNNs). Indeed, this feature will be crucial when translating these ideas to effective algorithms in the next chapter, where $G$ will be implemented as a GRU, a type of RNN. Of course, in any effective implementation $G$ will be restricted to a much smaller space of functions than $\mathcal{G}$. The proof of Theorem \ref{teoopt} relies on Cantor spaces and Borel isomorphisms, which don't pass stricter requirements of continuity and smoothness. Since standard optimization methods require differentiability to find minima, and any sequence of differentiable functions converging in $L^2$ to one of these measurable isomorphisms on Cantor spaces would necessarily have exploding gradients, these results are of little practical use. In practice, we will need a dimension higher than $1$ for $h_t$ so that it captures a relevant part of the past through sufficiently smooth functions. Yet it is good to know that in an abstract sense the past can always be recovered with a rich enough space of functions.

We now turn to the analysis of the barycentric variables $(y_t^*)_{t\in\mathbb{N}}$. As with the Markov and Kalman examples before, we want the $y_t$ to be an iid sequence and to be the deterministic image of $x_t$ and $h_t$ under some transport map $T$. With a few natural assumptions, we will now see that this comes off as a straightforward consequence of our formulation of the problem.

{\bf Assumption A.} The conditional distribution $P(x_t|x_{<t})$ is almost surely atomless in case $d_x=1$ and absolutely continuous with respect to Lebesgue measure if $d_x\geq 2$.

\begin{lemma}\label{lemmaiid}Under Assumption A the equilibrium solution $(y_t^*)_{t\in\mathbb{N}}$ is unique (in the sense that if $(G^*,(y_t^*)_{t\in\mathbb{N}})$ and $(\widetilde{G}^*,(\widetilde{y}_t^*)_{t\in\mathbb{N}})$ are two solutions of (\ref{eqlatent2}) then $y_t^*=\widetilde{y}^*_t$) and a Monge solution. Moreover, the $y_t$ are iid.
\end{lemma}

\begin{proof} Let $(G^*,(y_t^*)_{t\in\mathbb{N}})$ be a solution of (\ref{eqlatent2}) with value $V=\mathbb{E}[\|x_t-y_t^*\|^2]$. Let $(h_t)_{t\in\mathbb{N}}$ be a solution to the recurrence equation of $G^*$. By observation \ref{obsegen} and Theorem \ref{teoopt} the outer player could have chosen an optimal $\widehat{G}$ such that the stationary solution $\widehat{h}_t$ has the same information as the whole observed past, that is, $\sigma(\widehat{h}_t)=\sigma(x_{<t})$. Let $\widehat{y}_t$ be an optimal response of the inner player to $\widehat{h}_t$. Then $\widehat{y}_t{\perp\!\!\!\perp}\sigma(x_{<t})$. Since $h_t$ is $\sigma(x_{<t})$-measurable it follows that $\widehat{y}_t{\perp\!\!\!\perp} h_t$. Then $(G^*,(\widehat{y}_t)_{t\in\mathbb{N}})$ must be a solution, since $\widehat{y}_t$ is independent of $h_t$ and $V=\mathbb{E}[\|x_t-\widehat{y}_t\|^2]$ (the payoff for the outer player under $\widehat{G}$ must be the same as under $G^*$, since both are optimal).

Since $\sigma(\widehat{h}_t)=\sigma(x_{<t})$, $P(x_t\mid x_{<t})=P(x_t\mid \widehat{h_t})$ and Assumption A implies the condition for a unique, invertible Monge solution from Theorem \ref{teogeneralotb} for $\widehat{h_t}$. By Theorem \ref{teogeneralotb} then there exists a measurable map $T$ such that $\widehat{y}_t=T(x_t,\widehat{h_t})$. We must next show that $\widehat{y}_t=y_t^*$. Since $\widehat{y}_t$ is a Monge solution it is defined in $(\Omega,\mathcal{F},P)$ and hence in the extension where $y_t^*$ is defined. We already established that $(G^*,(\widehat{y}_t)_{t\in\mathbb{N}})$ is a solution, so $(\widehat{y}_t)_{t\in\mathbb{N}}$, like $(y_t^*)_{t\in\mathbb{N}}$, solves the OTBP for $h_t$.

It remains to verify that $P(x_t\in\cdot\mid h_t)$ satisfies the conditions of Theorem \ref{teogeneralotb} for the existence of a unique, invertible Monge solution. Since
$\sigma(h_t)\subseteq\sigma(x_{<t})$, this conditional distribution is a mixture of the distributions $P(x_t\in\cdot\mid x_{<t})$. Now, mixtures of atomless measures are atomless, and mixtures of absolutely continuous measures are absolutely continuous. Hence, the conditions of Theorem \ref{teogeneralotb} are satisfied, giving
$\widehat y_t=y_t^*$.

Finally, to prove the iid property, notice that since $y_s^*=\widehat{y_s}$ is a Monge solution, we have that $y_s^*=T(x_s,h_s)$ is $\sigma(x_{<t})$-measurable for $s<t$. Hence, since $\widehat{y}_t$ is independent of $\sigma(x_{<t})$ it follows that it must be independent of $\{y_s:s<t\}$.

The identical distribution follows from stationarity of $(x_t,h_t)$ and from the fact that the Monge map is time-homogeneous.
\end{proof}

\section{Closed-form results for the one-dimensional OTBP}

We will now state some general results for the solution to the OTBP in one dimension, which will be useful to compute the barycenter of some well-behaved processes in closed form. The proofs are deferred to the Appendix.

\begin{theorem}\label{teobarycenter1d}Let $\mathcal{A}$ be a sub-$\sigma$-algebra of the probability space $(\Omega,\mathcal{F},P)$ in which an $\mathbb{R}$-valued square-integrable random variable $x$ is defined, such that the regular conditional probabilities $\nu_\omega(B)=P(x\in B\mid\mathcal{A})(\omega)$ are atomless. Then the problem
$$\min_{y} \mathbb{E}[\|x-y\|^2]\quad \text{ s.t. } y{\perp\!\!\!\perp} \mathcal{A}$$
has unique solution
$$y = F_Y^{-1}( F_{x\mid\mathcal{A}}(x) )$$
where $F_{x\mid\mathcal{A}}(x')=\nu_\omega((-\infty, x'])=P(x\leq x'\mid\mathcal{A})(\omega)$ is the conditional cdf of $x$
 with respect to $\mathcal{A}$, and
 $$F_Y^{-1}(u) = \mathbb{E} \left[ F_{x|\mathcal{A}}^{-1}(u) \right]$$
 is a quantile function (i.e. a left-continuous, non-decreasing function on $[0,1]$), the inverse cdf of $y$.
 \end{theorem}
\begin{proof}
See Appendix \ref{app:proofs}
\end{proof}

\begin{corollary}\label{1dcor}
Let $z$ be a random vector, and $\epsilon$ be a square-integrable, atomless random variable independent of $z$. Let  $x=F(z,\epsilon)$, where $F$ is a real-valued measurable function and, for a.e. $z$, the map $e\mapsto F(z,e)$ is continuous and strictly increasing. Assume $x\in L^2$. Then the barycentric solution to
$$\min_y \mathbb E[|x-y|^2] \quad\text{ s.t. }\quad y\perp\!\!\!\perp z$$
has quantile function
$$F_Y^{-1}(u)=\mathbb E_z\left[F\left(z,F_\epsilon^{-1}(u)\right)\right]
\qquad u\in(0,1)$$
Equivalently, defining $\bar{F}(e):=\mathbb E_z[F(z,e)]$,
the barycentric variable is
$$y=\bar{F}(\epsilon)=\mathbb E_z[F(z,\epsilon)]$$
\end{corollary}

\begin{proof} See Appendix \ref{app:proofs}
\end{proof}

\begin{corollary}\label{1dcor2}
Let $z$ be a random vector, and $\epsilon$ be a square-integrable, atomless random variable independent of $z$. Let $x=g(z)+f(z)\epsilon$, where $g,f$ are measurable real-valued functions, $x\in L^2$, and $f(z)\neq 0$ a.s. Then the OTBP
$$\min_y \mathbb E[|x-y|^2]\quad\text{ s.t. }\quad y\perp\!\!\!\perp z$$
has a) a Monge solution
$$y=\mathbb E[g(z)]+\frac{\mathbb E[|f(z)|]}{|f(z)|}\left(x-g(z)\right)=\mathbb E[g(z)]+\mathbb E[|f(z)|]\operatorname{sgn}(f(z))\epsilon$$ 
if $\epsilon$ is symmetric,

and b) a Monge solution
$$y=\mathbb E[g(z)]+\frac{\mathbb E[f(z)]}{f(z)}\left(x-g(z)\right)
=\mathbb E[g(z)]+\mathbb E[f(z)]\epsilon$$
if $f(z)>0$
\end{corollary}

\begin{proof}
See Appendix \ref{app:proofs}
\end{proof}

These results are useful for computing the exact formula of the barycenter of random variables $X$ with covariate $Z$ when the underlying law is known. When coupled with a recurrent function $G$ that generates a variable $h_t$ encoding the past, this naturally extends to the computation of the Monge barycenter of stochastic processes $x_t$ in terms of $h_t$. However, just like the Rosenblatt transform, it requires full knowledge or perfect estimation of the conditional probabilities, thus providing no practical extraction procedure for the more general case where the data-generating process is unknown.

\section{A nonlinear example revisited}

As a simple, nonlinear example of (\ref{eqHMnonlin}) we will now look again at the GARCH$(1,1)$ model from Section \ref{wellknown}, which was defined by 
$$x_t=\sigma_t \epsilon_t,\quad\text{ with $\epsilon_t$ an iid white noise and }\quad \sigma_{t+1}^2=\omega+\alpha x_t^2+\beta \sigma_t^2 \quad \text{ ($\omega,\alpha,\beta>0$)}$$

In this model's equation the latent variable $\sigma_t$ depends both on its own past and on the past of the observed variable, while its evolution is deterministic. Most importantly, $x_t$ has multiplicative, not additive noise. We will apply the same philosophy as before and the barycentric method will allow us to recover parametric information from this GARCH$(1,1)$ process.

Recall the formulation of the problem as a maxmin game  (\ref{eqlatent2})
\begin{equation}\label{eqlatentgarch}\max_G\min_{y} \mathbb{E}[\|x_{t}-y_{t}\|^2]\quad \text{ s.t. } y_t{\perp \!\!\! \perp}h_t\quad \text{ with } h_{t}=G(x_{t-1},h_{t-1})
\end{equation}
Unlike the Kalman scenario we cannot simplify the analysis by resorting first to the relaxed barycentric problem (\ref{eqrelax}) because the constraint $\mathbb{E}[y_t\mid h_t]=\mathbb{E}[y_t]$ trivializes the problem: by construction, the noise is multiplicative, so $x_t$ already satisfies $\mathbb{E}[x_t\mid z_t]=\mathbb{E}[x_t]$, and the solution for the inner minimization is then $x_t=y_t$: the outer player cannot prevent $y$ from capturing all the variability in $x$.

We must stick then to (\ref{eqlatentgarch}), which is generally harder, but for this model, as in the Kalman case, we can find an explicit solution analytically.

\begin{lemma}\label{lemmagarch1}Let $G^*(x,z)=\omega+\alpha x^2+\beta z$ and $y_t^*:=\frac{x_t}{\sigma_t}\mathbb{E}[\sigma_t]$. Then $(G^*,y^*_t)$ is a solution of (\ref{eqlatentgarch}) and its value is $V:=\mathbb{E}[(\sigma-\mathbb{E}[\sigma])^2]$
\end{lemma}
\begin{proof}
By construction any $h_t$ is $\sigma(x_{<t})$-measurable. If the inner player chooses $y_t=y_t^*=\frac{x_t}{\sigma_t}\mathbb{E}[\sigma_t]$ then $y_t$ is independent of $h_t$ regardless of the choice of $G$, and for $y_t^*$ the squared distance is $\mathbb{E}[|\sigma-\mathbb{E}[\sigma]|^2]$. It follows that the outer player cannot do better than $V$. On the other hand, the outer player can choose $G^*$, so that $h_t=\sigma_t^2$. Then whichever $y_t$ is the optimal response of the second player, it must be independent of $\sigma_t$. In particular it must satisfy $\mathbb{E}[y_t^2 (\sigma_t-\mathbb{E}[\sigma_t])]=0$ since $y_t^2$ and $(\sigma_t-\mathbb{E}[\sigma_t])$ are functions of $y_t$ and $\sigma_t^2$, respectively. Lemma \ref{lemmagarch} below shows $y_t^*$ minimizes the squared distance to $x_t$ subject to this restriction, so it must also minimize it subject to the stronger condition of independence.

It follows that $y_t^*$ is an optimal response to $G^*$ and the value of this game is $V=\mathbb{E}[(\sigma_t-\mathbb{E}[\sigma_t])^2]$. 
\end{proof}

Notice that $V$ equals  $\mathbb{E}[\sigma_t^2]-\mathbb{E}[\sigma_t]^2$. Since $\mathbb{E}[x_t^2]=\mathbb{E}[\sigma_t^2]$ and $\mathbb{E}[y_t^{*2}]=\mathbb{E}[\sigma_t]^2$ it follows that $V=\mathbb{E}[x_t^2]-\mathbb{E}[y_t^{*2}]$, i.e. the value of this game is the difference between the variance of $x_t$ and that of the barycenter $y_t^*$ (see \cite[Theorem 3]{Yang22}).

\begin{lemma}\label{lemmagarch}
    If $y_t^*$ is the solution to the problem
    $$\min_{y_t}\mathbb{E}[\|x_t-y_t\|^2]\quad\text{ s.t. }\mathbb{E}[y_t^2 (\sigma_t-\mathbb{E}[\sigma_t])]=0$$
    with $x_t=\sigma_t\epsilon_t$ a stationary GARCH$(1,1)$ process then $y_t^*=\frac{x_t}{\sigma_t}\mathbb{E}[\sigma_t]$
\end{lemma}
\begin{proof}

Let $m=\mathbb{E}[\sigma_t]$. For any feasible $y_t$ we have
$$\mathbb{E}[(x_t-y_t)^2]=\mathbb{E}[(x_t-y_t)^2]+\frac{1}{m}\mathbb{E}[(\sigma_t-m)y_t^2]$$

Now,
\begin{align}
    (x_t-y_t)^2+\frac{1}{m}(\sigma_t-m)y^2&=(\epsilon_t\sigma_t-y_t)^2+\frac{1}{m}(\sigma_t-m)y^2\nonumber\\
    &=\frac{\sigma_t}{m}(y_t-m\epsilon_t)^2+\sigma_t^2\epsilon_t^2-m\sigma_t\epsilon_t^2\nonumber
\end{align}
so taking expectations and using $\mathbb{E}[\epsilon_t^2]=1$ and independence of $\epsilon_t$ and $\sigma_t$
$$\mathbb{E}[(x_t-y_t)^2]=\mathbb{E}\left[\frac{\sigma_t}{m}(y_t-m\epsilon_t)^2\right]+\mathbb{E}[\sigma_t^2]-m^2\geq \mathbb{E}[\sigma_t^2]-m^2$$
with equality if $y_t=y_t^*=m\epsilon_t=\frac{x_t}{\sigma_t}\mathbb{E}[\sigma_t]$



\end{proof}

As with the Kalman filter case, in the scenario where we know or assume the system is GARCH$(1,1)$ but do not know the parameters, we could try to recover parametric information through the barycentric method by restricting the function space for $G$. In this case, since $G(x,h)=\omega+\alpha x^2+\beta h$ yields an optimal solution for the maximizing player we would naturally consider $G$ of the form $G(x,h)=A+Bx^2+Ch$ ($A,B,C>0$, $B+C<1$). Let $P$ be the space of such quadratic polynomials. Notice too that, as in the Kalman case with linear $G$, there is a gauge invariance even after restricting to $P$. Indeed, if $G(x,h)=\omega+\alpha x^2+\beta h$ and $h_{t+1}=\omega+\alpha x_t^2+\beta h_t$ then for any $\lambda,\gamma>0$ we may define $\tilde{h}_t=\lambda h_t+\gamma$ and $\tilde{h}_{t+1}=(\lambda\omega+\gamma(1-\beta))+\lambda\alpha x_t^2+\beta \tilde{h}_t=:\tilde{G}(x_t,\tilde{h}_t)$ yields an alternative solution $\tilde{G}$ whenever $\lambda\alpha+\beta<1$. $\beta$, however, remains unchanged in $\tilde{G}$, and the following lemma shows this characterizes the set of solutions for the restricted problem.

\begin{lemma}\label{lemmauniquegarch}
    Consider problem (\ref{eqlatentgarch}) with $G$ restricted to $P$. Assume additionally that the noise distribution's cdf $F_\epsilon$ is absolutely continuous and strictly increasing on a neighbourhood of 0. Then the solution set consists of all those $G$ such that $C=\beta$.
\end{lemma}
\begin{proof} See Appendix \ref{app:proofs}.
\end{proof}

In dealing with the barycenter solution for a GARCH input $x_t$ we could have spared ourselves the proofs of Lemmas \ref{lemmagarch1} and \ref{lemmagarch} and make a direct appeal to Corollary \ref{1dcor2}. Yet Lemma \ref{lemmagarch} shows an interesting property. The minimization of the squared distance to $x_t$ subject to a correlation constraint (i.e. a constraint much weaker than full independence) makes $y_t$ independent nonetheless. In other words, if there were an outer maximizing player trying to choose functions of $h_t$ and $y$ for a correlation restriction on the minimization problem:
$$\max_{G,\varphi,\psi}\min_{y_t}\mathbb{E}[(x_t-y_t)^2]\quad\text{ s.t. }\mathbb{E}[\psi(h_t)\varphi(y)]=\mathbb{E}[\psi(h_t)]\mathbb{E}[\varphi(y)]$$
then it would suffice to choose $\varphi(y)=y^2$ and $\psi(h_t)=\sigma_t$ to force the inner player to adopt an independent solution, i.e. the solution to 
 $$\min_{y_t}\max_{G,\varphi,\psi}\mathbb{E}[(x_t-y_t)^2]\quad\text{ s.t. }\mathbb{E}[\psi(h_t)\varphi(y)]=\mathbb{E}[\psi(h_t)]\mathbb{E}[\varphi(y)]$$

This situation is unique, in that the problem is not concave-convex and the inequality $\sup\inf\leq\inf\sup$ will usually be strict. Yet it highlights an important problem that will arise in the second part of this work, when we try to develop an algorithmic implementation of this adversarial barycentric game.

\bigskip

\chapter{Algorithmic implementation}

We now turn to the problem of developing effective algorithmic solutions to our adversarial problem.

\section{Optimization schedule and duality}

To incorporate the independence constraint in (\ref{eqlatent2}) into an objective function we could do
$$\max_G\min_{y_t}\sup_{\psi,\varphi} \mathbb{E}\left[\|x_{t}-y_{t}\|^2\right]+\mathbb{E}\left[(\psi(h_t)-\bar{\psi})\varphi(y_t)\right]\quad  \text{ with } h_{t}=G(x_{t-1},h_{t-1}),$$
which is equivalent to imposing independence on the minimizing player because once the correlation term is nonzero the test functions $\psi$ or $\varphi$ can be scaled to make the supremum $\infty$. Now, Observation \ref{obsegen} implies this is equivalent to 
\begin{equation}\label{eqlagr}\min_{y_t}\max_G\sup_{\psi,\varphi} \mathbb{E}\left[\|x_{t}-y_{t}\|^2\right]+\mathbb{E}\left[(\psi(h_t)-\bar{\psi})\varphi(y_t)\right]\quad  \text{ with } h_{t}=G(x_{t-1},h_{t-1})
\end{equation}
since the optimal $G$ will be the one that captures the whole past, irrespective of what $y_t$ is chosen.

We could then naively try to implement this $\text{minmax}$ (or rather $\text{infsup}$) optimization. The problem is that the inner maximization is necessarily unbounded, so algorithms trying to implement it will often fail to control the instability caused by the tendency to blowup whenever the constraint is non-zero.

If the functional $\mathcal{L}(G,\psi,\phi,y_t)=\mathbb{E}\left[\|x_{t}-y_{t}\|^2\right]+\mathbb{E}\left[(\psi(h_t)-\bar{\psi})\varphi(y_t)\right]$ were concave-convex, we could appeal to some duality Max-Min theorem and compute the $\text{supinf}$. The objective, however, is neither convex on $y_t$ nor concave on $G,\varphi,\psi$ and, in general, duality fails, ie. the inequality $\sup\inf\le\inf\sup$ is strict.

There are circumstances, though, where the supinf does equal the infsup of the barycenter problem. Generally speaking, this will happen when the barycenter $Y^*$ and the covariate $Z$ satisfy $$X-Y^*=f(Z)g(Y^*)$$ for some admissible functions $f$ and $g$.

The following theorem makes this precise. Before stating it, define a pair of test functions $(\psi,\varphi)$ as admissible if $\psi(Z)$ is in $L^1_0$, $\varphi$ is differentiable and $\psi(Z)\varphi(Y)$ is in $L^1$ whenever a square-integrable $Y$ is independent of $Z$ (this is guaranteed, for instance, if $|\varphi(y)|<C+Dy^2$).

\begin{theorem}[Criterion for duality]\label{teorankone}
Let $(X,Z)$ be a pair of random variables with $X\in L^2(\mathbb R)$, and assume that the conditional laws $P(X\in\cdot\mid Z=z)$ are atomless. Let
$$Q_z(u):=F^{-1}_{X\mid Z=z}(u),\qquad\bar Q(u):=\mathbb E_Z[Q_Z(u)]$$
and define the centered conditional quantile displacement
$$D(z,u):=Q_z(u)-\bar Q(u)$$
Let the barycentric Monge solution be $Y^*=\bar Q(U)$ with $U:=F_{X\mid Z}(X)\sim\operatorname{Unif}(0,1)$ independent of $Z$
and
$$V:=\mathbb E[|X-Y^*|^2]=\inf_{Y:\,Y\perp\!\!\!\perp Z}\mathbb E[|X-Y|^2]$$

Let $\mathcal{A}$ be a class of admissible pairs of test functions. Assume that $D$ has a representation
$$D(z,u)=\psi(z)r(\bar Q(u))$$
for some $(\psi,\varphi)\in\mathcal{A}$, where
$\varphi'(y)=2r(y)$.
Assume moreover that, for a.e. $z$, the map
$T_z(y):=y+\psi(z)r(y)$
is nondecreasing.
Then the pointwise Lagrangian
$$\ell_{\psi,\varphi}(z,u;y):=|Q_z(u)-y|^2+\psi(z)\varphi(y)$$
is well-defined from below as an extended-real functional, and
$$\inf_Y\mathbb E[\ell_{\psi,\varphi}(Z,U;Y)]=V$$
Consequently,
$$\sup_{(\psi,\varphi)\in\mathcal A}\inf_Y
\mathbb{E}\left[|X-Y|^2+\psi(Z)\varphi(Y)\right]=V$$

Conversely, suppose that
$$\sup_{(\psi,\varphi)\in\mathcal A}\inf_Y
\mathbb E\left[|X-Y|^2+\psi(Z)\varphi(Y)\right]=V$$
and suppose that the supinf is realized by an admissible pair $(\psi,\varphi)$. Then
$$D(z,u)=\frac{1}{2}\psi(z)\varphi'(\bar Q(u))$$
for a.e. $(z,u)$.
\end{theorem}

    {\bf Note:} At the population level, the independence condition can be enforced using bounded measurable functions $\psi$ and $\varphi$, which in fact characterize independence and avoid integrability issues. In the theorem above, however, the class $\mathcal A$ may contain unbounded functions. Yet it is not used to define the independence constraint.  It is used as a class of test functions because in simple examples where duality holds the supporting test function may not be bounded. For instance, in GARCH$(1,1)$ processes the natural test function is $\varphi(y)=y^2$. In finite samples the objective is always well-defined. At the population level, unbounded test functions should be interpreted through their pointwise Lagrangians, as in the above theorem, where assumptions guarantee that this pointwise Lagrangian is well-defined as an extended-real functional.

\begin{proof}
    See Appendix \ref{app:proofs}
\end{proof}

The following corollary applies the theorem to our time-dependent setting, where, as before, $(x_t)_{t\in\mathbb Z}$ is a square-integrable stationary process with atomless conditional laws $P(x_t\in\cdot\mid \mathcal{F}_{t-1})$ ($\mathcal{F}_t$ the filtration generated by $x_t$).

In this setting we will say a class $\mathcal A$ of triples $(G,\psi,\varphi)$ is {\em admissible} if $\psi(h_t^G)$ is centered for $h_t^G=G(x_{t-1},h_{t-1}^G)$ the stationary solution to the recurrence, and the pair $(\psi(h_t^G),\varphi)$ is admissible in the sense of Theorem \ref{teorankone} for $h_t^G$, i.e. $\psi(h_t^G)\in L^1_0$, $\varphi$ is differentiable, and $\psi(h_t^G)\varphi(y_t)$ is in $L^1$ whenever $y_t\in L^2$ and $y_t\perp\!\!\!\perp \mathcal F_{t-1}$.

\begin{corollary}[Recurrent criterion for duality]\label{correcurrentduality}
Let $x_t$ be as above and write $\omega$ for a generic element of the underlying probability space. Let $y_t^*=\bar{Q}(U_t)$ be the barycentric Monge solution to (\ref{eqlatent2}) given by Theorem \ref{teobarycenter1d}, where
$U_t:=F_{x_t\mid \mathcal{F}_{t-1}(\omega)}(x_t)$ is $\operatorname{Unif}(0,1)$ and $U_t\perp\!\!\!\perp \mathcal{F}_{t-1}$.
Let
$$Q_\omega(u):=F^{-1}_{x_t\mid \mathcal{F}_{t-1}(\omega)}(u)$$
and $\bar Q(u):=\mathbb E[Q_\omega(u)]$.

Let $\mathcal A$ be an admissible class of triples $(G,\psi,\varphi)$.

Define$D(\omega,u):=Q_\omega(u)-\bar Q(u)$ and set
$$V=V_t:=\mathbb E[|x_t-y_t^*|^2]=\inf_{y_t:\,y_t\perp\!\!\!\perp\mathcal F_{t-1}}\mathbb E[|x_t-y_t|^2]$$

Assume that
$$D(\omega,u)=\psi(h_t^G(\omega))r(\bar Q(u))$$
for some $(G,\psi,\varphi)\in\mathcal A$, where $\varphi'(y)=2r(y)$.

Assume also that, for a.e. $\omega$, the map
$$T_\omega(y):=y+\psi(h_t^G(\omega))r(y)$$
is nondecreasing. Then
$$\sup_{(G,\psi,\varphi)\in\mathcal A}\inf_{y_t}
\mathbb E\left[|x_t-y_t|^2+\psi(h_t^G)\varphi(y_t)\right]=V$$

Conversely, suppose that
$$\sup_{(G,\psi,\varphi)\in\mathcal A}\inf_{y_t}
\mathbb E\left[|x_t-y_t|^2+\psi(h_t^G)\varphi(y_t)\right]=V$$
and suppose that the supinf is realized by a triple $(G,\psi,\varphi)\in\mathcal{A}$. Then
$$D(\omega,u)=\frac12\psi(h_t^G(\omega))\varphi'(\bar Q(u))$$
for a.e. $(\omega,u)$.
\end{corollary}
\begin{proof} See the proof of Corollary \ref{correcurrentdualityk}, since this is the special case $k=1$.
\end{proof}

The first observation we can make is that the examples of stochastic processes that we have considered so far satisfy the rank one condition. Indeed for an AR$(q)$ process
$$x_t=\sum_{j=1}^q a_jx_{t-j}+\epsilon_t,
\qquad \epsilon_t\perp\!\!\!\perp \mathcal F_{t-1}$$
write $m_t:=\sum_{j=1}^q a_jx_{t-j}$. Then
$$Q_\omega(u)=m_t(\omega)+Q_\epsilon(u),\qquad \bar Q(u)=\mathbb E[m_t]+Q_\epsilon(u)$$
Hence
$$D(\omega,u)=Q_\omega(u)-\bar Q(u)=m_t(\omega)-\mathbb E[m_t]$$
Thus the rank one representation holds with $h_t=(x_{t-1},\dots,x_{t-q})$, $\psi(h_t)=m_t-\mathbb E[m_t]$, $r(y)=1$ and $\varphi(y)=2y$. The associated map is $T_\omega(y)=y+\psi(h_t(\omega))$, which is increasing.

For a GARCH$(1,1)$ process
$$x_t=\sigma_t\epsilon_t,\qquad
\sigma_t^2=\omega_0+\alpha x_{t-1}^2+\beta\sigma_{t-1}^2,
\qquad \epsilon_t\perp\!\!\!\perp\mathcal F_{t-1},$$
we have
$$Q_\omega(u)=\sigma_t(\omega)Q_\epsilon(u),\qquad\bar Q(u)=\bar\sigma Q_\epsilon(u)$$
Where $\bar\sigma:=\mathbb E[\sigma_t]$.

Therefore
$$D(\omega,u)=(\sigma_t(\omega)-\bar\sigma)Q_\epsilon(u)=\left(\frac{\sigma_t(\omega)}{\bar\sigma}-1\right)\bar Q(u)$$
and the rank one representation holds with
$\psi(h_t)=\frac{\sigma_t}{\bar\sigma}-1$, $r(y)=y$, $\varphi(y)=y^2$

The associated map is
$$T_\omega(y)=y+\left(\frac{\sigma_t(\omega)}{\bar\sigma}-1\right)y=\frac{\sigma_t(\omega)}{\bar\sigma}y$$
This map is increasing because $\sigma_t>0$.

Finally, for the linear Gaussian SSM case, the Kalman filter equations imply that the stationary conditional law of $x_t$ given the past is Gaussian with conditional mean and constant conditional variance given by:
$$x_t\mid\mathcal F_{t-1}\sim \mathcal{N}(m_t, s^2),\qquad s^2=\beta^2\hat{S}_{t\mid t-1}+\tau_\eta^2$$
where $m_t=\hat{x}_{t\mid t-1}$ is $\mathcal F_{t-1}$-measurable and $s^2$ is deterministic. Hence
$$Q_\omega(u)=m_t(\omega)+s\Phi^{-1}(u),\qquad\bar Q(u)=\mathbb E[m_t]+s\Phi^{-1}(u)$$
Therefore
$$D(\omega,u)=m_t(\omega)-\mathbb E[m_t]$$
So the Kalman example has the same rank one structure as the AR$(q)$ case, that is, $\psi(h_t)=m_t-\mathbb E[m_t]$, $r(y)=1$, $\varphi(y)=2y$
with $T_\omega(y)=y+\psi(h_t(\omega))$ increasing.

It follows that the supinf must equal the infsup of (\ref{eqlagr}), so that an algorithm that solves for the supinf should converge to the Monge barycenter. Since for these processes we know the formula of the Monge barycenter, we will use them as a test case for our supinf (or maxmin) solver, which we introduce in the next section.

\section{A recurrent maxmin solver}

Our first attempt was therefore to implement a numerical solver for the supinf (or maxmin) problem
\begin{equation}\label{eq33}\sup_{G,\psi,\varphi}\inf_{y_t}\mathbb E(x_t-y_t)^2+\mathbb E[(\psi(h_t^{(G)})-\bar \psi)\varphi(y_t)]
\end{equation}
via the following algorithm

\begin{algorithm}
\caption{Recurrent Barycentric Maxmin Solver}
\label{algo1}
\begin{algorithmic}
\Require $(x_t)_{1\leq t\leq T}$, $N$
\State Initialize $h_0$, $\lambda$, $(y_t)_{1\leq t\leq T}$, $\varphi_\theta$ and $G_\theta$
\State $n \gets 0$
\For{$n=1\dots N$}
    \State Compute $(z_t)_{1\leq t\leq T}$ using $G_\theta$, $h_0$ and $(x_t)_{1\leq t\leq T}$
    \State $z_t\gets \frac{z_t-\bar{z}}{\operatorname{sd}(z)} $
    \State $L\gets\sum_t \|x_t-y_t\|^2+\lambda\sum_t z_t\varphi_\theta(y_t)$
    \State Update $(y_t)_{1\leq t\leq T}$ to numerical near optimality minimizing $L$
    \State Update $\theta$ and $\lambda$ doing a gradient ascent step wrt $L$
\EndFor

\Return $y_t$
\end{algorithmic}
\end{algorithm}
\noindent
Here $z_t$ is a linear readout of the recurrent function $G_\theta$, and $G_\theta$ and $\varphi_\theta$ were implemented as a Gated Recurrent Unit (GRU) and a MLP with two hidden layers, respectively. The $y_t$ update used LBFGS and $\theta$ was updated using Adam.

A few remarks are in order.

First, as we anticipated, the algorithm implements a solution to the supinf problem, i.e., at each iteration, while a single gradient step is done for the maximizing parameters $\theta,\lambda$, the minimizing variable $y_t$ is updated (approximately) to the optimal response $y_t^*(\theta,\lambda)$. Hence, right before each gradient ascent step, $y_t$ is at the optimal response $y^*(\theta,\lambda)$ and, because of the envelope theorem, the gradient $\nabla_{\theta,\lambda}L$ at $(\theta,\lambda,y_t^*(\theta,\lambda))$ coincides with the gradient $\nabla_{\theta,\lambda}L^*(\theta,\lambda)$ ($L^*(\theta,\lambda)=L(\theta,\lambda,y^*(\theta,\lambda))$) that incorporates the minimization's best response $y^*(\theta,\lambda)$, i.e., there is no need to compute $\nabla_{\theta,\lambda}y^*(\theta,\lambda)$.

Secondly, $\psi(h_t^{(G)})$ in (\ref{eq33}) has been replaced with $\lambda z_t$ in the algorithm. The reason for this is that it reduces non-identifiability. Indeed, as mentioned above, $G_\theta$ is implemented\footnote{The variable $\theta$ denotes all the parameters involved in defining this recurrence.} as a GRU (Gated Recurrent Unit). Given an input $u_t=x_{t-1}$ and a hidden recurrent state $h_{t-1}\in\mathbb R^d$ ($d$ is called the ``hidden dimension''), a GRU defines $h_t\in\mathbb R^d$ by
$$R_t=\sigma(W_Ru_t+U_Rh_{t-1}+b_R)$$
$$S_t=\sigma(W_Su_t+U_Sh_{t-1}+b_S),$$
$$\widetilde h_t=\tanh(W_Hu_t+b_H+R_t\odot(U_Hh_{t-1}+c_H)),$$
and
$$h_t=(1-S_t)\odot \widetilde h_t+S_t\odot h_{t-1},$$
where $\sigma(a)=(1+e^{-a})^{-1}$ is applied componentwise and $\odot$ denotes componentwise multiplication. 

The correlation constraint then involves not $h_t$ itself but a linear readout $z_t=Ah_t+a$. Thus the algorithm searches over recurrent statistics of the past through
$$z_t=A\,G_\theta(x_{t-1},h_{t-1})+a$$
The $z_t$ is then normalized and the multiplier $\lambda$ is there to compensate for the normalization of $z_t$. This separation between the recurrent state $h_t$ and the tested statistic $z_t$ is useful because $h_t$ is not identifiable. Any invertible change of coordinates in the recurrent state can be compensated by changing the recurrence parameters, while the constraint only depends on the scalar quantity whose correlation with a function of $y_t$ is being tested. Thus, the readout $z_t$, after being centered and normalized over the sample, becomes the object that plays the role of $\psi(h_t)$ in the theoretical formulation.

Thirdly, all our results so far have involved expectations, whether of squared error $\mathbb{E}[(x_t-y_t)^2]$ or correlations $\mathbb{E}[\psi(h_t)\varphi(y_t)]$. For an algorithmic implementation we had to replace expectation with averages of finite samples. The convergence of these averages to the corresponding expectation requires an extra assumption that was not needed to establish the results of the previous sections: ergodicity. For the specific examples considered so far (Kalman, GARCH and AR), the conditions for stationarity also imply ergodicity, but in general this is definitely not true.

It should also be noted that while pointwise averages are guaranteed to converge under ergodicity, a true justification of our finite-sample procedure, both for our work and for the general OTBP problem, requires conditions on $x_t$ and the functional spaces that guarantee some form of uniform ergodicity to ensure consistency, i.e. that finite-sample solutions approximate population-level solutions. This is generally absent in the literature, and exceeds the scope of our work too. The methodology would ultimately be tested on empirical grounds.

Our first experiment was to run and test this algorithm with processes that, as we have shown, satisfy the conditions of Corollary \ref{correcurrentduality}: AR, Kalman and GARCH.

Since for these processes we know the closed form formula for the barycenter, the first metric of our test was the squared error between the learned $\widehat y_t$'s at the end of the algorithm and the oracle solution, i.e. the finite-sample forms $y_t^{*,f}$ of the known barycenters $y_t^*$. These are
\begin{align}y_{t}^{*,f}:&=
x_t-\sum_{j=1}^q a_jx_{t-j}+\frac1T\sum_{s=1}^T\sum_{j=1}^q a_jx_{s-j}\approx x_t-\mathbb E[x_t\mid\mathcal F_{t-1}]+\mathbb E[\mathbb E[x_t\mid\mathcal F_{t-1}]]\nonumber\\
&=\epsilon_t+\mathbb E[m_t]
\end{align}
for AR$(q)$, where $m_t:=\sum_{j=1}^q a_jx_{t-j}$,

$$y_{t}^{*,f}:=
\left(\frac1T\sum_{s=1}^T\sigma_s\right)\frac{x_t}{\sigma_t}
\approx
\frac{\mathbb E[\sigma_t]}{\sigma_t}x_t=\mathbb E[\sigma_t]\epsilon_t
$$

for GARCH$(1,1)$, and

$$y_{t}^{*,f}:=x_t-m_t+\frac1T\sum_{s=1}^T m_s\approx x_t-\mathbb E[x_t\mid\mathcal F_{t-1}]+\mathbb E[\mathbb E[x_t\mid\mathcal F_{t-1}]]$$
for scalar Kalman, where $m_t:=\mathbb E[x_t\mid\mathcal F_{t-1}]$

\begin{figure}[htbp]
    \centering

    \begin{subfigure}[b]{0.45\textwidth}
        \centering
         \includegraphics[width=\textwidth]{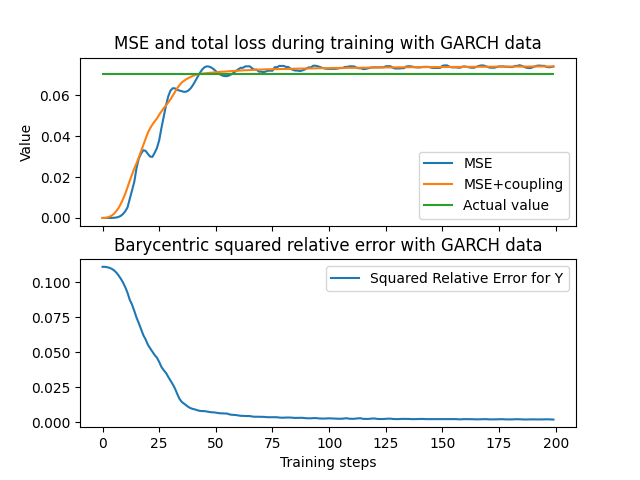}
        \caption{GARCH$(1,1)$}
        \label{fig:subgarch}
    \end{subfigure}
    \hfill
    \begin{subfigure}[b]{0.45\textwidth}
        \centering
        \includegraphics[width=\textwidth]{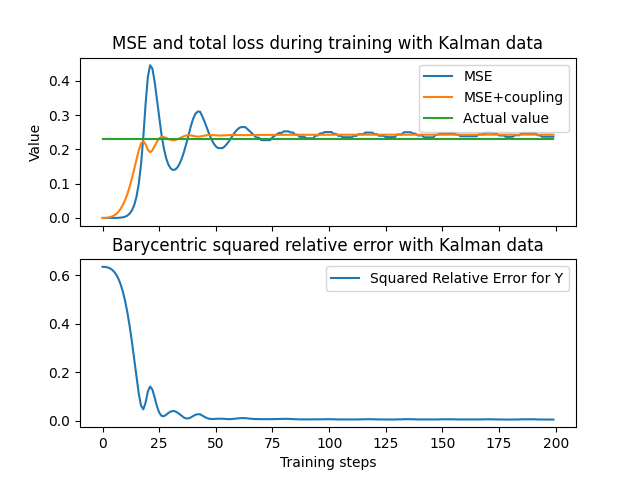}
        \caption{Kalman}
        \label{fig:subkalman}
    \end{subfigure}
    
    \hfill
    \begin{subfigure}[b]{0.45\textwidth}
        \centering
        \includegraphics[width=\textwidth]{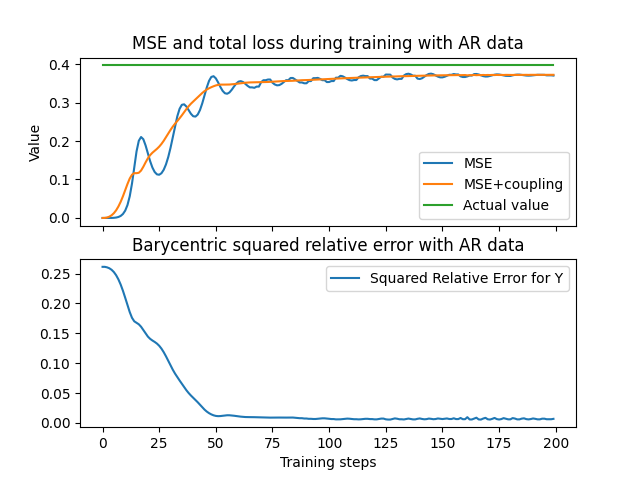}
        \caption{AR$(2)$}
        \label{fig:subar}
    \end{subfigure}
    
    \caption[Maxmin solver performance on synthetic data]{Maxmin solver performance on synthetic data. Plots showing the MSE $T^{-1}\sum_t(x_t-\widehat y_t)^2$ and the squared barycentric relative error $\frac{\sum_t(\widehat y_t-y_t^{*,f})^2}{\sum_t(y_t^{*,f})^2}$ during a run of Algorithm \ref{algo1}. Input data was initialized with a burn-in of $500$ observations to approach stationarity. Input size after burn-in was $4000$. $G_\theta$ was a GRU with hidden dim $d=2$ and linear readout, and $\varphi_\theta$ was a MLP with two hidden layers of dimension 7.  }
    \label{figu1}
\end{figure}

Figure \ref{figu1} shows the barycentric squared relative error
$$SRE_{\widehat y}=\frac{\sum_t(\widehat y_t-y_t^{*,f})^2}{\sum_t(y_t^{*,f})^2}$$ through 200 training steps of Algorithm \ref{algo1} with synthetically generated input data $(x_t)_{t=1\dots T}$ of length $T=4000$ with law a) GARCH$(1,1)$ with parameters $\omega=0.05$, $\alpha=0.18$, $\beta=0.75$ and Gaussian noise, b) Kalman (LG-SSM) with parameters $\alpha=0.8$, $\beta=1$, $\tau_\eta=\tau_\epsilon=0.4$, and c) AR$(2)$ with parameters $\phi_0=0$, $\phi_1=0.5$, $\phi_2=-0.2$ and Gaussian noise.

We also show the evolution of the mean squared error (MSE) $T^{-1}\sum_t(x_t-\widehat y_t)^2$ and compare it with the oracle equilibrium value $V$ for the transport cost $\sum_t(x_t-y_t^{*,f})^2$. The value for the MSE plus the learned penalty term $T^{-1}\sum_t \widehat z_t\hat\varphi_\theta(y_t)$ is also plotted. Table \ref{tab:final_values} summarizes the final values for these runs.

The results show that the algorithm effectively computes the  barycenter for all these processes.

\begin{table}[htbp]
    \centering
    \caption{Summary of final metrics across stochastic models}
    \label{tab:final_values}
    \begin{tabular}{llcccc}
        \toprule
        \textbf{Model} & \textbf{Parameters} & \textbf{Eq. $V$}  & \textbf{MSE} & \textbf{Coupling} & $\frac{\|\widehat y-y^*\|^2}{\|y^*\|^2}$ \\
        \midrule
        Kalman & $\alpha=0.8, \beta=1.0, \tau_\eta=0.4, \tau_\epsilon=0.4$ & $0.2301$  & $0.23728$ & $0.00600$ & $0.0059$ \\
        GARCH & $\omega=0.05, \alpha=0.18, \beta=0.75$ & $0.0702$ & $0.07413$  & $0.00005$ & $0.0019$ \\
        AR$(2)$ & $\phi_1=0.5, \phi_2=-0.2, \sigma^2=1.5$ & $0.3970$  & $0.37017$ & $0.00195$ & $0.0067$ \\
        \bottomrule
    \end{tabular}
\end{table}

\section{Duality failure and extension to higher rank}

 The previous section shows that our algorithm can effectively extract barycentric variables from samples of some well-known processes.

 This algorithm, however, would start to fail as soon as we apply it to input data coming from processes that do not satisfy the conditions of Corollary \ref{correcurrentduality} or for which maxmin duality fails, since the Algorithm solves the supinf problem, and for such processes the supinf value is strictly smaller than the infsup value that the barycenter solves.

 It is not hard to produce an example of a process that will fail to meet the conditions of Corollary \ref{correcurrentduality} and for which duality fails. Consider the following process, which we may call AR--GARCH

 \begin{equation}\label{eqargarch0}x_{t}=\phi_1 x_{t-1}+\phi_2 x_{t-2}+\eta_t,\qquad \eta_t=\sigma_t\epsilon_t,\qquad \sigma_t^2=\omega+\alpha\eta_{t-1}^2+\beta\sigma_{t-1}^2\quad (\alpha,\beta\geq 0)
 \end{equation}
where $\epsilon_t$ is an iid white noise independent of the past, and $|\phi_1|+|\phi_2|<1$, $\alpha+\beta<1$ to ensure a stationary solution exists.

To see that this process cannot have a rank one representation  $D(\omega,u)=f(\omega)g(u)$ for the quantile displacement, write the conditional quantile as
$$Q_\omega(u)=m_t(\omega)+\sigma_t(\omega)Q_\epsilon(u)$$
where
$m_t:=\phi_1x_{t-1}+\phi_2x_{t-2}$ and $Q_\epsilon$ is the quantile function of $\epsilon_t$. The centered conditional quantile displacement is therefore
\begin{equation}\label{eqargarch}D(\omega,u)=Q_\omega(u)-\bar Q(u)=(m_t(\omega)-\bar m)+(\sigma_t(\omega)-\bar\sigma)Q_\epsilon(u)
\end{equation}
where $\bar{Q}(u)=\mathbb{E}[m_t]+\mathbb{E}[\sigma_t] Q_\epsilon(u)$.

This generically cannot have a representation $D(\omega,u)=f(\omega)g(u)$. Indeed, if it did, then by (\ref{eqargarch}) the two past-measurable coefficients
$m_t-\bar m$
 and $\sigma_t-\bar\sigma$ would have to be in $\operatorname{span}(f)$, since $\sigma_t-\bar\sigma=f\frac{g(u_1)-g(u_2)}{Q_\epsilon(u_1)-Q_\epsilon(u_2)}$ for fixed $u_1, u_2$ such that $Q_\epsilon(u_1)\neq Q_\epsilon(u_2)$, and $m_t-\bar m=fg(u_1)-Q_\epsilon(u_1)(\sigma_t-\bar\sigma)$.

In other words, there are constants $a,b$, not both zero, such that
$$a(m_t-\bar m)+b(\sigma_t-\bar\sigma)=0
\qquad\text{a.s.}$$
But in the AR--GARCH model $m_t$ depends linearly on the past values $x_{t-1},x_{t-2}$, while $\sigma_t$  depends on the past squared innovations through $\sigma_t^2=\omega+\alpha\eta_{t-1}^2+\beta\sigma_{t-1}^2$. These cannot be linearly dependent random variables (except in degenerate  cases). Thus $D$ cannot have a rank one representation.

At the practical level, an algorithm like Algorithm \ref{algo1} that is a finite-sample solver for the population-level problem
$$\sup_{\psi,\varphi}\inf_{y_t}\mathbb E(x_t-y_t)^2+\mathbb E[(\psi(m_t,\sigma_t)-\bar \psi)\varphi(y_t)]$$
will fail to compute the barycenter $y_t^*=\bar m + \bar\sigma\epsilon_t$ and will return an inverse map of the restricted form
$$x_t=y_t+\frac{1}{2}(\psi(m_t,\sigma_t)-\bar \psi)\varphi'(y_t)$$

As a consequence, if the supremum in the maxmin problem is attained then it cannot equal the barycentric infsup value. Indeed, let
$$f(\omega):=\psi(m_t(\omega),\sigma_t(\omega))-\bar\psi$$
If the population-level maxmin problem attained the same value as the infsup problem, then the barycentric solution $y_t^*=\bar m+\bar\sigma\epsilon_t$ would have to be a minimizer of the inner maxmin problem for some choice of $f$ and $\varphi$, since $y_t^*$ being independent of the past implies the constraint term vanishes at $y_t^*$. Therefore equality of the maxmin and infsup values would force $y_t^*$ to satisfy 
$$2(y_t^*-x_t)+f(\omega)\varphi'(y_t^*)=0$$
Since $x_t-y_t^*=(m_t-\bar m)+(\sigma_t-\bar\sigma)\epsilon_t$, this gives 
$$f(\omega)\varphi'(\bar m+\bar\sigma\epsilon_t)=
2[(m_t-\bar m)+(\sigma_t-\bar\sigma)\epsilon_t]$$
The LHS is the product of a measurable function of the past and a function of the noise $\epsilon_t$. Evaluating this equation at two distinct values of $\epsilon_t$\footnote{In a subset of full $P_\epsilon$-measure of the essential support of $\epsilon_t$}, we get that both $m_t-\bar m$ and $\sigma_t-\bar\sigma$ must be in the span of $f$, and, hence, linearly dependent. As shown above, this is false for the AR--GARCH model. It follows that an attained maxmin value cannot equal the infsup value for the AR--GARCH process.

On the other hand, it is evident that for the AR--GARCH process the quantile discrepancy $D$ does have a rank two representation of the form 
$$D(\omega,u)=f_1(\omega)g_1(u)+f_2(\omega)g_2(u)$$
as equation (\ref{eqargarch}) shows. This suggests a natural substitute for Algorithm \ref{algo1} that is suitable for any scenario where $D$ has a rank 2 or, more generally, a rank $k$ representation for some finite $k$, i.e.
$$D(\omega,u)=f_1(\omega)g_1(u)+\dots+f_k(\omega)g_k(u),$$

This algorithm should be a finite-sample solver for the population-level problem
\begin{equation}\label{supinfk}\sup_{G,\{\psi^{(i)},\varphi^{(i)}\}_{i=1}^k}\inf_{y_t}\mathbb E(x_t-y_t)^2+\sum_{i=1}^k\mathbb E[(\psi^{(i)}(h_t^{(G)})-\bar \psi^{(i)})\varphi^{(i)}(y_t)]
\end{equation}

Indeed, a straightforward extension of Theorem \ref{teorankone} and Corollary \ref{correcurrentduality} shows that for such processes, with the usual regularity assumptions, the value of (\ref{supinfk}) equals the corresponding $\inf\sup$ value, whose unique solution for $y_t$ is the Monge representative of the barycentric distribution.

For the extension of Theorem  \ref{teorankone} we will write
$$\psi=(\psi^{(1)},\ldots,\psi^{(k)}),\qquad
\varphi=(\varphi^{(1)},\ldots,\varphi^{(k)})
$$
for a pair of vector-valued test functions and call the pair
{\em admissible} if $\psi^{(i)}(Z)$ is in $L^1_0$, $\varphi^{(i)}$ is differentiable and $\psi^{(i)}(Z)\varphi^{(i)}(Y)$ is in $L^1$ whenever a square-integrable $Y$ is independent of $Z$, for every $i=1,\ldots,k$.

\begin{theorem}[Criterion for rank $k$ duality]\label{teorankk}
Fix an integer $k\geq 1$. Let $(X,Z)$ be a pair of random variables with $X\in L^2(\mathbb R)$, and assume that the conditional laws $P(X\in\cdot\mid Z=z)$ are atomless.

Let $Q_z(u)$, $\bar Q(u)$, $D(z,u)$, $U$, $Y^*$ and $V$ be as in Theorem \ref{teorankone}, and let $\mathcal{A}_k$ be a class of admissible pairs of vector-valued test functions. 

Assume that $D$ has a representation
$$D(z,u)=\sum_{i=1}^k\psi^{(i)}(z)r_i(\bar Q(u))$$
for some $(\psi,\varphi)\in\mathcal{A}_k$, where
$(\varphi^{(i)})'(y)=2r_i(y)$ for $i=1,\ldots,k$. Assume moreover that, for a.e. $z$, the map
$T_z(y):=y+\sum_{i=1}^k\psi^{(i)}(z)r_i(y)$
is nondecreasing.

Then the pointwise Lagrangian
$$\ell_{\psi,\varphi}(z,u;y):=|Q_z(u)-y|^2+\sum_{i=1}^k\psi^{(i)}(z)\varphi^{(i)}(y)$$
is well-defined from below as an extended-real functional, and
$$\inf_Y\mathbb E[\ell_{\psi,\varphi}(Z,U;Y)]=V$$
Consequently,
$$\sup_{(\psi,\varphi)\in\mathcal A_k}\inf_Y\mathbb{E}\left[|X-Y|^2+\sum_{i=1}^k\psi^{(i)}(Z)\varphi^{(i)}(Y)\right]=V$$

Conversely, suppose that
$$\sup_{(\psi,\varphi)\in\mathcal A_k}\inf_Y
\mathbb E\left[|X-Y|^2+\sum_{i=1}^k\psi^{(i)}(Z)\varphi^{(i)}(Y)\right]=V$$
and suppose that the supinf is realized by an admissible pair $(\psi,\varphi)$. Then
$$D(z,u)=\frac{1}{2}\sum_{i=1}^k\psi^{(i)}(z)(\varphi^{(i)})'(\bar Q(u))$$
for a.e. $(z,u)$.
\end{theorem}

\begin{proof}
See Appendix \ref{app:proofs}.
\end{proof}

As in the rank one case, the following corollary applies the theorem to our time-dependent setting, where, as before, $(x_t)_{t\in\mathbb Z}$ is a square-integrable stationary process with atomless conditional laws $P(x_t\in\cdot\mid \mathcal{F}_{t-1})$.

In this setting we will say a class $\mathcal A_k$ of triples $(G,\psi,\varphi)$ (where $\psi=(\psi^{(1)},\ldots,\psi^{(k)})$ and $\varphi=(\varphi^{(1)},\ldots,\varphi^{(k)})$)
is {\em admissible} if each $\psi^{(i)}(h_t^G)$ is centered for $h_t^G$, and the pair $(\psi(h_t^G),\varphi)$ is admissible in the sense of Theorem \ref{teorankk} for $h_t^G$.

\begin{corollary}[Recurrent criterion for rank $k$ duality]\label{correcurrentdualityk}
Fix an integer $k\geq1$. Let $x_t$ be as above and write $\omega$ for a generic element of the underlying probability space. Let $y_t^*$, $U_t=F_{x_t\mid \mathcal{F}_{t-1}(\omega)}(x_t)$, $Q_\omega(u)$, $\bar Q(u)$, $D(\omega,u)$ and $V=V_{t}$ be as in Corollary \ref{correcurrentduality}.

Assume that
$$D(\omega,u)=\sum_{i=1}^k\psi^{(i)}(h_t^G(\omega))r_i(\bar Q(u))$$
for some $(G,\psi,\varphi)\in\mathcal A_k$, where
$$(\varphi^{(i)})'(y)=2r_i(y),\qquad i=1,\ldots,k.$$

Assume also that, for a.e. $\omega$, the map
$$T_\omega(y):=y+\sum_{i=1}^k\psi^{(i)}(h_t^G(\omega))r_i(y)$$
is nondecreasing. Then
$$\sup_{(G,\psi,\varphi)\in\mathcal A_k}\inf_{y_t}
\mathbb E\left[|x_t-y_t|^2+\sum_{i=1}^k\psi^{(i)}(h_t^G)\varphi^{(i)}(y_t)\right]=V$$

Conversely, suppose that
$$\sup_{(G,\psi,\varphi)\in\mathcal A_k}\inf_{y_t}
\mathbb E\left[|x_t-y_t|^2+\sum_{i=1}^k\psi^{(i)}(h_t^G)\varphi^{(i)}(y_t)\right]=V$$
and suppose that the supinf is realized by a triple $(G,\psi,\varphi)\in\mathcal{A}_k$. Then
$$D(\omega,u)=\frac12\sum_{i=1}^k\psi^{(i)}(h_t^G(\omega))(\varphi^{(i)})'(\bar Q(u))$$
for a.e. $(\omega,u)$.
\end{corollary}

\begin{proof}
Apply Theorem \ref{teorankk} with $X=x_t$ and with the covariate given by the past $\mathcal F_{t-1}$, which can always be encoded by some measurable function. The conditional quantile is
$$Q_\omega(u)=F^{-1}_{x_t\mid\mathcal F_{t-1}(\omega)}(u),$$
and the barycentric Monge solution is $y_t^*=\bar Q(U_t)$.

Thus the assumption
$$D(\omega,u)=\sum_{i=1}^k\psi^{(i)}(h_t^G(\omega))r_i(\bar Q(u))$$
corresponds to the rank $k$ representation assumption in Theorem \ref{teorankk}.

The converse follows from the converse part of Theorem \ref{teorankk}.
\end{proof}

Having proved the necessary extension of the duality characterization for processes with higher rank $k$, we are now ready to introduce the solver based on it. 

\begin{algorithm}
\caption{Recurrent Barycentric Maxmin Solver with $k$ channels}
\label{algo2}
\begin{algorithmic}
\Require $(x_t)_{1\leq t\leq T}$, $N$, $k$
\State Initialize $\{\lambda_i\}_{i=1}^k, h_0\in\mathbb{R}^d$, $(y_t)_{1\leq t\leq T}$, $\varphi^{(1)}_\theta,\dots,\varphi^{(k)}_\theta$ and $G_\theta$
\State $n \gets 0$
\For{$n=1\dots N$}
    \State Compute $(z_t)_{1\leq t\leq T}\subseteq\mathbb{R}^k$ using $G_\theta$, $h_0$ and $(x_t)_{1\leq t\leq T}$
    \State $z_t^{(i)}\gets \frac{z_t^{(i)}-\bar{z}^{(i)}}{\operatorname{sd}(z^{(i)})} $ for $i=1\dots k$
    \State $L\gets\sum_t \|x_t-y_t\|^2+\sum_{i=1}^k\lambda_i\sum_t z_t^{(i)}\varphi_\theta^{(i)}(y_t)$
    \State Update $(y_t)_{1\leq t\leq T}$ to numerical near optimality minimizing $L$
    \State Update $\theta,\{\lambda_i\}_{i=1}^k$ doing a gradient ascent step wrt $L$
\EndFor

\Return $(y_t)_{1\leq t\leq T}$
\end{algorithmic}
\end{algorithm}
\noindent

Here, as before, $G_\theta$ is implemented as a GRU and coupled with a readout $z_t=Ah_t+a$ with $A\in\mathbb{R}^{k\times d}$ and $a\in\mathbb{R}^k$ so that $z_t\in\mathbb{R}^k$. For $\varphi$ we tried two implementations, one uses a MLP with two hidden layers, the other is monomial and fixes $\varphi^{(k)}(y)=y^k$.

Figure \ref{figu2} shows two performance metrics ($\text{MSE}/V$ and SRE of the learned barycenter $\widehat y_t$) for 4 typical runs of Algorithm \ref{algo2} with neural/MLP $\varphi$ on synthetic input data generated with the AR--GARCH equations (\ref{eqargarch0}) with parameters $\phi_1=0.5,\phi_2=-0.2,\omega=0.05,\alpha=0.18,\beta=0.75$ . Notice that for the AR--GARCH model the oracle finite-sample barycenter is $$y_t^{*,f}=\phi_1\bar x_{-1}+\phi_2\bar x_{-2}+\left(\frac{1}{T}\sum_{s=1}^T\sigma_s\right)\frac{\eta_t}{\sigma_t}.$$
We ran the algorithm on the same seed to test and compare the two-channel ($k=2$) setup, where two pairs of test functions or channels, $(z_t^{(1)},\varphi^{(1)})$ and $(z_t^{(2)},\varphi^{(2)})$, are used, against $k=1$, which is equivalent to the setup of Algorithm \ref{algo1}. For both setups we considered two possibilities for the hidden dimension of the GRU, $d=2$ and $d=4$. The plots suggest that for $k=2$ the GRU with hidden dimension 2 is not expressive enough to find the optimal pair of test functions $z_t^{(1)},z_t^{(2)}$, which should capture both the mean, $m_t$, and scale, $\sigma_t$, components. Its performance is therefore close to that of the $k=1$ algorithm, whose optimization equals that of Algorithm \ref{algo1}. Because $k=1$ optimizes with the wrong penalty, the barycentric variables are not forced to be independent of both the location and scale components of the past, and the transport cost ends up below the barycentric value, while the squared relative error against the oracle barycenter $y_t^{*,f}$ is of order $10^{-1}$ at the end of the simulation. On the other hand, when $d=4$ and $k=2$, the algorithm has both enough expressivity ($d=4$) and channel capacity ($k=2$) to be able to represent both components of the rank-two displacement, and the squared relative error falls below $10^{-2}$. Notice that when $d=4$ and $k=1$ there is no improvement in performance, so that expressivity alone, without channel capacity, cannot alter the intrinsic limitations of Algorithm \ref{algo1}. The final values for these runs are presented in Table \ref{tab:argarch}.

\begin{figure}[htbp]
        \centering
         \includegraphics[width=0.65\textwidth]{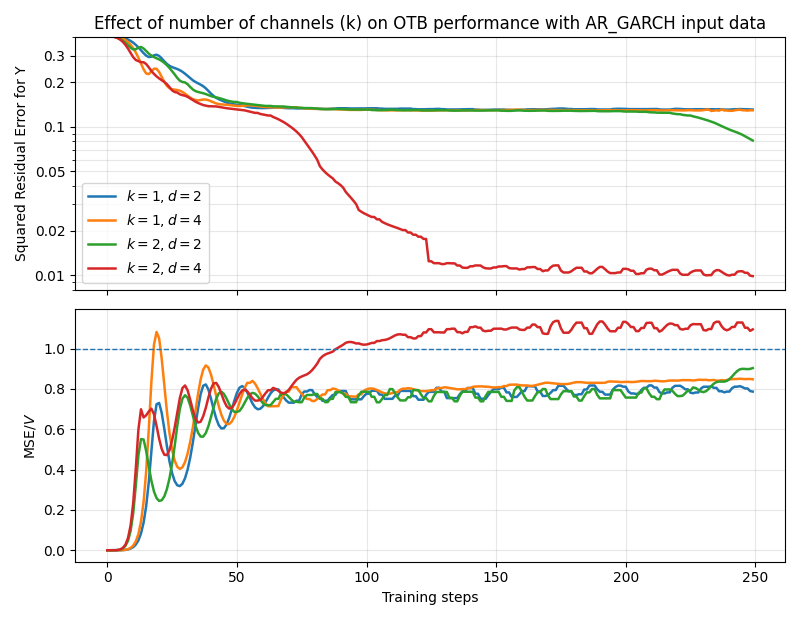}
    \hfill
    \caption[One and two-channel Maxmin solver with AR--GARCH data]{One and two-channel Maxmin solver with AR--GARCH data. Plot of the MSE/V $\frac{\sum_t(x_t-\widehat y_t)^2}{\sum_t(x_t-y_t^*)^2}$ and the squared barycentric relative error $\frac{\sum_t(\widehat y_t-y_t^{*,f})^2}{\sum_t(y_t^{*,f})^2}$ during a 250-step run of Algorithm \ref{algo2} for different numbers of channels, $k=1$ and $k=2$. The same input path of size $T=4000$ generated by the AR--GARCH model was used for all cases. $\varphi_\theta$ was a MLP with two hidden layers of dimension 7. Notice the vertical axis for the SRE plot has a log-scale. }
    \label{figu2}
\end{figure}
 
\begin{table}[ht]
\centering
\begin{tabular}{c c c c c}
\hline
$k$ & $d$ & Final $\|\widehat y-y^*\|^2/\|y^*\|^2$ & Best $\|\widehat y-y^*\|^2/\|y^*\|^2$ & Final $\mathrm{MSE}/V$ \\
\hline
1 & 2 & 0.1306 & 0.1295 & 0.7867 \\
1 & 4 & 0.1290 & 0.1277 & 0.8477 \\
2 & 2 & 0.0807 & 0.0807 & 0.9034 \\
2 & 4 & 0.0099 & 0.0099 & 1.0952 \\
\hline
\end{tabular}
\caption[Experiment comparing one-channel ($k$=1) and two-channel ($k=2$) recurrent solvers on AR--GARCH input data]{Experiment comparing one-channel ($k$=1) and two-channel ($k=2$) recurrent solvers on AR--GARCH input data. Parameters were $\phi_1=0.5,\phi_2=-0.2,\omega=0.05,\alpha=0.18,\beta=0.75$. The oracle value is $V=0.2517$. The runs with $k=1$ channel remain far from the oracle barycentric $y_t^*$, while the solver with $k=2$ channels $d=4$ approximates it much more accurately.}
\label{tab:argarch}
\end{table}
 
Of course not all processes will have a rank $k$ representation for some finite $k$. For such processes, the supinf problem (\ref{supinfk}) will not coincide with the $\inf\sup$ barycentric problem.

The following simple model, sometimes used to model volatility of financial returns, is generally not expected to admit a rank $k$ representation for any $k$. Let
\begin{align}
    r_t&=e^{(h_t/2)}\epsilon_t\nonumber\\
    h_{t+1}&=\mu+\phi(h_t-\mu)+\eta_t
\end{align}
where $r_t$ is the observed process, $h_t$ the unobserved state, $\phi$ and $\mu$ are constants and $\epsilon_t$ and $\eta_t$ are standard Gaussian iid noise independent of each other\footnote{Notice that, as in the GARCH model, returns are modeled as the product of an iid noise $\epsilon_t$ and a latent volatility. Unlike the GARCH model, however, this latent volatility $e^{h_t/2}$ depends both on its past value and on an independent noise $\eta_t$. This makes it a state space model like the Kalman system.}. The model can be linearized by defining $\gamma=\mathbb{E}[\log\epsilon_t^2]$ and $x_t=\log r_t^2=h_t+\log\epsilon_t^2$, and centered to obtain  
\begin{align}
     x_t^c&=h_t^c+\zeta_t\nonumber\\
    h_{t+1}^c&=\phi h_t^c+\eta_t
\end{align}
where $\zeta_t=\log\epsilon_t^2-\gamma$ and $h^c_t=h_t-\mu$.

We thus arrive at a linear state space model. Unlike the Kalman system, however, it is not Gaussian, since $\zeta_t$ is not Gaussian. It is in fact highly skewed, and this affects the computation of the posterior distribution $P(h_t\mid x_{\leq t})$, which, unlike the Kalman posterior, is not Gaussian and cannot even be confined to any finite-dimensional manifold in the space of probability distributions. More generally, exact finite-dimensional filters are known to exist only for special classes of nonlinear state space models, so that the posterior distribution in a general nonlinear or non-Gaussian model must be treated as an infinite-dimensional, measure-valued state.

 However, in such cases Algorithm \ref{algo2} with sufficient capacity $k$ can still be numerically effective to approximate the barycentric solution. A maximizing player with enough channel capacity (equivalently, enough test functions $\psi^{(i)},\varphi^{(i)}$) should be able force the $y$ variable to be `nearly-independent' of the past.

\section{Predictive sampling: the Recurrent Barycentric Encoder}

\subsection{Construction of the RBE from the learned recurrence and transport map}

We are now ready to introduce the main application of the previous construction. As we have seen, the a.s. atomlessness of $P(\cdot\mid z)$ \footnote{This condition has to be strengthened to absolute continuity for $x$ of dimension 2 or higher.} not only guarantees the existence of the Monge map $T_z:=T(\cdot,z)$ that solves the OTBP but also its a.s. invertibility. We can therefore use the inverse of this map to pull-back the barycentric distribution to the conditional distribution $P(x\in \cdot\mid z)$. For our time-dependent version of the OTBP, $z$ is replaced by some variable encoding the past $x_{\le t}$, and the conditional distribution that we will sample is $P(x_{t+1}\mid x_{\le t})$, i.e. we will sample the predictive distribution of the next observation.

Computing the inverse map $T_{x_{\le t}}^{-1}$ is relatively straightforward from our chosen algorithm. Indeed, for the $k$-channel objective given $z$, $\varphi$ and $\lambda_i$'s
$$L(y)=\sum_{t=1}^T |x_t-y_t|^2+\sum_{i=1}^k\sum_{t=1}^T \lambda_iz_t^{(i)}\varphi^{(i)}(y_t)$$
the first-order condition of the optimal response $y_t$ is
$$2(y_t-x_t)+\sum_{i=1}^k \lambda_iz_t^{(i)}\partial_y\varphi^{(i)}(y_t)=0,\qquad \text{ for all $t$}$$
Equivalently,
$$x_t=y_t+\frac12\sum_{i=1}^k\lambda_i z_t^{(i)}\partial_y\varphi^{(i)}(y_t)$$

Thus, after learning the multiplier $\widehat\lambda_i$, test function $\varphi_{\widehat \theta}^{(i)}$ and the
GRU with linear readout $G_{\widehat \theta}$  at the end of training with Algorithm \ref{algo2}, the learned inverse map $\widehat T_{x'_{\le t}}^{-1}$ associated with a path $x'_{\le t}=x'_0,\dots,x'_{t}$ is constructed as\footnote{Strictly speaking, $\widehat T_{x'_{\le t}}^{-1}$ defines an inverse for a transport map only when it is strictly increasing as a function of $y$. In any case, this issue did not arise in practice, since all the experiments below yielded learned maps that were strictly increasing on the range of the learned barycentric variables.}

\begin{equation}\label{eqinverse}\widehat T_{x'_{\le t}}^{-1}(y):=y+\frac12\sum_{i=1}^k\widehat\lambda_i \widehat z^{(i)}_{t+1}\partial_y\varphi_{\hat\theta}^{(i)}(y)
\end{equation}
where $\widehat z^{(i)}$ is obtained from applying the learned $G_{\widehat \theta}$ to the new sample path $x'_{\le t}$.

The learned predictive distribution for a new path $(x'_t)_{t=1\dots T'}$ at each time $t+1$ is then
$$\widehat{P}_{t+1}=\left(\widehat{T}^{-1}_{x'_{\leq t}}\right)_\#\widehat{\mu}$$
where, in our implementation, $\widehat\mu$ will be the empirical distribution of the learned barycentric variables $\widehat{y}_t$. This could be replaced with a smoothing kernel or some other density approximator based on the empirical distribution of the $\widehat{y}_t$. The construction of $\widehat T^{-1}$ and $\widehat{P}$ from the learned artifacts $G_{\widehat \theta},\varphi_{\widehat\theta}$ and $\widehat\lambda$ constitute the {\em Recurrent Barycentric Encoder}. 

Notice the alternative path $x'_{\le t}$ may or may not overlap with the $x_{1:T}$ used for training.

Of course our main interest is in using this conditional sampler with samples $x_t'$ that are different from the training sample. For such test samples, we expected the push forward under $\hat{T}_{x'_{\leq t}}^{-1}$ of the learned $\widehat y_t$ to be a good approximation of the distribution of $x'_{t+1}$ conditioned on the past $\{x_{s}'\}_{s\leq t}$.

To test whether this is true we ran a series of experiments. For the first experiment we used input data generated by some of the models we have seen. After training the recurrent unit and the barycentric variables with Algorithm \ref{algo2} on a path of length $T=4000$, we generated a new synthetic path $x'_{1:T}$ and ran the trained recurrence $G_{\widehat\theta}$ through  $x'_1,\dots,x'_{T-1}$ to compute the normalized readout $\widehat z_{T}$. Next, we sampled independent values $y^{(1)},\dots,y^{(M)}$ from the previously obtained $\{\widehat y_s\}$, and returned
$$\widehat x_T^{(j)}=\widehat T_{x'_{\le T-1}}^{-1}(y^{(j)}),\qquad j=1,\dots,M$$
We then tested how close the empirical distribution of these samples was to the true conditional distribution $P(x_T\in\cdot\mid x'_{\leq T-1})$
by generating samples $x_T^{*(j)}$, $j=1,\dots,M$ using this true conditional distribution conditioned on $x'_{1:T-1}$, which we can do because the true conditional distribution is known to us in closed form when using synthetic data from our models of choice, which were AR--GARCH, Kalman, AR$(2)$ and GARCH, all with standard Gaussian innovations $\epsilon_t$.

\begin{table}[ht]
\centering
\begin{tabular}{lcccccc}
\hline
Process
& $\widehat{\mu}$
& $\widehat{\sigma}$
& $\mu_{\mathrm{oracle}}$
& $\sigma_{\mathrm{oracle}}$
& $W_1/\sigma_{\mathrm{oracle}}$
& KS \\
\hline
AR--GARCH
& $-0.2338$
& $0.6441$
& $-0.2507$
& $0.6582$
& $0.0266$
& $0.0136$ \\

GARCH
& $0.0193$
& $0.5856$
& $0.0000$
& $0.5742$
& $0.0401$
& $0.0291$ \\

AR
& $0.4687$
& $1.2132$
& $0.5751$
& $1.2247$
& $0.0911$
& $0.0369$ \\

Kalman
& $-0.2143$
& $0.6094$
& $-0.2352$
& $0.6158$
& $0.0337$
& $0.0234$ \\
\hline
\end{tabular}
\caption[One-step conditional sampling diagnostics]{One-step conditional sampling diagnostics. Here
$\widehat{\mu}$ and $\widehat{\sigma}$ are the mean and standard deviation
of the learned RBE distribution, while
$\mu_{\text{oracle}}$ and $\sigma_{\text{oracle}}$ are the corresponding quantities for the sample generated by the true conditional. Sample size is $M=3000$ in both cases.}
\label{tab3}
\end{table}

Table \ref{tab3} summarizes the results of these tests for the four processes. Since all conditional distributions are Gaussian, they are completely determined by their mean $\mu$ and their standard deviation $\sigma$. The table shows the sample mean $\widehat \mu$ and sample standard deviation $\widehat \sigma$ for the $\widehat T^{-1}_{x'_{\le T-1}}$-generated samples $\{\widehat x_T^{(j)}\}_{j=1}^M$, against the sample mean and standard deviation, $\mu_{\text{oracle}}$ and $\sigma_{\text{oracle}}$, of the oracle samples $\{x_T^{*(j)}\}_{j=1}^M$. The $W_1/\sigma_{\text{oracle}}$ column shows the 1-Wasserstein distance between the two distributions normalized by the oracle standard deviation, while the KS column corresponds to the Kolmogorov-Smirnov distance.

Figure \ref{figu3} shows the histograms for the $\widehat T^{-1}_{x'_{\le T-1}}$-generated samples $\{\widehat x_T^{(j)}\}_{j=1}^M$ against the oracle samples $\{x_T^{*(j)}\}_{j=1}^M$ generated by the true conditional distribution on the same past.

\begin{figure}[htbp]
    \centering

    \begin{subfigure}[b]{0.45\textwidth}
        \centering
         \includegraphics[width=\textwidth]{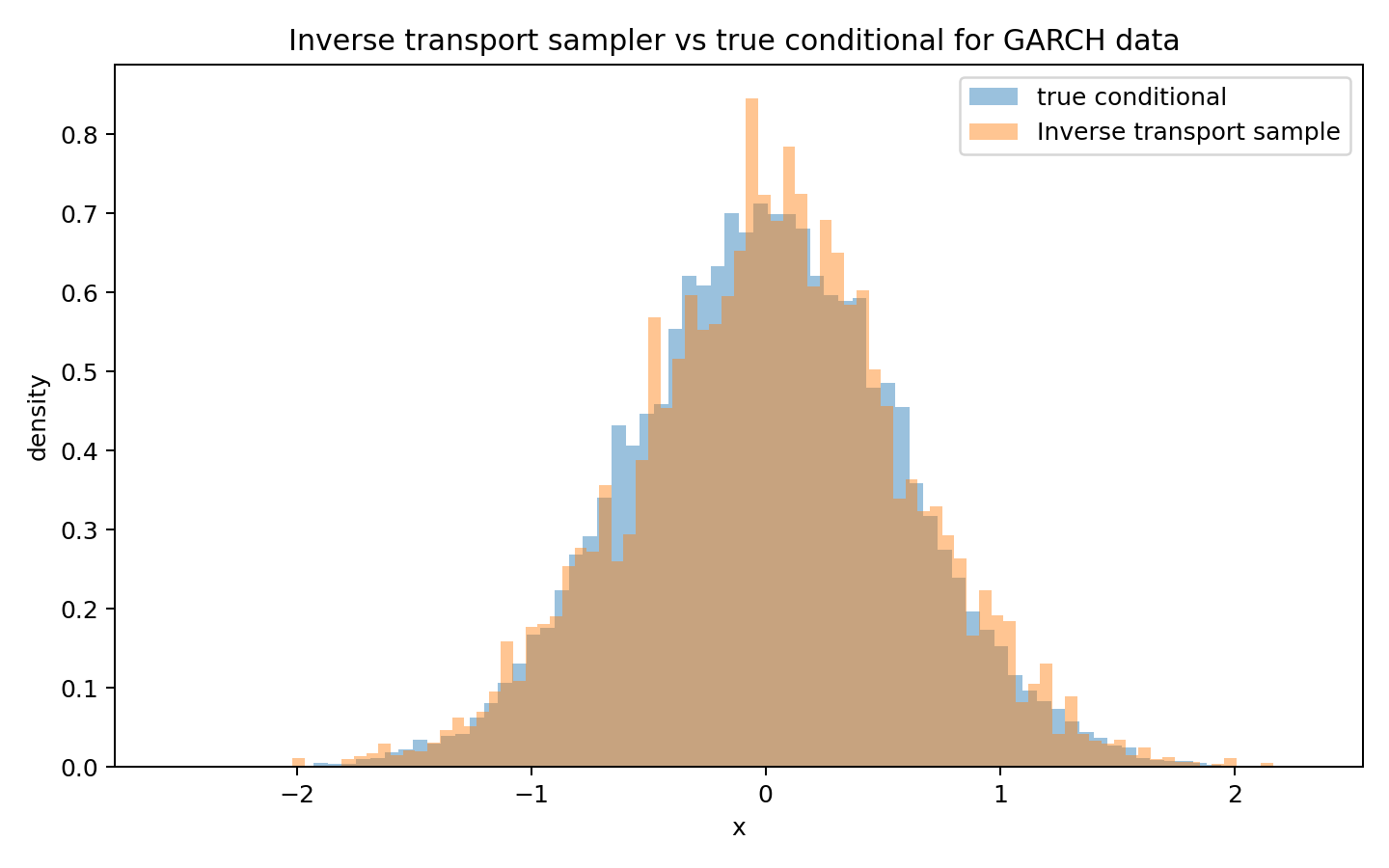}
        \caption{GARCH$(1,1)$}
        \label{fig:subgarchinv}
    \end{subfigure}
    \hfill
    \begin{subfigure}[b]{0.45\textwidth}
        \centering
        \includegraphics[width=\textwidth]{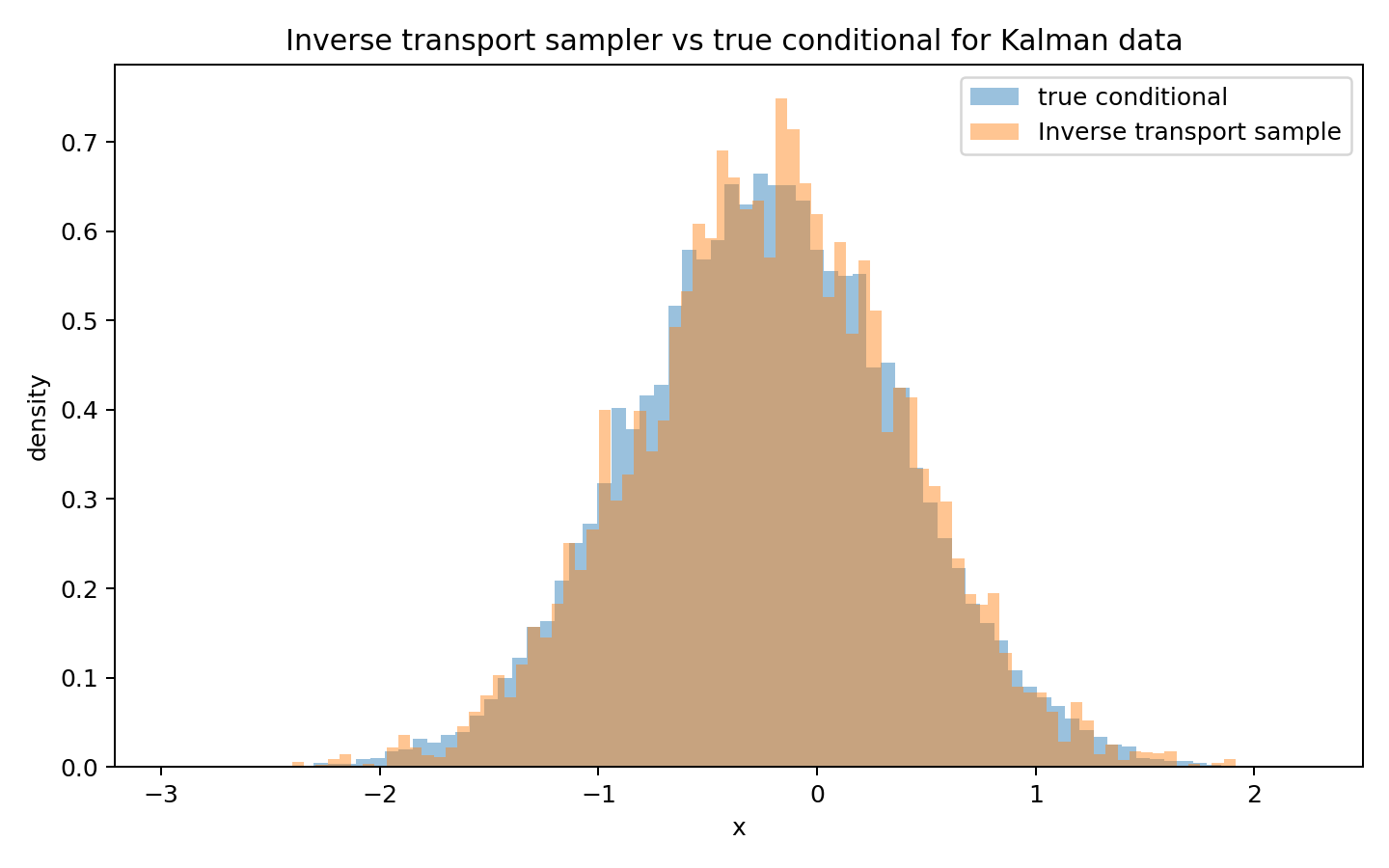}
        \caption{Kalman}
        \label{fig:subkalmaninv}
    \end{subfigure}
    
    \hfill
    \begin{subfigure}[b]{0.45\textwidth}
        \centering
        \includegraphics[width=\textwidth]{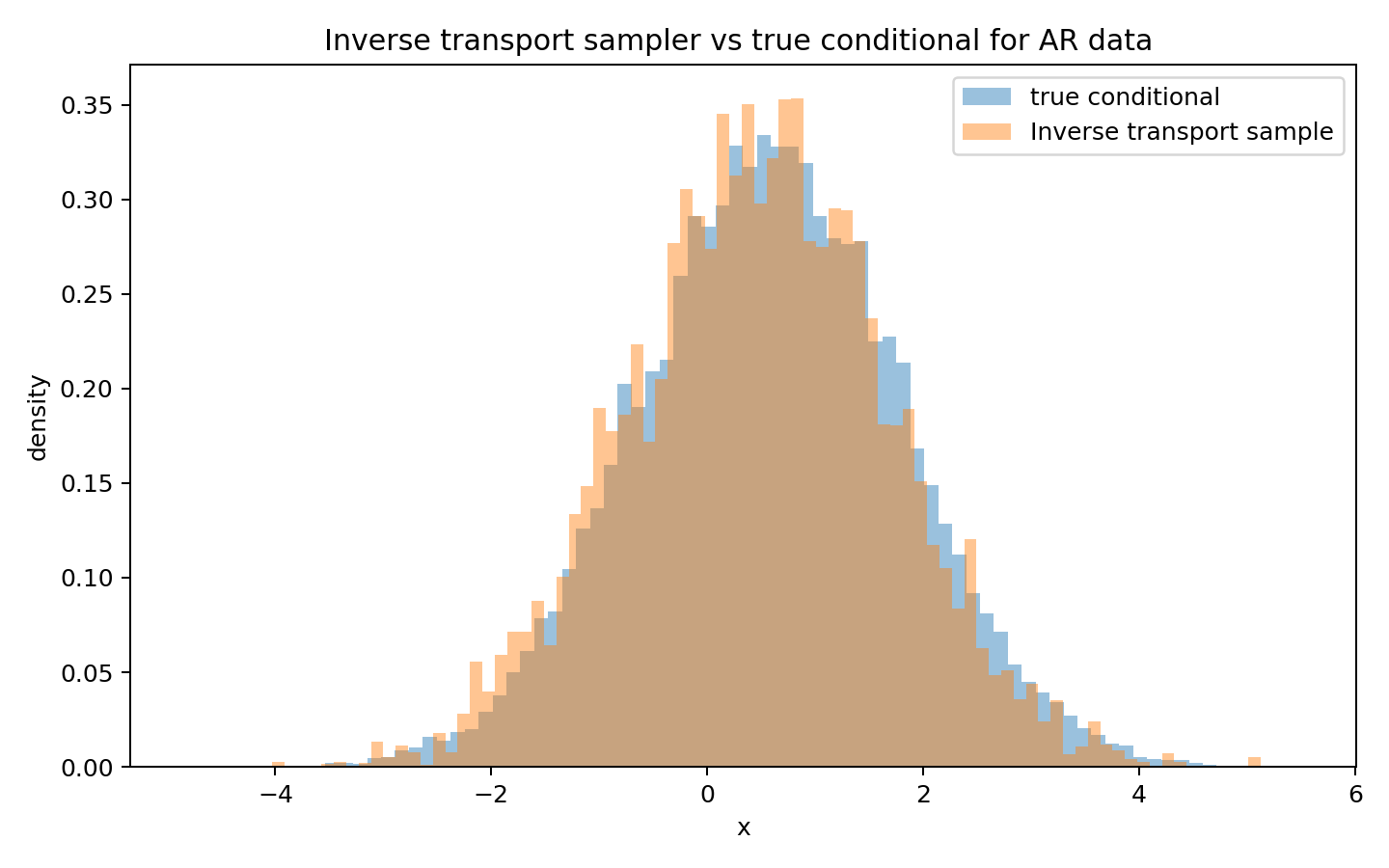}
        \caption{AR$(2)$}
        \label{fig:subarinv}
    \end{subfigure}
    \hfill
    \begin{subfigure}[b]{0.45\textwidth}
        \centering
        \includegraphics[width=\textwidth]{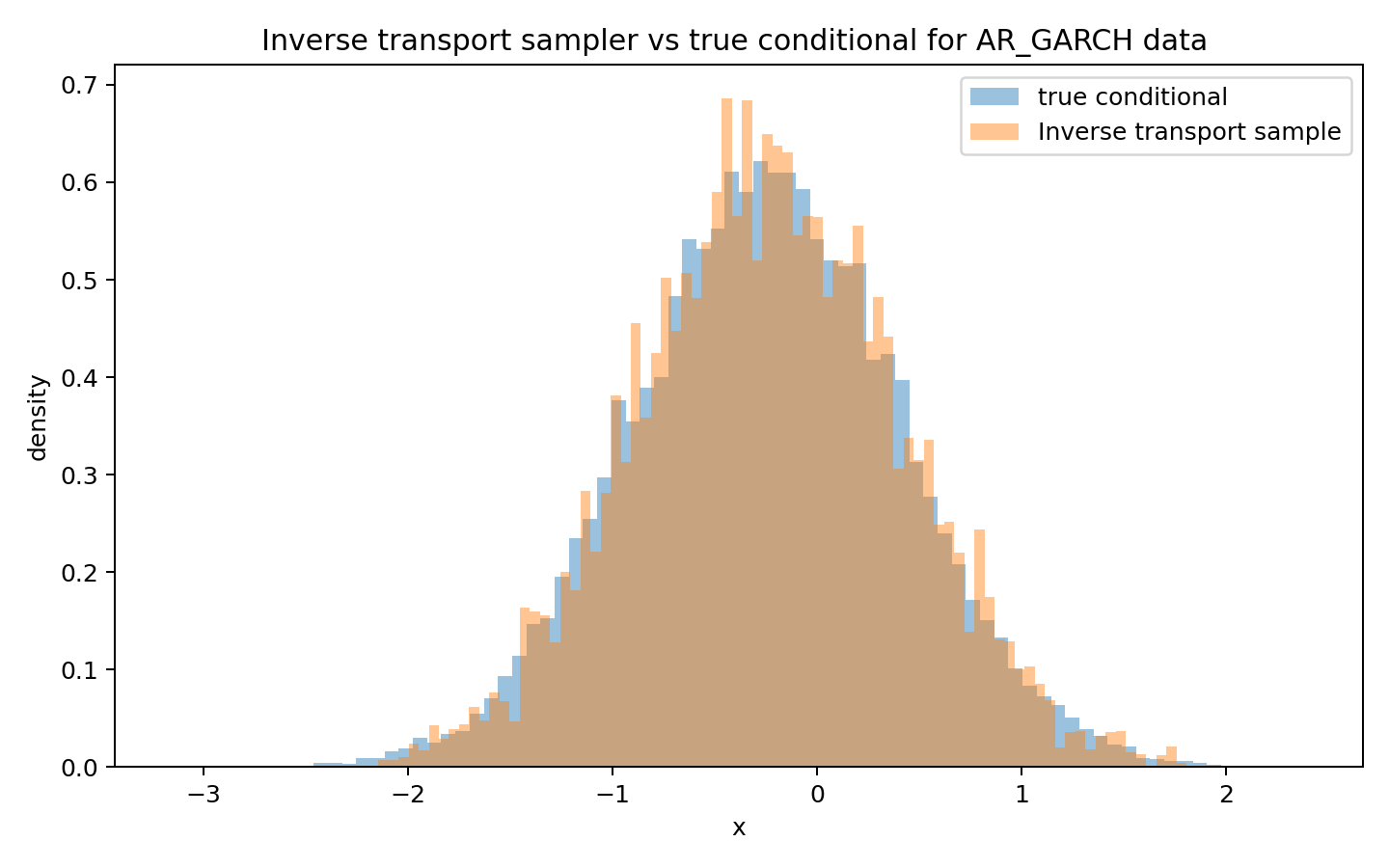}
        \caption{ARGARCH}
        \label{fig:subargarchinv}
    \end{subfigure}

    \caption{Histograms for the RBE sampler against the true conditional sampler.  }
    \label{figu3}
\end{figure}

\subsection{Systematic calibration test}

To generalize these results we performed a more systematic experiment. As before, we trained the algorithm on a path of length $T=4000$ and then generated a new independent test path of length 4000. From each test path we sampled $Q=250$ distinct times $t_q$, each having at least 300 observations of available past. For each of these histories, we drew a fixed sample of $M=3000$ values from the learned barycentric variables $\{\widehat y_s\}$ and centered this realized sample so that $\frac{1}{M}\sum_{j=1}^M y^{(j)}=0$

We then applied the learned inverse map corresponding to each test history,
$$\widehat x_{t_q}^{(j)}=\widehat T_{x'_{\leq t_q-1}}^{-1}(y^{(j)}),\qquad j=1,\dots,M$$
and proceeded to compute several discrepancy metrics.

For each history we computed the normalized mean and scale errors $$e_{\mu,q}=\frac{\widehat\mu_q-\mu_q}{\sigma_q},
\qquad e_{\sigma,q}=\frac{\widehat\sigma_q-\sigma_q}{\sigma_q}$$
where $\widehat\mu_q$ and $\widehat\sigma_q$ are the sample mean and standard deviation of the generated samples, and $\mu_q$ and $\sigma_q$ are the mean and standard deviation of the true conditional distribution
$P(x_{t_q}\in\cdot\mid x'_{\leq t_q-1})$,

We also computed the normalized Wasserstein distance
$$\frac{W_1\!\left(\widehat P_q^M,P(x_{t_q}\in\cdot\mid x'_{\leq t_q-1})\right)}{\sigma_q}$$
where 
$$\widehat P_q^M=\frac{1}{M}\sum_{j=1}^M\delta_{\widehat x_{t_q}^{(j)}}$$

Notice that because the true conditional distribution is Gaussian for all four models, $\mu_q$, $\sigma_q$ and the Wasserstein distances can be computed directly against the true Gaussian conditional distribution with no need to generate an additional finite oracle sample. The same is true of the Kolmogorov-Smirnov (KS) distance. Both metrics are shown on Table \ref{tableInverse}. The results confirm that the RBE sampler stays close to the true conditional distribution on out-of-sample paths. The mean normalized 1-Wasserstein distance ranges from $0.0515$ for GARCH to $0.0835$ for AR--GARCH, while its $95$th percentile remains below $0.17$ for all processes, and the mean KS distance remains below $0.045$.

\begin{table}[htbp]
\centering
\begin{tabular}{lccccccc}
\hline
& \multicolumn{4}{c}{$W_1/\sigma_{\mathrm{oracle}}$}
& \multicolumn{3}{c}{KS} \\
\cline{2-5}\cline{6-8}
Process
& Mean
& Median
& $q_{0.95}$
& Maximum
& Mean
& $q_{0.95}$
& Maximum \\
\hline
AR$(2)$
& $0.0694$
& $0.0530$
& $0.1406$
& $1.0590$
& $0.0373$
& $0.0681$
& $0.4412$ \\

GARCH
& $0.0515$
& $0.0439$
& $0.1056$
& $0.2053$
& $0.0254$
& $0.0438$
& $0.0872$ \\

Kalman
& $0.0549$
& $0.0471$
& $0.1133$
& $0.1807$
& $0.0298$
& $0.0535$
& $0.0826$ \\

AR--GARCH
& $0.0835$
& $0.0753$
& $0.1693$
& $0.4755$
& $0.0444$
& $0.0818$
& $0.2565$ \\
\hline
\end{tabular}
\caption[Statistics and discrepancy metrics over the same $Q=250$ alternative test paths]{Statistics and discrepancy metrics over the same $Q=250$ alternative test paths. For such path, the RBE predictive distribution has $M=3000$ observations. The Wasserstein metric was normalized by the oracle conditional standard deviation $\sigma_q$.}
\label{tableInverse}
\end{table}

It must be noted that the maxima for both the $W_1$ and the KS metric with AR$(2)$ data are considerably larger than its mean, median and $95$th percentile. This is caused by a small number of extreme test histories for which the learned GRU readout is far outside the region most frequently observed during training. The median normalized 1-Wasserstein distance is only $0.0530$, and its $95$th percentile is $0.1406$, showing that this large maximum is not representative of its usual out-of-sample performance. AR--GARCH also has some larger errors in unusually high $\sigma_q$ states.

We also computed the coefficients of determination $R_\mu^2$ and $R_\sigma^2$ between the generated and oracle conditional means and standard deviations across the $Q$ histories. These are shown on Table \ref{tabsystematic}. Notice that for GARCH the true conditional mean is constant, and for AR$(2)$ and Kalman the same is true of the conditional standard deviation, so that the corresponding coefficients of determination are not defined. These are hence omitted.

The results show that the conditional means are recovered accurately whenever they vary with the past. The corresponding coefficients of determination are $0.9592$ for AR$(2)$, $0.9953$ for Kalman and $0.9420$ for AR--GARCH. For GARCH the true conditional mean is identically zero, and its mean absolute normalized estimation error is only $0.0008$.

\begin{table}[htbp]
\centering
\begin{tabular}{lcccccc}
\hline
Process
& $\mathrm{SRE}_{\mathrm{train}}$
& $R_\mu^2$
& $R_\sigma^2$
& $\overline{|e_\mu|}$
& $\overline{|e_\sigma|}$ \\
\hline
AR$(2)$
& $0.0065$
& $0.9592$
& ---
& $0.0595$
& $0.0275$ \\

GARCH
& $0.0033$
& ---
& $0.9421$
& $0.0008$
& $0.0496$ \\

Kalman
& $0.0056$
& $0.9953$
& ---
& $0.0370$
& $0.0312$ \\

AR--GARCH
& $0.0114$
& $0.9420$
& $0.7939$
& $0.0587$
& $0.0594$ \\
\hline
\end{tabular}
\caption[Conditional mean and std calibration test over $Q=250$ out-of-sample histories.]{Conditional mean and std calibration test over $Q=250$ out-of-sample histories. The final training SRE is, as before, the relative squared error of the learned barycentric variables with respect to the oracle barycentric variables. The last two columns show the mean absolute normalized conditional mean and scale errors. Notice the $R^2$ value is left blank when the corresponding oracle parameter is constant.}
\label{tabsystematic}
\end{table}

The conditional scale $\sigma_q$ is also accurately recovered for GARCH, with $R_\sigma^2=0.9421$. For AR$(2)$ and Kalman the true conditional standard deviation is constant, so $R_\sigma^2$ is not defined, but their mean absolute normalized scale errors are only $0.0275$ and $0.0312$, respectively. The scale component of AR--GARCH is more difficult to recover. The mean absolute normalized scale error remains $0.0594$, but its $R_\sigma^2$, $0.7939$, is lower. Most of this discrepancy comes from an underestimation of the conditional scale $\sigma_q$ for some of the largest values, as can be seen in the corresponding plot (Figure \ref{fig:inverse_calibration_argarch}). A possible explanation lies in the fact that the squared residual error from training was noticeably higher ($0.0114$) for the AR--GARCH data than for the other processes ($0.0033-0.0065$).

Figures \ref{fig:inverse_calibration_ar}--\ref{fig:inverse_calibration_argarch} give a general picture of all calibration results. The first row compares the generated and oracle conditional means and standard deviations and shows the normalized mean error $e_{\mu,q}$ for each of the four processes. The second row shows the normalized scale error $e_{\sigma,q}$, the normalized 1-Wasserstein distance and the KS distance. Notice the GARCH conditional mean has a vertical plot because its oracle conditional mean is identically zero. Similarly, the oracle conditional standard deviation is constant for AR$(2)$ and Kalman.

\begin{figure}[p]
    \centering
    \includegraphics[width=\textwidth]
    {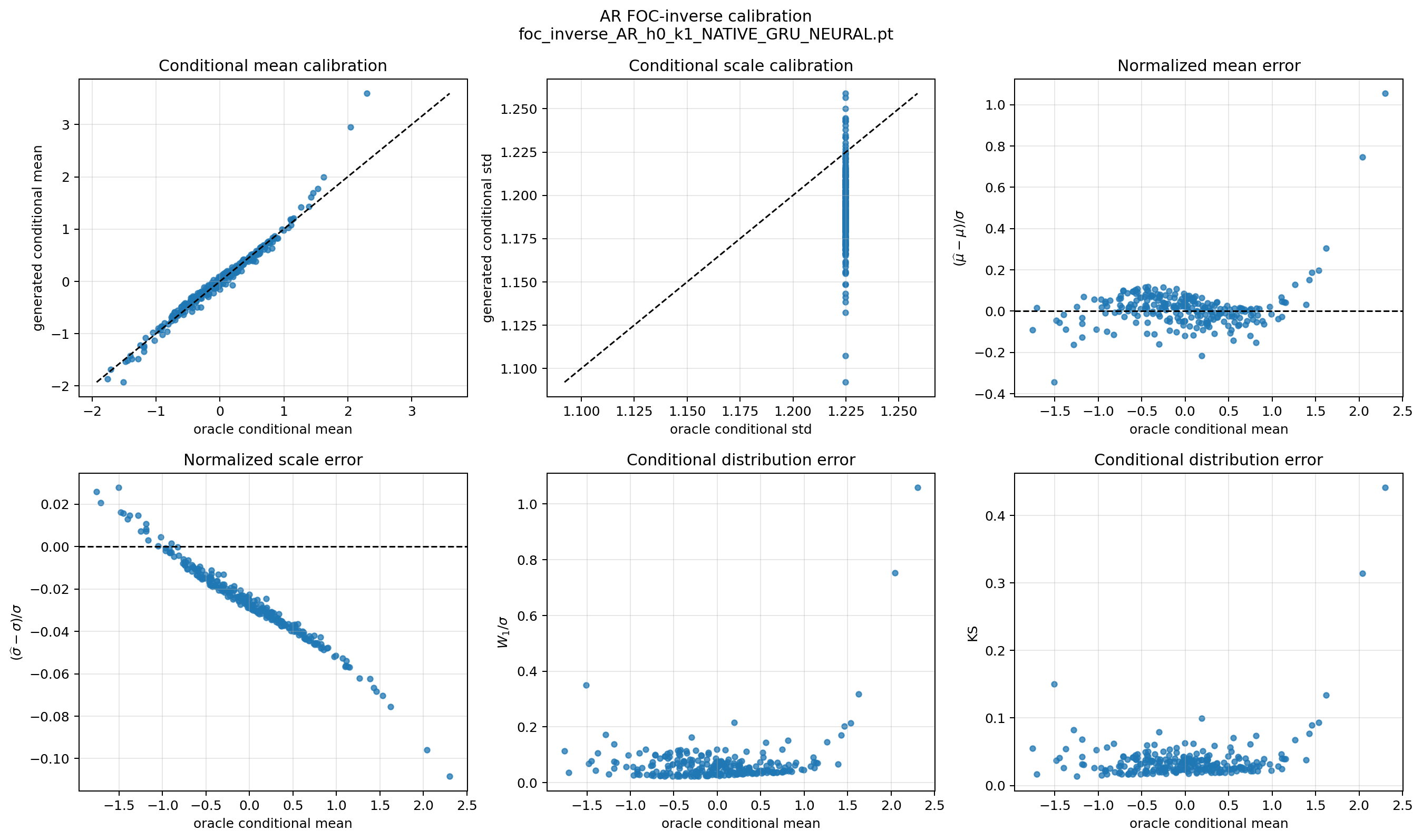}
    \caption[Systematic calibration test for the RBE with AR$(2)$ data]{Systematic calibration tests for the RBE with AR$(2)$ data. Each point represents one of $250$ out-of-sample histories.}
    \label{fig:inverse_calibration_ar}
\end{figure}

\begin{figure}[p]
    \centering
    \includegraphics[width=\textwidth]
    {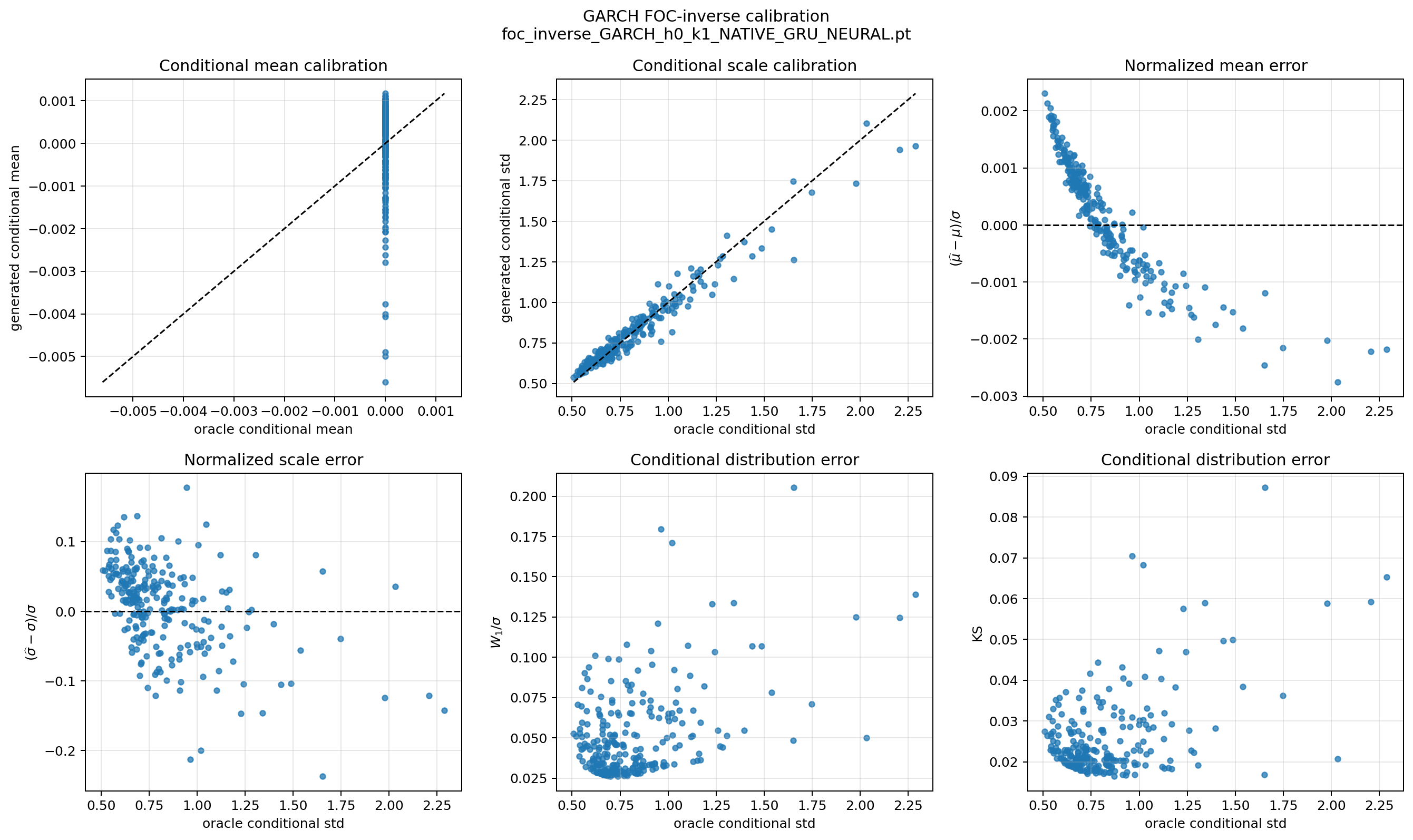}
    \caption[Systematic calibration test for the RBE with GARCH data]{Systematic calibration tests for the RBE with GARCH data. Each point represents one of $250$ out-of-sample histories.}
    \label{fig:inverse_calibration_garch}
\end{figure}

\begin{figure}[p]
    \centering
    \includegraphics[width=\textwidth]
    {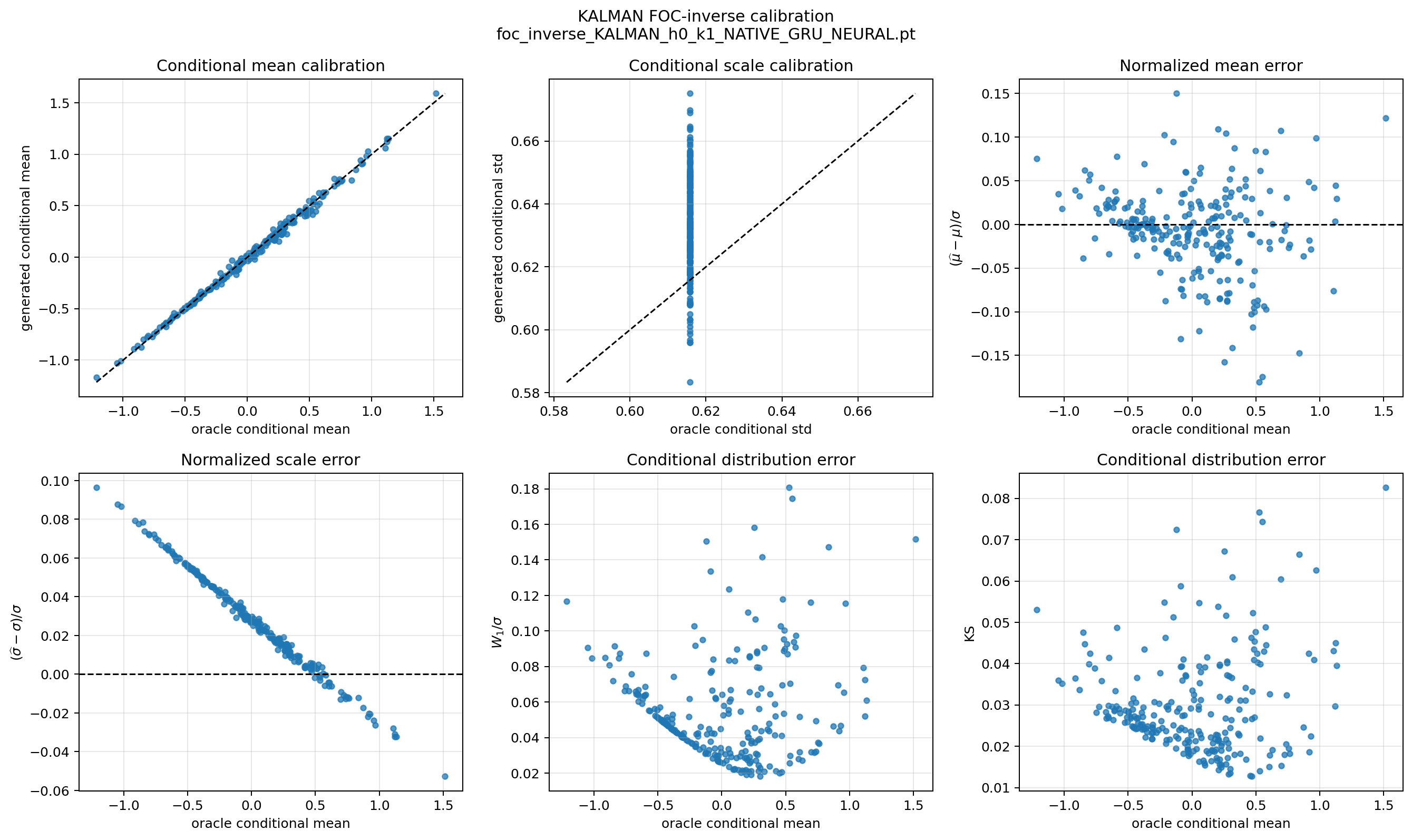}
    \caption[Systematic calibration test for the RBE with Kalman data]{Systematic calibration tests for the RBE with Kalman data. Each point represents one of $250$ out-of-sample histories.}
    \label{fig:inverse_calibration_kalman}
\end{figure}
\begin{figure}[p]
    \centering
    \includegraphics[width=\textwidth]
    {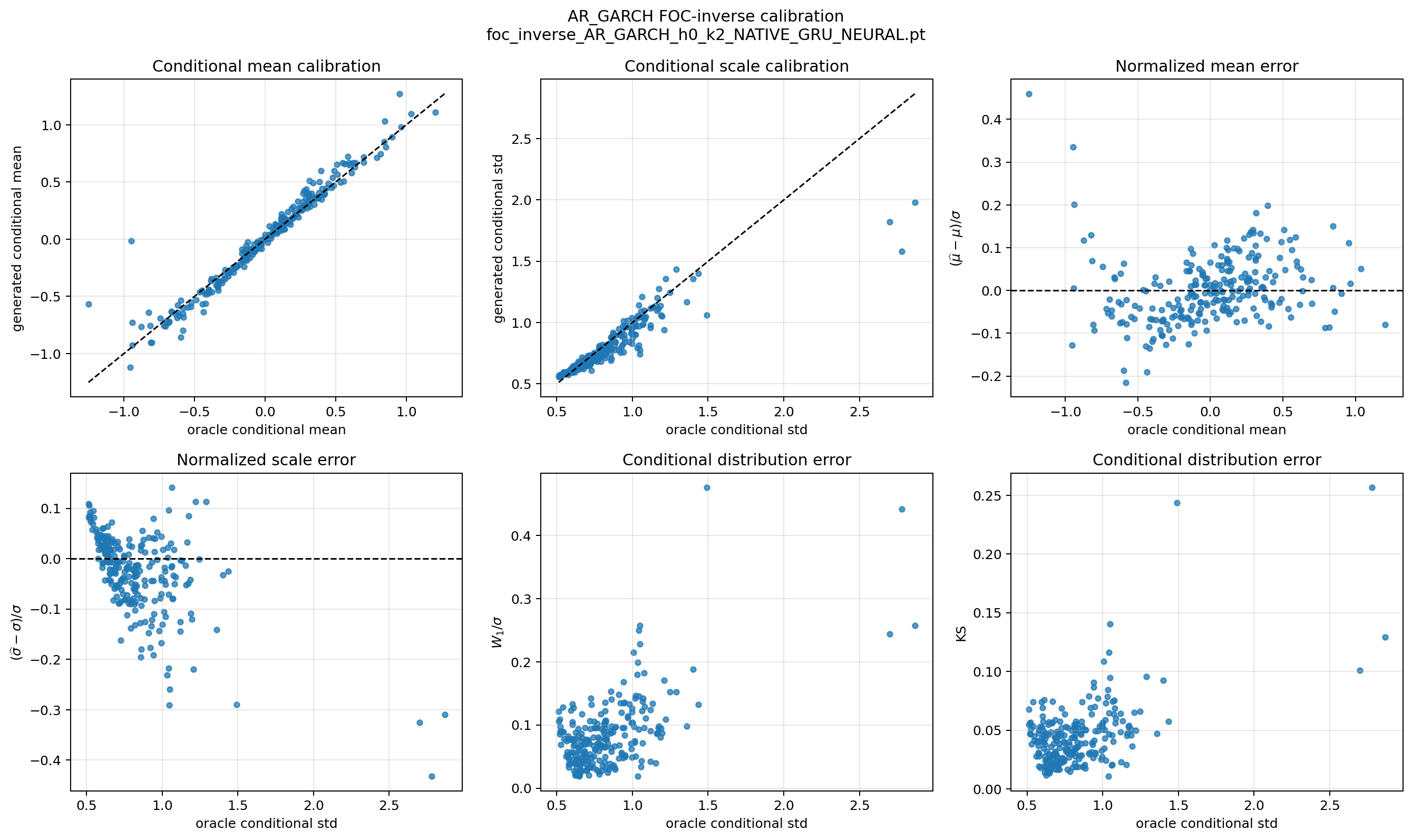}
    \caption[Systematic calibration test for the RBE with AR--GARCH data]{Systematic calibration tests for the RBE with AR--GARCH data. Each point represents one of $250$ out-of-sample histories.}
    \label{fig:inverse_calibration_argarch}
\end{figure}

These results indicate that the learned conditional predictive distribution of the RBE remains well calibrated over a large number of alternative past histories with different conditional means and standard deviations. The main loss of accuracy occurs in rare histories that are near the tails of the learned recurrent representation, particularly for the scale component of AR--GARCH data.

\subsection{Shifting the predictive horizon}

The preceding experiments concern one-step-ahead prediction, but the same construction may be applied at any predictive horizon $h\geq 1$. For a fixed $h$, we can use $x_{t+h}$ as the target variable and the past $x_{\leq t}$ or a variable that encodes it as the covariate, yielding a predictive sampler that looks further ahead into the future. That is, the RBE sampler will be trained to approximate
$$P(x_{t+h}\in\cdot\mid x_{\leq t})$$

Given a new test history $x'_{\leq T}$, we then draw values $y^{(1)},\dots,y^{(M)}$ from the empirical distribution of the learned barycentric variables with horizon $h$ and compute
$$\widehat x_{T+h}^{(j)}=\widehat T_{h,x'_{\leq T}}^{-1}(y^{(j)}),\qquad j=1,\dots,M$$
The empirical distribution of these generated values then approximates the $h$-step conditional distribution $P(x_{T+h}\in\cdot\mid x'_{\leq T})$.

Notice that for $h>1$ the learned barycentric variables need not be iid. Indeed, the variables $y_t$, $y_{t'}$ associated with $|t-t'|<h$ depend on overlapping future intervals and may therefore remain dependent over the first $h-1$ lags. On the other hand, $y_t^{(h)}$ is independent of $\mathcal F_t$ and is measurable with respect to $\mathcal F_{t+h}$, so that $y_t^{(h)}$ and $y_s^{(h)}$ are independent whenever $t-s\geq h$. More importantly for the sampler, the sequence $\{y_t^{(h)}\}$ remains stationary and ergodic.  As a consequence, Tucker's generalization of Glivenko-Cantelli's theorem\cite{Tucker59} ensures that their empirical cdf converges uniformly a.s. to the stationary population cdf, i.e. to the cdf of the barycentric distribution $\mu_h$. More precisely, if
$$\widehat F_{h,T}(a)=\frac{1}{T-h}\sum_{t=1}^{T-h}\mathbf 1_{\{y_t^{(h)}\leq a\}}$$
then 
$$\sup_a|\widehat F_{h,T}(a)-F_h(a)|\longrightarrow 0 \qquad\text{a.s.}$$
 
Thus the serial dependence of the horizon-$h$ barycentric variables should not prevent us from using their empirical distribution for our RBE sampler.

Notice too that for the processes studied so far, since the stationarity conditions imply not only ergodicity but also mixing, the conditional distribution with horizon $h$ given the history $x'_{\leq T}$ should converge to the stationary distribution as $h\to\infty$. In other words, the effect of the particular observed history should therefore diminish rapidly as the predictive horizon increases.

To verify this behavior, we repeated the previous experiment with AR--GARCH data for increasing values of $h$. At each horizon $h$, we generated samples both with the learned RBE sampler and with an oracle sampler that uses the true conditional distribution. Figure \ref{fighorizonhisto} shows the resulting histograms, and Figure \ref{fighorizon} reports the corresponding distance metrics. As $h$ increases, both predictive distributions approach the stationary marginal distribution, and their conditional means and standard deviations approach their stationary values. In fact, this convergence occurs relatively quickly, which is not surprising since under our parameter choice with $|\phi_1|+|\phi_2|<1$ and $\alpha+\beta<1$ the AR--GARCH model is geometrically $\beta$-mixing\cite{meitz08}.

\begin{figure}
    \centering
    \includegraphics[width=0.7\linewidth]{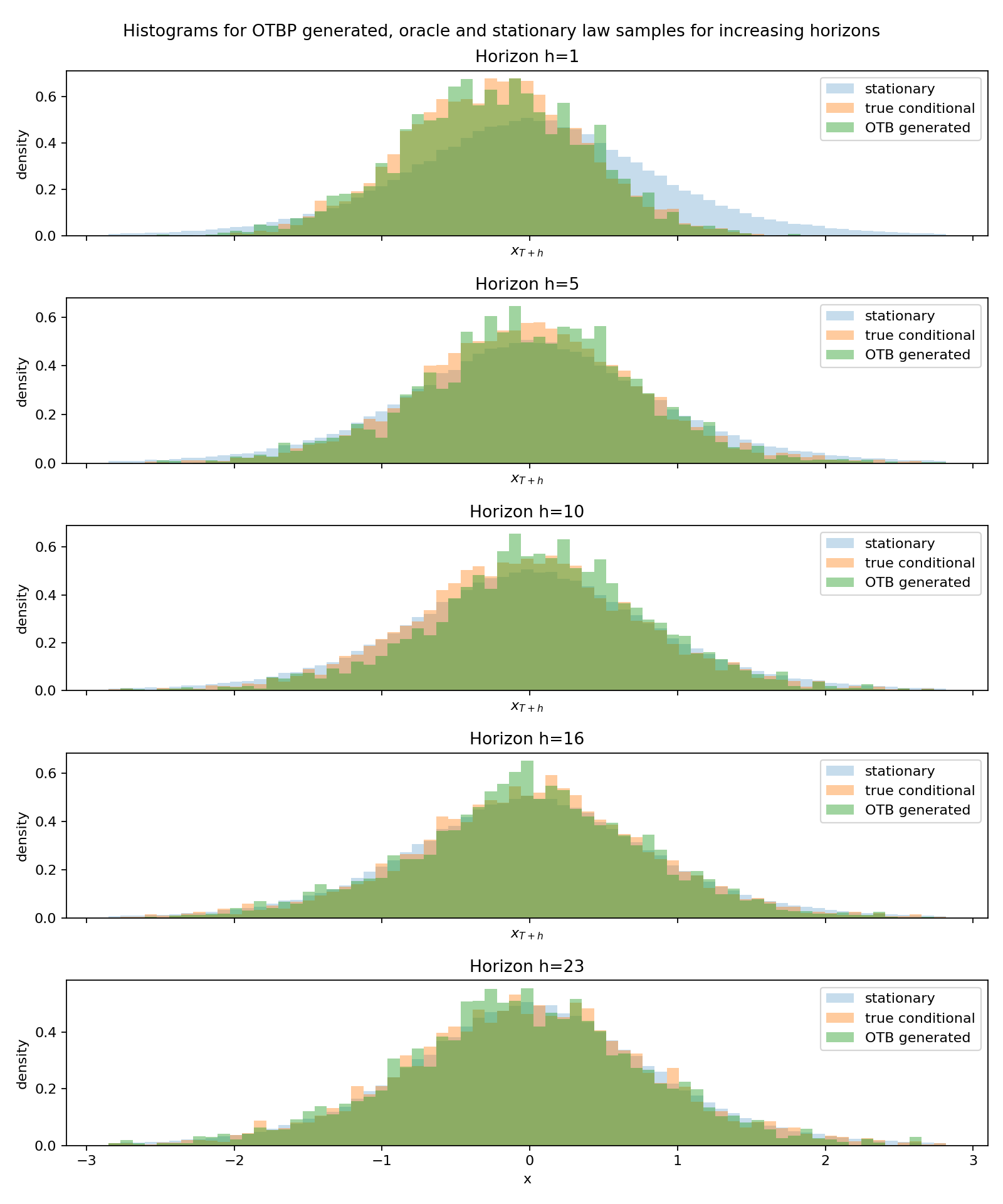}
    \caption{Histograms for sampled $x_{T+h}$ with increasing horizon $h$ for AR--GARCH data.}
    \label{fighorizonhisto}
\end{figure}

\begin{figure}
    \centering
    \includegraphics[width=0.7\linewidth]{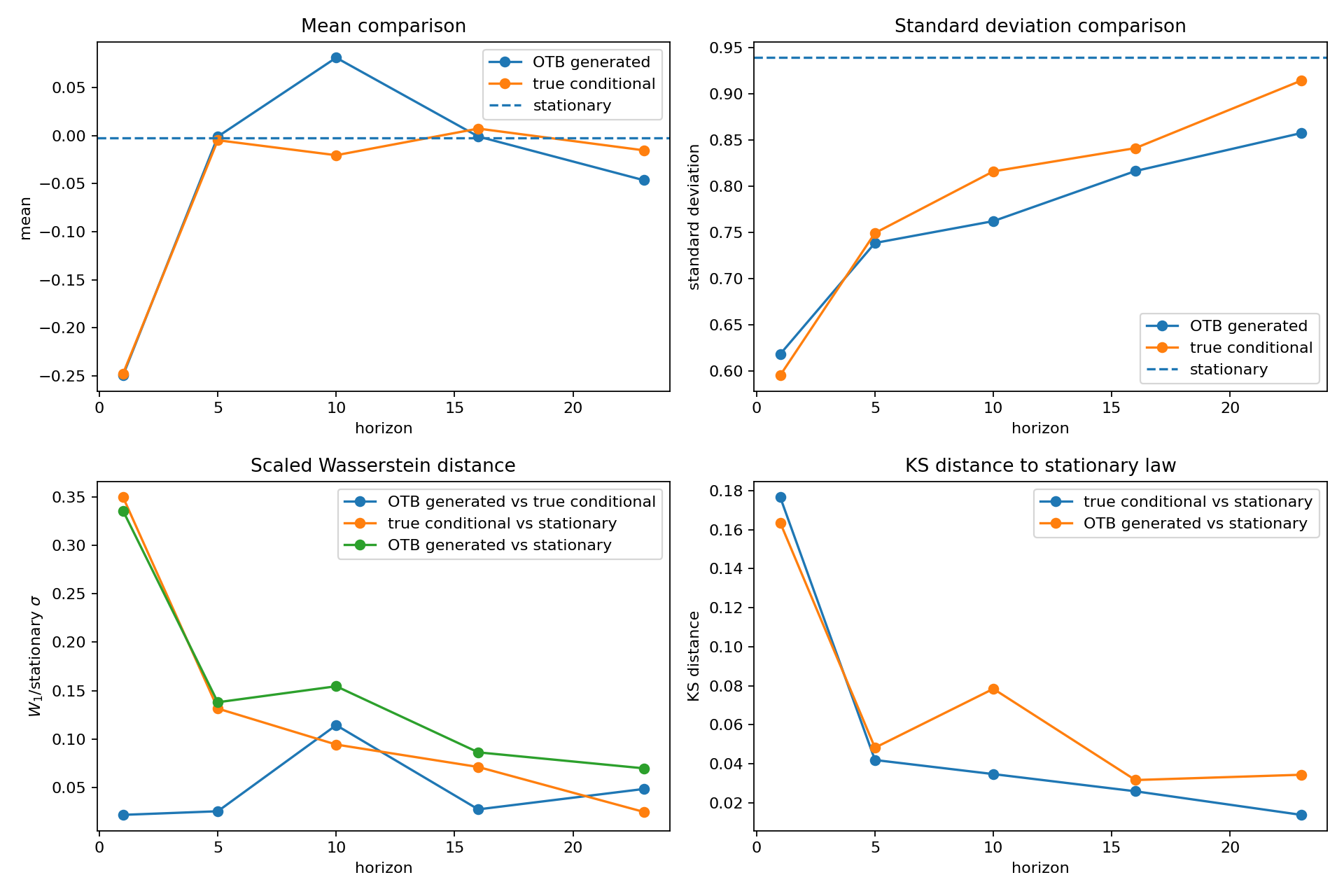}
    \caption{Metrics of convergence to the stationary distribution for the RBE predictive distribution with AR--GARCH data at increasing horizon $h$.}
    \label{fighorizon}
\end{figure}

\section{Application to historical time series}\label{historical}

\subsection{Data and experimental setup}

We finally arrive at the point of applying the RBE sampler to real-world data. For this purpose we chose three historical time series: 4-hour Bitcoin log returns, daily SP500 log returns and the daily North Atlantic Oscillation (NAO) index. We implemented $G_\theta$ as a GRU with linear readout and hidden dimension 4. For the test functions $\varphi$ we considered two choices. The first was to use two monomial channels, i.e. $\varphi^{(1)}(y)=y$ and $\varphi^{(2)}(y)=y^2$. The second had three neural channels. The details are in Appendix \ref{appimplementation}.

Notice that for the choice with two monomial channels the inverse map is
\begin{equation}\label{locscale}
x=\frac{1}{2}\lambda_1 Z_t^{(1)}+\left(1+\lambda_2 Z_t^{(2)}\right)y
\end{equation}
so the methodology is akin to estimating a location-scale model.

We trained five restarts and evaluated both the five resulting RBE samplers and their averaged ensemble, which is defined as
$$\widehat F_t^{\mathrm{ens}}(x)=\frac{1}{5}\sum_{r=1}^5\widehat F_{t,r}(x)$$
where $\widehat F_{t,r}$ is the predictive cdf obtained from restart $r$.

In the diagnostics shown below, the ensemble is used as the primary predictive distribution. Notice though that there is no ensemble barycentric variable, since averaging conditional cdfs does not define a common transport map in any natural way. Diagnostics that involve the barycentric variables are therefore presented separately for each restart.

To evaluate the predictive distributions the dataset was split into training and test samples. For Bitcoin and the NAO index we used a training sample of length $5000$, of which the first $500$ were discarded as a warm-up for the recurrent state, followed by a test sample of $1000$ observations. The SP500 file had only $2511$ observations, so we split it into $2009$ training and $502$ test observations. To prevent any leakage, the data was normalized using the mean and standard deviation from the training block alone. See Appendix \ref{appimplementation} for more details.

\subsection{Predictive diagnostics: CRPS}\label{predictive}
For historical data the true conditional law is unknown, so the 1-Wasserstein and oracle errors used in previous subsections for evaluation are not available. We therefore had to use other metrics to evaluate the predictive RBE distribution. The first diagnostic we computed is the continuous ranked probability score (CRPS), defined as
$$\operatorname{CRPS}(F,x)=\int_{\mathbb R}\left(F(u)-\mathbf 1_{\{x\leq u\}}\right)^2du=\mathbb E|X'-x|-\frac12\mathbb E|X'-X''|$$
where $X'$ and $X''$ are independent random variables with distribution $F$. Lower values indicate better predictive power. For empirical cdfs $\widehat F$ defined by finite samples $\{x_i\}_{i=1}^M$ the CRPS is computed as
$$\operatorname{CRPS}(\widehat F,x)=\frac{1}{M}\sum_{i=1}^{M}\left|x_{i}-x\right|-\frac{1}{2M^2}\sum_{i=1}^{M}\sum_{j=1}^{M}\left|x_i-x_j\right|$$
Notice that CRPS is invariant under translation and linear under positive scaling, so that the CRPS ratios used for comparison below are invariant under normalization.

For comparison we calculated the ratio of the RBE ensemble cdf's mean CRPS
$$\frac{1}{T_{test}}\sum_{t=1}^{T_{test}}\operatorname{CRPS}(\widehat F_t^{\text{ens}},x_t')$$
on the test sample $x'_{1:T_{\text{test}}}$ with the mean CRPS of alternative predictive models on the same sample. The alternative models' distributions were: the unconditional empirical distribution of the training sample, rolling empirical and rolling Gaussian distributions with windows of $250$, $500$ and $1000$ observations, $K$-nearest-neighbor (KNN) conditional distributions, and a Gaussian GARCH model for the two financial series. For NAO we additionally used AR$(p)$ conditional samplers with order $p=1,4,7,20$. In the KNN model with parameters $(p,K)$, the $K$ training histories $H_s=(X_{s-1},\dots,X_{s-p})$ (indexed by the end time $s$) closest to the current vector of $p$ lags $(x'_{t-1},\dots,x'_{t-p})$ form the neighborhood $\mathcal N_{p,K}$, and their corresponding target observations form the predictive empirical distribution

$$\widehat F_{t,p,K}^{\mathrm{KNN}}(x)=\frac{1}{K}\sum_{s\in\mathcal N_{p,K}(t)}\mathbf 1_{\{X_s\leq x\}}$$

Table \ref{tablecrps} summarizes the results from this experiment. The ``best alternative'' column corresponds to the model with the lowest mean CRPS among all the alternatives considered\footnote{The selection is conservative in that it favors the alternatives rather than the RBE.}.

\begin{table}[htbp]
\centering
\small
\resizebox{\textwidth}{!}{%
\begin{tabular}{lccccccc}
\hline
Series
& Monomial RBE
& Neural RBE
& Unconditional
& Best alternative
& Alternative CRPS
& Monomial gain
& Neural gain\\
\hline
Bitcoin
& $0.506175$
& $0.504803$
& $0.513228$
& KNN $(3,250)$
& $0.512080$
& $1.15\%$
& $1.42\%$ \\

SP500
& $0.434881$
& $0.433724$
& $0.446289$
& KNN $(5,100)$
& $0.442959$
& $1.82\%$
& $2.08\%$ \\

NAO
& $0.141409$
& $0.140584$
& $0.639432$
& AR$(4)$
& $0.141430$
& $0.015\%$
& $0.60\%$ \\
\hline
\end{tabular}
}
\caption[Mean out-of-sample CRPS of the monomial and neural RBE ensemble predictive distribution and the main alternatives.]{Mean out-of-sample CRPS of the monomial and neural RBE ensemble predictive distribution and the main alternatives. The gains are $1-\operatorname{CRPS}_{\mathrm{RBE}}/\operatorname{CRPS}_{\mathrm{alternative}}$. For Bitcoin and SP500 the best alternatives were KNN conditional empirical distributions. For NAO it was an AR$(4)$ conditional sampler.}
\label{tablecrps}
\end{table}

Both RBE predictive models have lower mean CRPS than every alternative baseline for all three series: the monomial ensemble has $1.15\%$ against the best alternative for Bitcoin and $1.82\%$ for the SP500. For the neural ensemble the corresponding gains are $1.42\%$ and $2.08\%$. The improvement of the NAO scores relative to the unconditional distribution is much larger, but this is mostly a consequence of the strong persistence of the series. Indeed, the monomial RBE ensemble and the AR$(4)$ model have essentially the same score, differing only by $0.015\%$, while the neural ensemble has a gain of $0.60\%$ against the AR$(4)$ model. Still, these gains are sample mean differences that may arise from sampling variation even when the two models have the same expected CRPS. We therefore constructed confidence intervals for the expected CRPS difference and tested whether the reported gains are statistically significant. Since the differences may have serial dependence, we used a paired circular block bootstrap rather than an iid bootstrap to compare each RBE model with the best alternative. For each $m\in\{\mathrm{mon},\mathrm{neu}\}$, define 
$$
d_t^{(m)}=\operatorname{CRPS}(\widehat F_t^{C},x_t')-
\operatorname{CRPS}(\widehat F_t^{(m)},x_t')
$$
where $\widehat F_t^{C}$ is the predictive cdf of the  alternative model. 

After fixing a block length $l$, for each bootstrap sample we drew $B=\lfloor T_{\text{test}}/l\rfloor$ blocks by choosing a starting position $J$ independently from $\{1,\ldots,T_{\mathrm{test}}\}$ and forming a  block $(d_J,d_{J+1},\ldots,d_{J+\ell-1})$  using circular concatenation, i.e. interpreting indices modulo $T_{\mathrm{test}}$. We used $10000$ bootstrap samples for each block length $l$.

If $q_\alpha$ denotes the empirical $\alpha$-quantile of $\overline d^{*,r}-\overline d$ for $r=1,\ldots,10000$, the $95\%$ one-sided bootstrap confidence interval for $\mathbb{E}[d_t]$ conditional on the selected best-performing alternative\footnote{Indeed, the interval does not account for the fact that the alternative was itself selected using the test sample. We should regard it only as a conditional interval given the best alternative.} is $
\left[\overline d-q_{0.95},\infty\right)$. This interval corresponds to the test
$$
H_0:\mathbb{E}[d_t]\leq0 \qquad\text{ with }\qquad H_1:\mathbb{E}[d_t]>0
$$
at level $5\%$. Thus, a positive lower endpoint for the interval is equivalent to a $p$-value lower than 0.05, where $p$-values are computed as the proportion of bootstrap samples for which
$$
\overline d^{*,r}-\overline d\geq \overline d
$$

Since the $p$-values can depend on the block length, we repeated the calculation for each dataset using different values of $l$ chosen according to some natural scale for each series, that span from short to moderately persistent dependence. For Bitcoin we chose $l=6,12,42$, which correspond to a day, two days and a week. For SP500 we use $l=5,10,20$,  corresponding approximately to one, two and four trading weeks, and for the NAO index we chose $l=7,14,30$, corresponding to one and two weeks and to a month, respectively. .

\begin{table}[htbp]
\centering
\small
\resizebox{\textwidth}{!}{%
\begin{tabular}{llcccc}
\hline
Series
& RBE variant
& Block length $l$
& Mean difference $\overline d$
& $95\%$ lower bound
& $p$-value \\
\hline
Bitcoin & Monomial & $6$  & $0.005905$ & $0.001856$ & $0.0121^{*}$ \\
Bitcoin & Monomial & $12$ & $0.005905$ & $0.001742$ & $0.0146^{*}$ \\
Bitcoin & Monomial & $42$ & $0.005905$ & $0.001450$ & $0.0193^{*}$ \\
Bitcoin & Neural   & $6$  & $0.007277$ & $0.003004$ & $0.0064^{*}$ \\
Bitcoin & Neural   & $12$ & $0.007277$ & $0.002572$ & $0.0093^{*}$ \\
Bitcoin & Neural   & $42$ & $0.007277$ & $0.002059$ & $0.0166^{*}$ \\
\hline
SP500 & Monomial & $5$  & $0.008078$ & $0.002409$ & $0.0104^{*}$ \\
SP500 & Monomial & $10$ & $0.008078$ & $0.003031$ & $0.0061^{*}$ \\
SP500 & Monomial & $20$ & $0.008078$ & $0.003336$ & $0.0028^{*}$ \\
SP500 & Neural   & $5$  & $0.009235$ & $0.002057$ & $0.0222^{*}$ \\
SP500 & Neural   & $10$ & $0.009235$ & $0.003340$ & $0.0078^{*}$ \\
SP500 & Neural   & $20$ & $0.009235$ & $0.003650$ & $0.0048^{*}$ \\
\hline
NAO & Monomial & $7$  & $0.000021$ & $-0.000913$ & $0.4762$ \\
NAO & Monomial & $14$ & $0.000021$ & $-0.000837$ & $0.4801$ \\
NAO & Monomial & $30$ & $0.000021$ & $-0.000855$ & $0.4716$ \\
NAO & Neural   & $7$  & $0.000846$ & $0.000005$ & $0.0495^{*}$ \\
NAO & Neural   & $14$ & $0.000846$ & $0.000106$ & $0.0306^{*}$ \\
NAO & Neural   & $30$ & $0.000846$ & $0.000067$ & $0.0364^{*}$ \\
\hline
\end{tabular}%
}
\caption[Circular block bootstrap comparisons for each RBE ensemble against the best alternative in Table \ref{tablecrps}]
{Circular block bootstrap comparisons for each RBE ensemble against the best alternative in Table \ref{tablecrps}. The mean difference $\overline d$ is the alternative's CRPS minus the RBE CRPS, so a positive value favors the RBE model. An asterisk indicates rejection of the null hypothesis of nonpositive difference. }
\label{tablecrpsbootstrap}
\end{table}

We see that the confidence intervals for all considered block lengths are strictly positive for SP500 and Bitcoin data with both RBE variants. The evidence that both RBE ensembles have lower mean CRPS than their best alternatives is therefore robust to the block length. For NAO, the monomial RBE and AR$(4)$ scores are statistically indistinguishable. The neural RBE has a larger positive difference, with one-sided $p$-values between $0.031$ and $0.049$. This would suggest some borderline evidence in favor of the neural model for this series.

Finally, we did a separate block bootstrap comparing the neural and monomial RBE ensembles. This produced $95\%$ intervals containing zero for all three series and every block length considered, indicating there is no clear evidence that either RBE variant has better predictive accuracy than the other for these series.

The fact that we cannot distinguish statistically the performance of our 3-channel neural RBE model from that of the 2-channel monomial RBE model and the fact that (\ref{locscale}) makes the latter a location-scale model led us to try one final, alternative benchmark: a recurrent location-scale model of the form
$$
X_{t+1}=a_t+b_t\varepsilon_{t+1},\qquad \varepsilon_{t+1}\sim N(0,1)
$$
where a shared GRU with the same hidden dimension as the GRU of the RBE models (i.e. 4) was run through $X_{\leq t}$. A linear readout gave $a_t$, while a second linear readout $c_t$ gave $b_t=10^{-3}+\operatorname{softplus}(c_t)$.

To have a comparable setup with that of the RBE models, we trained five restarts of this model with Gaussian quasi log-likelihood (see Appendix \ref{appimplementation} for more details). Table \ref{tablecrpsrnn} below compares the obtained mean CRPS for the equally-weighted mixture of these restarts against the mean CRPS of the RBE models.

\begin{table}[htbp]
\centering
\small
\begin{tabular}{lccc}
\hline
Series
& Monomial RBE
& Neural RBE
& Location-scale mixture \\
\hline
Bitcoin & $0.506175$ & $0.504803$ & $0.523276$ \\
SP500   & $0.434881$ & $0.433724$ & $0.434915$ \\
NAO     & $0.141409$ & $0.140584$ & $0.139499$ \\
\hline
\end{tabular}
\caption[Mean CRPS comparison of both RBE variants with the recurrent location-scale ensemble]
{Mean CRPS comparison of both RBE variants with the recurrent location-scale ensemble.}
\label{tablecrpsrnn}
\end{table}

As before, we ran block bootstrap comparisons to evaluate the statistical significances of these differences. Since in this case either model managed to outperform the other, we did two-sided tests. Table \ref{tablecrpsrnnbootstrap} shows the results.

\begin{table}[htbp]
\centering
\small
\resizebox{\textwidth}{!}{%
\begin{tabular}{llcccc}
\hline
Series
& RBE variant
& Block length $l$
& Mean difference $\overline d$
& $95\%$ basic interval
& Two-sided $p$-value \\
\hline
Bitcoin & Monomial & $6$  & $0.017101$ & $[0.012606,\,0.021648]$ & $0.0001^{*}$ \\
Bitcoin & Monomial & $12$ & $0.017101$ & $[0.012638,\,0.021836]$ & $0.0001^{*}$ \\
Bitcoin & Monomial & $42$ & $0.017101$ & $[0.012171,\,0.022090]$ & $0.0001^{*}$ \\
Bitcoin & Neural   & $6$  & $0.018473$ & $[0.014129,\,0.022888]$ & $0.0001^{*}$ \\
Bitcoin & Neural   & $12$ & $0.018473$ & $[0.014209,\,0.022785]$ & $0.0001^{*}$ \\
Bitcoin & Neural   & $42$ & $0.018473$ & $[0.013793,\,0.023499]$ & $0.0001^{*}$ \\
\hline
SP500 & Monomial & $5$  & $0.000033$ & $[-0.004130,\,0.003877]$ & $0.9887$ \\
SP500 & Monomial & $10$ & $0.000033$ & $[-0.004225,\,0.003818]$ & $0.9893$ \\
SP500 & Monomial & $20$ & $0.000033$ & $[-0.003894,\,0.003693]$ & $0.9875$ \\
SP500 & Neural   & $5$  & $0.001191$ & $[-0.005744,\,0.006895]$ & $0.6946$ \\
SP500 & Neural   & $10$ & $0.001191$ & $[-0.004191,\,0.005487]$ & $0.6229$ \\
SP500 & Neural   & $20$ & $0.001191$ & $[-0.003745,\,0.005022]$ & $0.5990$ \\
\hline
NAO & Monomial & $7$  & $-0.001910$ & $[-0.003212,\,-0.000268]$ & $0.0144^{*}$ \\
NAO & Monomial & $14$ & $-0.001910$ & $[-0.003237,\,-0.000258]$ & $0.0133^{*}$ \\
NAO & Monomial & $30$ & $-0.001910$ & $[-0.003245,\,-0.000226]$ & $0.0163^{*}$ \\
NAO & Neural   & $7$  & $-0.001085$ & $[-0.002915,\,0.001182]$ & $0.2820$ \\
NAO & Neural   & $14$ & $-0.001085$ & $[-0.002909,\,0.001204]$ & $0.2945$ \\
NAO & Neural   & $30$ & $-0.001085$ & $[-0.002946,\,0.001210]$ & $0.3034$ \\
\hline
\end{tabular}%
}
\caption[Block bootstrap comparison between each RBE ensemble and the recurrent location-scale (RLS) ensemble]
{Block bootstrap comparison between each RBE ensemble and the recurrent location-scale (RLS) ensemble. We used $10000$ bootstrap samples and set
$d_t=\operatorname{CRPS}_{\mathrm{RLS},t}
-\operatorname{CRPS}_{\mathrm{RBE},t}$, so a positive value favors the RBE. The reported
$p$-values are for the two-sided test $H_0:\mathbb{E}[d_t]=0$. An asterisk indicates rejection of the null hypothesis $\mathbb{E}[d_t]=0$.}
\label{tablecrpsrnnbootstrap}
\end{table}

We see that the lower CRPS for both RBE models with Bitcoin data is statistically robust for all considered block lengths. On the other hand, all three models are statistically indistinguishable for SP500 in terms of CRPS. For NAO, the recurrent location-scale mixture outperforms the monomial RBE for NAO, while its difference from the neural RBE is not statistically significant. This indicates our RBE models do not have a detectable advantage over a recurrent location-scale mixture benchmark for SP500 and NAO datasets.

\subsection{Calibration and innovation diagnostics: PIT Kolmogorov-Smirnov and Ljung-Box tests}

The second set of diagnostics we computed are based on the probability integral transform (PIT)
$$U_t^{(\widehat F)}=\widehat F_t^{\text{ens}}(X_t)$$
If $\widehat F_t^{\text{ens}}$ were the true conditional cdf and if it were continuous, then the sequence $U_t=U_t^{(\widehat F)}$ would be iid $\operatorname{Unif}(0,1)$.

To test for independence we applied the Ljung-Box test. The Ljung-Box statistic at maximum lag $m$ has an asymptotic $\chi^2$ distribution under the null hypothesis of no autocorrelation through lag $m$, so that a rejection at $p$-value below $\alpha=0.05$ means a rejection of no autocorrelation at level $0.05$ for this lag. We therefore consider it a `pass' if the null is not rejected. For our tests we used maximum lags $m=5,10,20,40$, with $20$ used as the primary lag.

For restart $r$, let $\widehat y_{t,r}$ be the training barycentric variable and let
$$\widehat y_{t,r}^{\text{test}}=\widehat T_{x_{<t}',r}(x'_t)$$
be the variable obtained by applying the learned transport map to an out-of-sample $x_t'$. Equivalently, 
$$\widehat F_{y,r}(\widehat y_{t,r}^{\text{test}})=\widehat F_{t,r}(x'_t)$$

We applied the Ljung-Box test to the PIT variables $U_t$, the training barycentric variables $\widehat y_{t,r}$, the test sample barycentric variables $\widehat y_{t,r}^{\text{test}}$, and the centered squares of all these sequences. We also applied the test to the residuals of some alternative models such as AR and GARCH. The reason for testing the centered squares is that while the direct autocorrelation test detects linear serial dependence, conditional scale dependence can only be captured by a quadratic test\footnote{For example, a simple GARCH model would pass a linear autocorrelation test, but it clearly has quadratic dependence.}. Table \ref{tableLB} shows the results for lag $m=20$. Only the monomial RBE results are shown, since the neural RBE results are very similar and lead to the same rejection decisions in all cases. Indeed, we use only the monomial RBE values for the analysis and discussion in the next subsection\footnote{Notice too that for the neural variant we didn't compute the test sample barycentric variables $\widehat y_{t,r}^{\text{test}}$, which require a numerical solver to estimate the transport map from the first-order conditions.}.

It should be noted that the training sample $p$-values we use are the usual asymptotic Ljung-Box reference distribution $p$-values. We do not adjust for applying the estimation of the transport map on the training sample used for estimation, since we have no analogue of classical degrees of freedom corrections for fitted residuals of statistical models (such as AR$(p)$)\footnote{It should be possible (but costly) to attempt to derive the corresponding post-estimation null distribution}. In any case, our conclusions will not depend on a borderline training sample $p$-value. For out-of-sample tests no such post-estimation correction is needed.

We next tested that the marginal distribution is uniform. For this we considered the empirical probability integral transform distribution
$$\widehat G_n(u)=\frac{1}{n}\sum_{t=1}^n\mathbf 1_{\{U_t\leq u\}}$$
the Kolmogorov-Smirnov statistic is
$$D_n=\sup_{u\in[0,1]}|\widehat G_n(u)-u|$$
Under the null hypothesis that the $U_t$ are iid and uniform, the distribution of $\sqrt{n}D_n$ converges to the Kolmogorov distribution. In our setting the $U_t$'s arise by applying the PIT to the discrete RBE predictive distribution. This is a discrete distribution obtained from the learned barycentric variables $\widehat y_t$ and the learned transport map $\widehat T$, so we know it is not strictly speaking a uniform random variable, since it is not even continuous. We can nonetheless use the KS test as a useful marginal calibration diagnostic. Another important caveat is that the iid assumption is precisely what is being tested in the other set of diagnostics, using the Ljung-Box test. When serial independence fails the statistic $D_n$ is still a meaningful measure of marginal calibration, since $D_n\to 0$ almost surely if the $(U_t)$ are stationary and ergodic with uniform marginal. However, if the $U_t$ are no longer iid then the Kolmogorov distribution is no longer the limit distribution of $\sqrt{n}D_n$. 

In addition to the KS distance, sample mean and standard deviation of the PIT reported in Table \ref{tablepit}, we computed the empirical coverage of central predictive intervals. For a continuous predictive cdf and any level $\alpha\in[0,1]$
$$C_\alpha=P\left(\frac{1-\alpha}{2}\leq U_t\leq\frac{1+\alpha}{2}\right)=P\left(q_{t,(1-\alpha)/2}\leq x'_t\leq q_{t,(1+\alpha)/2}\right)=\alpha$$
where $q_{t,\beta}$ is the quantile of level $\beta$ of that predictive distribution. For our RBE learned distribution we have
$$\widehat q_{t,\beta}=\inf\{x:\widehat F_t^{\mathrm{ens}}(x)\geq\beta\}$$
Its level $\alpha$ central interval coverage is
$$\widehat C_\alpha=\frac1{T_{\text{test}}}\sum_{t=1}^{T_{\text{test}}}\mathbf 1_{\{\widehat q_{t,(1-\alpha)/2}\leq x'_t\leq\widehat q_{t,(1+\alpha)/2}\}}$$
and, for a correctly calibrated predictive distribution, it should be close to $C_\alpha=\alpha$. We therefore computed it for $\alpha=0.25,0.50,0.75,0.90$ and $0.98$. Results are shown in Table \ref{tablecoverage}. We see that the RBE $\widehat C_\alpha$'s are closer to the true value than alternatives like GARCH and AR$(4)$.

Notice, however, that central coverage does not detect every departure from uniformity, since it only measures the sum of the two tail frequencies. A $90\%$ interval could contain $90\%$ of the observations while $8\%$ fall below it and only $2\%$ above it, instead of the correct proportions $5\%$ and $5\%$. The central coverage would then be correct even though the predictive distribution is displaced. Indeed, we shall see that the RBE predictive distribution for Bitcoin has almost exact central coverage while its PIT mean ($0.4751$) is below $0.5$.

\begin{table}[htbp]
\centering
\scriptsize
\resizebox{\textwidth}{!}{%
\begin{tabular}{llcccc}
\hline
Series
& Method
& Mean $U$
& $\operatorname{sd}(U)$
& $D_{\mathrm{KS}}$
& $p_{\mathrm{KS}}$ \\
\hline
Bitcoin
& RBE ensemble
& $0.47508$
& $0.28817$
& $0.04907$
& $0.01567^{*}$ \\

Bitcoin
& Empirical unconditional cdf
& $0.47551$
& $0.29456$
& $0.05900$
& $0.00181^{*}$ \\
\hline
SP500
& RBE ensemble
& $0.50451$
& $0.28033$
& $0.04195$
& $0.33098$ \\
\hline
NAO
& RBE ensemble
& $0.50843$
& $0.29242$
& $0.03407$
& $0.19174$ \\

NAO
& AR$(4)$
& $0.50918$
& $0.29257$
& $0.03615$
& $0.14289$ \\
\hline
\end{tabular}
}
\caption[Diagnostics for the out-of-sample PIT variables.]
{Diagnostics for the out-of-sample PIT variables. The KS column contains the  $p$-values for the Kolmogorov-Smirnov test. An asterisk indicates rejection of the uniformity null hypothesis at level $0.05$. RBE values correspond to the monomial option. The empirical unconditional row for Bitcoin uses the empirical cdf of the whole training sample, while the AR$(4)$ row for NAO uses the empirical cdf of its training residuals. Notice the standard deviation of a uniform random variable is $1/\sqrt{12}\simeq0.28868$. }
\label{tablepit}
\end{table}

\begin{table}[htbp]
\centering
\begin{tabular}{llccccc}
\hline
Series & Method
& $25\%$ & $50\%$ & $75\%$ & $90\%$ & $98\%$ \\
\hline
Bitcoin & RBE ensemble
& $24.0\%$ & $50.1\%$ & $74.7\%$ & $91.1\%$ & $98.5\%$ \\

Bitcoin & Unconditional
& $22.8\%$ & $47.5\%$ & $71.6\%$ & $89.7\%$ & $98.2\%$ \\

Bitcoin & GARCH
& $35.6\%$ & $62.1\%$ & $81.7\%$ & $91.5\%$ & $96.3\%$ \\

SP500 & RBE ensemble
& $24.9\%$ & $52.0\%$ & $77.7\%$ & $91.6\%$ & $97.4\%$ \\

NAO & RBE ensemble
& $23.9\%$ & $49.7\%$ & $74.3\%$ & $88.4\%$ & $97.8\%$ \\

NAO & AR$(4)$
& $23.4\%$ & $48.7\%$ & $74.8\%$ & $88.5\%$ & $97.8\%$ \\
\hline
\end{tabular}
\caption[Empirical coverage $\widehat C_\alpha$ of central intervals of different predictive distributions.]{Empirical coverage $\widehat C_\alpha$ of central intervals of different predictive distributions. `Unconditional' refers to the empirical distribution of the training data.}
\label{tablecoverage}
\end{table}

Finally, Figure \ref{fig:historical_pit} shows the empirical PIT cdf and the PIT variables autocorrelations up to lag $40$ for all three times series.

\begin{figure}[htbp]
    \centering
    \includegraphics[width=\textwidth]{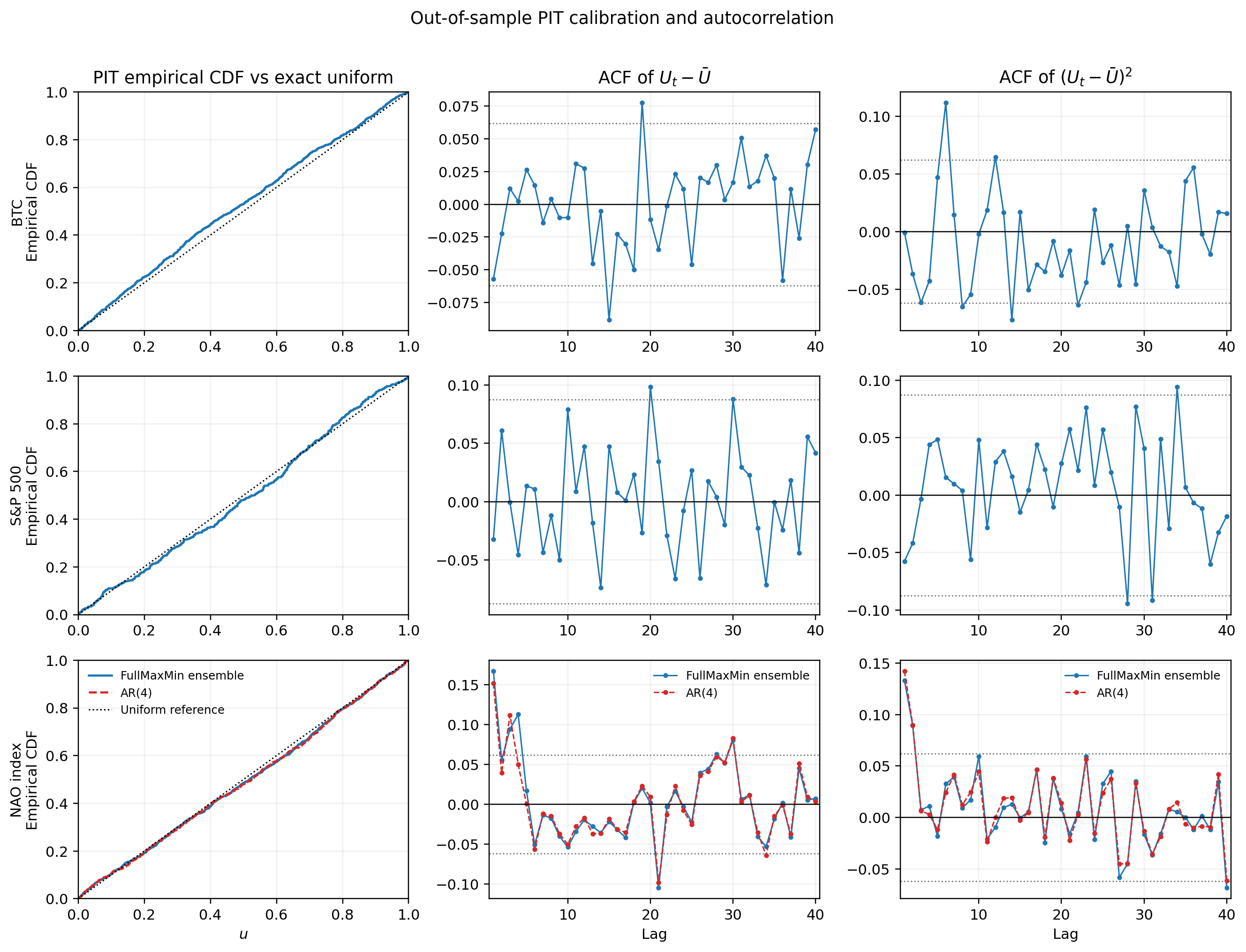}
    \caption[Out-of-sample PIT calibration and serial dependence]{Out-of-sample PIT calibration and serial dependence. The left column shows the empirical cdf of the RBE PIT $U_t$ against the uniform cdf. The center and right columns show the autocorrelation function (ACF) of $U_t-\bar U$ and $(U_t-\bar U)^2$, respectively. The horizontal dotted lines are the pointwise bounds $\pm1.96/\sqrt n$. For NAO, the dashed red curve corresponds to the AR$(4)$ model.}
    \label{fig:historical_pit}
\end{figure}

\subsection{Evaluation of the calibration diagnostics}

Of all three datasets, the SP500 gives the clearest positive result. Its PIT variables pass the KS test for the uniform distribution with $p=0.3310$, and together with the training and out-of-sample barycentric variables they also pass both the linear and the squared Ljung-Box tests for all five restarts, indicating that the extracted barycentric variables have no detectable dependence. 

For Bitcoin, the standard deviation of the PIT is almost exactly the uniform value, and its central empirical distribution is close to uniform, but its mean at $0.4751$ is significantly below the expected $1/2$, which corresponds to an upward bias in the predictive distribution\footnote{To see this, notice that if $F_0$ is the training distribution and $G$ is the test distribution, then for the unconditional PIT $U=F_0(X)$,
$$\mathbb E[U]= P(Y\leq X),\qquad Y\sim F_0,\quad X\sim G$$
with $X$ and $Y$ independent. In other words, $\mathbb E[U]=0.4755$ means a test observation is larger than an independent training observation $47.55\%$ of the time, instead of the expected $50\%$.}. In fact, the KS test rejects uniformity of the marginal distribution with $p=0.0157$ (see Table \ref{tablepit}). It must be noted, however, that the same  upward bias affects the unconditional distribution, whose PIT mean is $0.4755$ and which also fails the KS test, indicating that the discrepancy was already present between the training and test blocks in the input data and was not created by the RBE. A natural explanation is that there is a structural change in the underlying law after the train/test cut. Nevertheless, our metrics cannot distinguish such a nonstationary change in the distribution from a transition between persistent latent regimes within a stationary law that the train/test samples do not capture because persistence makes the effective sample size much smaller than $n=4500$.

Regarding serial dependence for Bitcoin, at first sight, the training sample Ljung-Box results appear somewhat less satisfactory than the out-of-sample results. While all RBE restarts pass both out-of-sample Ljung-Box tests at maximum lag $m=20$, four of the five training sequences reject the Ljung-Box no autocorrelation null for $\widehat y_t$, and five out of five also reject it for the centered squares. To put these results in perspective, we compared with the standardized residuals of the alternative GARCH model
$$\widehat \epsilon_t=\frac{x_t- \widehat \mu}{\widehat\sigma_t}$$

These pass the Ljung-Box test, but their centered squares, like the RBE restarts, reject. Given that several of the statistics are close to the $5\%$ critical value, we could try to explain these patterns by looking not only at the pass/reject decision, but also at the Ljung-Box statistics themselves, and comparing them with the Ljung-Box statistics of the original Bitcoin returns. These are shown in Table \ref{tableLB}, 

\begin{table}[htbp]
\centering
\scriptsize
\resizebox{\textwidth}{!}{%
\begin{tabular}{lllcccc}
\hline
Series
& Sample
& Sequence $R_t$
& $Q_{20}(R)$
& $p_{\mathrm{LB}}(R)$
& $Q_{20}((R-\bar R)^2)$
& $p_{\mathrm{LB}}((R-\bar R)^2)$ \\
\hline
Bitcoin
& Training
& Original returns
& $54.42$
& ${5.01\times10^{-5}}^{*}$
& $357.05$
& ${1.57\times10^{-63}}^{*}$ \\

Bitcoin
& Training
& RBE's $\widehat y_t$
& $[30.85,38.85]$
& $[0.00696^{*},0.05715]$
& $[34.29,39.24]$
& $[0.00622^{*},0.02424^{*}]$ \\

Bitcoin
& Training
& GARCH residuals
& $30.81$
& $0.05776$
& $33.23$
& $0.03185^{*}$ \\

Bitcoin
& Test
& Original returns
& $21.28$
& $0.38096$
& $155.11$
& ${6.58\times10^{-23}}^{*}$ \\

Bitcoin
& Test
& RBE's $\widehat y_t$
& $[23.02,26.24]$
& $[0.15793,0.28759]$
& $[16.91,27.47]$
& $[0.12245,0.65903]$ \\

Bitcoin
& Test
& GARCH residuals
& $23.06$
& $0.28609$
& $19.21$
& $0.50852$ \\

Bitcoin
& Test
& RBE ensemble PIT
& $27.30$
& $0.12699$
& $46.56$
& $0.00067^{*}$ \\
\hline
SP500
& Training
& RBE's $\widehat y_t$
& $[21.00,24.87]$
& $[0.20662,0.39742]$
& $[11.67,19.10]$
& $[0.51525,0.92684]$ \\

SP500
& Test
& RBE's $\widehat y_t$
& $[10.24,12.72]$
& $[0.88920,0.96356]$
& $[11.68,26.06]$
& $[0.16386,0.92667]$ \\

SP500
& Test
& RBE ensemble PIT
& $20.20$
& $0.44540$
& $11.32$
& $0.93748$ \\
\hline
NAO
& Training
& RBE's $\widehat y_t$
& $[17.17,30.35]$
& $[0.06435,0.64221]$
& $[614.58,889.82]$
& $[{1.2\times10^{-175}}^{*},{2.4\times10^{-117}}^{*}]$ \\

NAO
& Test
& RBE's $\widehat y_t$
& $[37.29,47.12]$
& $[0.0006^{*},0.0108^{*}]$
& $[43.81,92.80]$
& $[{2.40\times10^{-11}}^{*},0.00160^{*}]$ \\

NAO
& Test
& RBE ensemble PIT
& $68.32$
& ${3.42\times10^{-7}}^{*}$
& $38.15$
& $0.00849^{*}$ \\

NAO
& Train
& AR$(4)$ residuals
& $26.681$
& $0.05869$
& $824.402$
& ${9.3132\times10^{-162}}^{*}$ \\

NAO
& Test
& AR$(4)$ residuals
& $33.80$
& $0.0275^{*}$
& $74.0067$
& ${3.9859\times10^{-8}}^{*}$ \\

NAO
& Test
& AR$(4)$ PIT
& $54.50$
& ${4.88\times10^{-5}}^{*}$
& $39.33$
& $0.00606^{*}$ \\

\hline
\end{tabular}
}
\caption[Ljung-Box statistics with $m=20$ for different sequences $R_t$ and their centered squares]
{Ljung-Box statistics with $m=20$ for different sequences $R_t$ and their centered squares. An asterisk indicates rejection of the no autocorrelation null hypothesis at level $0.05$. For RBE, the intervals show the range of values for the five restarts of the monomial variant. The endpoints of each $p$-value interval are marked separately, so that an interval with only its lower endpoint marked contains both rejecting restarts and restarts that pass. }
\label{tableLB}
\end{table}

The original returns already contain significant linear dependence in the training period, with $Q_{20}=54.42$ and $p=5.01\times10^{-5}$. The five RBE restarts reduced this to $Q_{20}\in[30.85,38.85]$. Thus, all restarts remove a significant part of the original linear dependence. For GARCH the corresponding statistic is $30.81$, which is very close to the $\chi_{20}^{2,5\%}$ critical value $31.41$. Notice that the best RBE value, $30.85$, is almost identical to the GARCH value. It would therefore be a mistake to interpret the fact that GARCH and one RBE restart fall immediately below the critical value while the other four RBE restarts fall above it (three of which do so with a $p$-value above $2\%$) as evidence of a qualitative difference between the two methods. On the contrary, their similarity becomes more apparent when we compare the scale component dependence captured by the centered squares. The centered squares Ljung-Box statistic of the original training returns is $Q_{20}=357.05$, whereas the RBE values are in the range $34.29-39.24$. Gaussian GARCH gives $Q_{20}=33.23$. This indicates that both methods remove most of the dependence in squared returns, although the much longer training sample still detects a small residual dependence, both for GARCH and for the RBE. We may then say that the RBE makes the centered squares Ljung-Box statistic fall by approximately an order of magnitude, quite similar to what a model designed specifically for conditional scale does.

This interpretation is reinforced by the out-of-sample results. Indeed, the test sample original returns pass the Ljung-Box test ($Q_{20}=21.28$ with $p=0.3810$), so it is not surprising the RBE barycentric variables pass for all five restarts. In contrast, the original squared returns do reject, with $Q_{20}=155.11$ and $p$-value $6.58\times10^{-23}$, while the squared centered barycentric variables learned by our method have $Q_{20}$ statistics  between $16.91$ and $27.47$, none of which reject. GARCH, for its part, gives $Q_{20}=19.21$. We can therefore say that the RBE sampler removes the detectable out-of-sample scale dependence as effectively as GARCH (all while achieving a lower CRPS, as we saw in the previous section).

These diagnostic discrepancies compel us to identify the sources for the different training and out-of-sample results with Bitcoin data. We find two. First, the training sample has size $n=4500$, whereas the testing sample has size $n=1000$. This is an important difference, because the Ljung-Box statistic is
$$Q_m=n(n+2)\sum_{k=1}^{m}\frac{\widehat\rho_k^2}{n-k}$$
so that a fixed sequence of autocorrelations produces a statistic that grows approximately linearly with sample size $n$. Thus, $n=4500$ will have a scale of approximately $0.0292$, compared to $0.0620$ for $n=1000$. Hence, residual autocorrelations that are detected in the training sequence may remain undetected in a test sample less than $25\%$ its size.

The second source is that the autocorrelation structure of the original returns seems to change through time. To evaluate this hypothesis we divided the training sample into four disjoint blocks of length $1000$ and applied Ljung-Box diagnostics to them. Results are shown in Table \ref{tableBlocks}. The original returns pass in the first block, with $p=0.0722$, but reject in the next three, with $p$-values $0.00070$, $0.00207$ and $0.01741$. The out-of-sample returns again do not reject, with $p=0.3810$. On the other hand, all five RBE restarts pass on the first, third and fourth training block. Thus, the discrepancy between the training and test Ljung-Box results is likely a reflection both of the greater power in the larger training sample and also of some non-homogeinity or variability across time in the autocorrelation structure.

\begin{table}[htbp]
\centering
\scriptsize
\resizebox{\textwidth}{!}{%
\begin{tabular}{clcccc}
\hline
Block
& Sequence $R_t$
& $Q_{20}(R)$
& $p_{\mathrm{LB}}(R)$
& $Q_{20}((R-\bar R)^2)$
& $p_{\mathrm{LB}}((R-\bar R)^2)$ \\
\hline
1
& Original returns
& $29.86$
& $0.07222$
& $107.38$
& $5.89\times10^{-14}$ \\

1
& RBE $\widehat y_t$
& $[26.00,28.13]$
& $[0.10630,0.16589]$
& $[21.71,26.26]$
& $[0.15726,0.35669]$ \\

1
& GARCH residuals
& $26.00$
& $0.16582$
& $24.32$
& $0.22877$ \\
\hline
2
& Original returns
& $46.45$
& $0.00070$
& $46.32$
& $0.00073$ \\

2
& RBE $\widehat y_t$
& $[32.67,39.20]$
& $[0.00630,0.03665]$
& $[8.51,14.10]$
& $[0.82549,0.98796]$ \\

2
& GARCH residuals
& $29.98$
& $0.07013$
& $9.43$
& $0.97738$ \\
\hline
3
& Original returns
& $42.97$
& $0.00207$
& $226.28$
& $6.65\times10^{-37}$ \\

3
& RBE $\widehat y_t$
& $[23.57,29.56]$
& $[0.07734,0.26182]$
& $[21.36,31.22]$
& $[0.05239,0.37616]$ \\

3
& GARCH residuals
& $28.11$
& $0.10672$
& $35.71$
& $0.01665$ \\
\hline
4
& Original returns
& $35.54$
& $0.01741$
& $32.74$
& $0.03605$ \\

4
& RBE $\widehat y_t$
& $[15.31,24.28]$
& $[0.23026,0.75865]$
& $[19.31,24.81]$
& $[0.20899,0.50159]$ \\

4
& GARCH residuals
& $26.17$
& $0.16013$
& $21.33$
& $0.37789$ \\
\hline
\end{tabular}
}
\caption[Ljung-Box statistics in Bitcoin training split blocks]
{Ljung-Box statistics in Bitcoin training split blocks. Se used maximum lag $m=20$ for four disjoint training blocks of length $1000$ (the first $500$ observations of the original block were omitted). The intervals for the RBE represent the range for the five restart values.}
\label{tableBlocks}
\end{table}

More insight about this remaining source of dependence is provided by looking at how the Ljung-Box decision changes with the maximum lag $m$. The linear test for the original returns does not reject at $m=5$ ($p$ is $0.1156$), but rejects at $m=10$ with $p=0.00133$ (see Table \ref{tablelags}). Moreover, all five RBE training centered-square sequences pass at $m=5$, with $p$-values between $0.4029$ and $0.6949$, whereas all five reject at $m=10$, with $p$-values between $0.0054$ and $0.0273$. The corresponding GARCH squared-residual $p$-value decreases from $0.2553$ at $m=5$ to $0.0719$ at $m=10$, $0.0319$ at $m=20$ and $0.0010$ at $m=40$.

There is a natural tentative explanation for this. The dataset contains six four-hour observations per day, so lag $6$ is the first lag corresponding to a complete daily cycle. Its exclusion at $m=5$ and inclusion at $m=10$ suggests that the remaining dependence might contain a periodic component. To test this hypothesis we plotted the autocorrelation functions up to lag $84$, which corresponds to two weeks. The results are shown in Figure \ref{fig:btc_innovation_acf}.

\begin{table}[htbp]
\centering
\scriptsize
\resizebox{\textwidth}{!}{%
\begin{tabular}{lcccc}
\hline
Sequence 
& $m=5$
& $m=10$
& $m=20$
& $m=40$ \\
\hline
Original returns

& $0.11557$
& $0.00133^{*}$
& $(5.01\times10^{-5})^{*}$
& $0.00241^{*}$ \\

Squared original returns

& $(1.77\times10^{-44})^{*}$
& $(1.02\times10^{-55})^{*}$
& $(1.57\times10^{-63})^{*}$
& $(2.97\times10^{-68})^{*}$ \\
\hline
RBE learned $\widehat y_t$

& $[0.26378,0.63427]$
& $[0.02832^{*},0.17259]$
& $[0.00696^{*},0.05715]$
& $[0.04862^{*},0.28194]$ \\

Squared RBE $\widehat y_t^2$

& $[0.40287,0.69495]$
& $[0.00541^{*},0.02730^{*}]$
& $[0.00622^{*},0.02424^{*}]$
& $[0.00199^{*},0.00786^{*}]$ \\
\hline
GARCH residuals

& $0.57297$
& $0.17802$
& $0.05776$
& $0.14311$ \\

Squared GARCH residuals

& $0.25531$
& $0.07191$
& $0.03185^{*}$
& $0.00097^{*}$ \\
\hline
\end{tabular}
}
\caption[Ljung-Box $p$-values for the Bitcoin training sample at maximum lags $m=5,10,20,40$]
{Ljung-Box $p$-values for the Bitcoin training sample at maximum lags $m=5,10,20,40$. An asterisk indicates rejection of the null hypothesis of no autocorrelation at level $0.05$. The intervals for RBE are the ranges across the five restarts.}
\label{tablelags}
\end{table}

\begin{figure}[htbp]
    \centering
    \includegraphics[width=\textwidth]
    {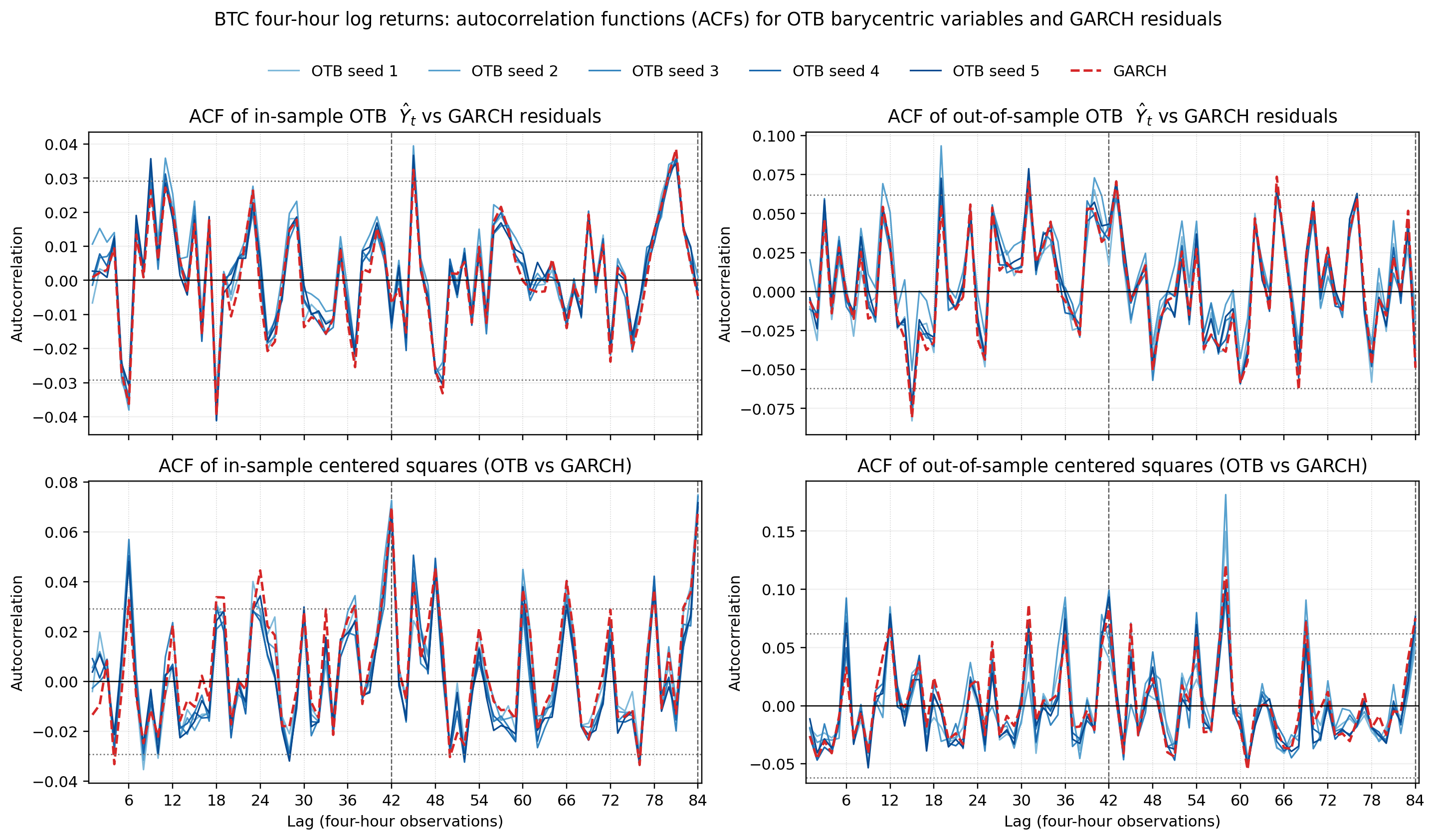}
    \caption[Autocorrelations of Bitcoin RBE and GARCH innovations]
    {Autocorrelations of Bitcoin RBE and GARCH innovations. Sample autocorrelations of the five RBE barycentric sequences, the GARCH standardized residuals and their centered squares. The plots on the left show the training sequence and those on the right use the out-of-sample sequence of length $1000$. Vertical grid lines mark a day (six 4-hour lags). Dashed lines mark lags $42$ and $84$, corresponding to one and two weeks. The horizontal dotted lines are the pointwise bounds $\pm1.96/\sqrt n$.}
    \label{fig:btc_innovation_acf}
\end{figure}

The training sample centered-square plots support our interpretation. There are positive autocorrelations at several multiples of $6$, including a noticeable peak at lag $6$ and particularly large peaks at lags $42$ and $84$, corresponding to one and two weeks, respectively. These peaks appear in both the RBE and GARCH innovations, which show remarkably similar autocorrelation patterns\footnote{The strong similarity should not make us conclude that the two methods estimate the same conditional distribution. Their interval behavior differs, and GARCH has mean CRPS $0.513909$, compared with $0.506175$ for the RBE ensemble.}. It is clear then that part of the remaining dependence is a feature of the data not extracted by either method.

A simple mechanism for this periodic pattern is the following. Suppose that, after conditional mean is removed, the process has a residual innovation of the form $R_t=s_t\epsilon_t$, where the $\epsilon_t$ are iid with unit variance but $s_t$ depends on time of day or day of the week and is hence periodic. The expected sample autocovariance of the centered squares is then approximately 
$$\frac{1}{n-k}\sum_{t=1}^{n-k}
 (s_t^2-\overline{s^2})(s_{t+k}^2-\overline{s^2})$$
If $s_{t+6}=s_t$ or $s_{t+42}=s_t$, corresponding to dependence on time of the day or day of the week, then the expression would be positive at $k=6$ or $k=42$.

In any event, the plots suggest a clear periodic component that was not removed, at least not significantly, by either the RBE or the GARCH fitting. This shouldn't be surprising, since both the RBE and the GARCH model lack clock or calendar covariates.  

For NAO the PIT passes the uniform KS test, but both $U_t$ and $(U_t-\bar U)^2$ reject the Ljung-Box null. The training barycentric variables do not show linear autocorrelation, but their centered squares reject for all restarts. Additionally, both the out of sample barycentric variables $\widehat y_{t,r}^{\text{test}}$ and their centered squares reject for all restarts. This might seem like a negative result at first, but it is informative to make a comparison with the AR$(4)$ model. First, the AR$(4)$ model's mean CRPS was statistically indistinguishable from the monomial RBE ensemble: $0.141430$ against $0.141409$ for the RBE. Their PIT means, standard deviations, KS distances and average $90\%$ interval widths are also almost identical. It seems that with NAO data the monomial RBE learned essentially the same conditional distribution as the AR$(4)$ model. Indeed, just as with the learned barycentric variables, the AR$(4)$ training residuals do not reject linear autocorrelation at lag $20$ ($p=0.0587$), but their squares reject with $p<10^{-150}$, and the out of sample residuals and their squares reject too, with $p=0.0275$ and $p=3.99\times10^{-8}$, respectively (see Table \ref{tableLB}).

A likely explanation for this failure by both the RBE predictive distribution and the AR$(4)$ model is that the conditional dynamics of NAO are not time-homogeneous on the period considered. To see why this could produce exactly the observed pattern, consider the simple model
$$x_{t+1}=ax_t+s(t)\epsilon_{t+1}$$
where the $\epsilon_t$ are iid, centered and have unit variance, but $s(t)$ is a deterministic function of time. An OTBP problem whose covariate is a recurrent state that does not identify $t$ or the temporal variable $s(t)$ may remove the autoregressive location and still return
$$\widehat y_{t+1}-\text{mean}(\widehat y)\simeq s(t)\epsilon_{t+1}$$
For $k\neq0$ these variables will be uncorrelated, since $\mathbb E[\widehat y_t\widehat y_{t-k}]=0$,
but their squares satisfy $\mathbb E[\widehat y_{t+1}^2]=s(t)^2$. Thus, the autocovariance of $\widehat y_t^2$ reproduces the temporal factor $s(t)^2$ and a Ljung-Box test will likely reject for the square, even if it passes for $\widehat y_t$. If, additionally, the autoregressive coefficient itself varies with time, then fitting a recurrent optimal transport barycentric solver that is blind to time can leave extra linear dependence in the out-of-sample innovations.

Figure \ref{fig:nao_stability} explores this possibility. The upper plot shows AR$(4)$ coefficients estimated in rolling windows of $1000$ observations, shifted by $50$ observations at a time. The lower plot shows the $180$-observation rolling standard deviation of residuals from the fixed AR$(4)$ model fitted on the training block. The vertical line marks the beginning of the test period. While this is technically not a test for parameter change (notice the coefficients overlap), the plots nonetheless give clear evidence against a time-homogeneous conditional scale. The rolling standard deviation of the AR$(4)$ residuals shows persistent, almost periodic oscillations through both the training and the testing periods, with a peak around 2019. The rolling autoregressive coefficients also show a large excursion that coincides with an episode of unusually high volatility around 2019. These patterns show no discontinuity at the train/test cut, so it's more natural to interpret them as some recurrent form of time-inhomogeneity, rather than as a single structural break.

\begin{figure}[htbp]
    \centering
    \includegraphics[width=\textwidth]{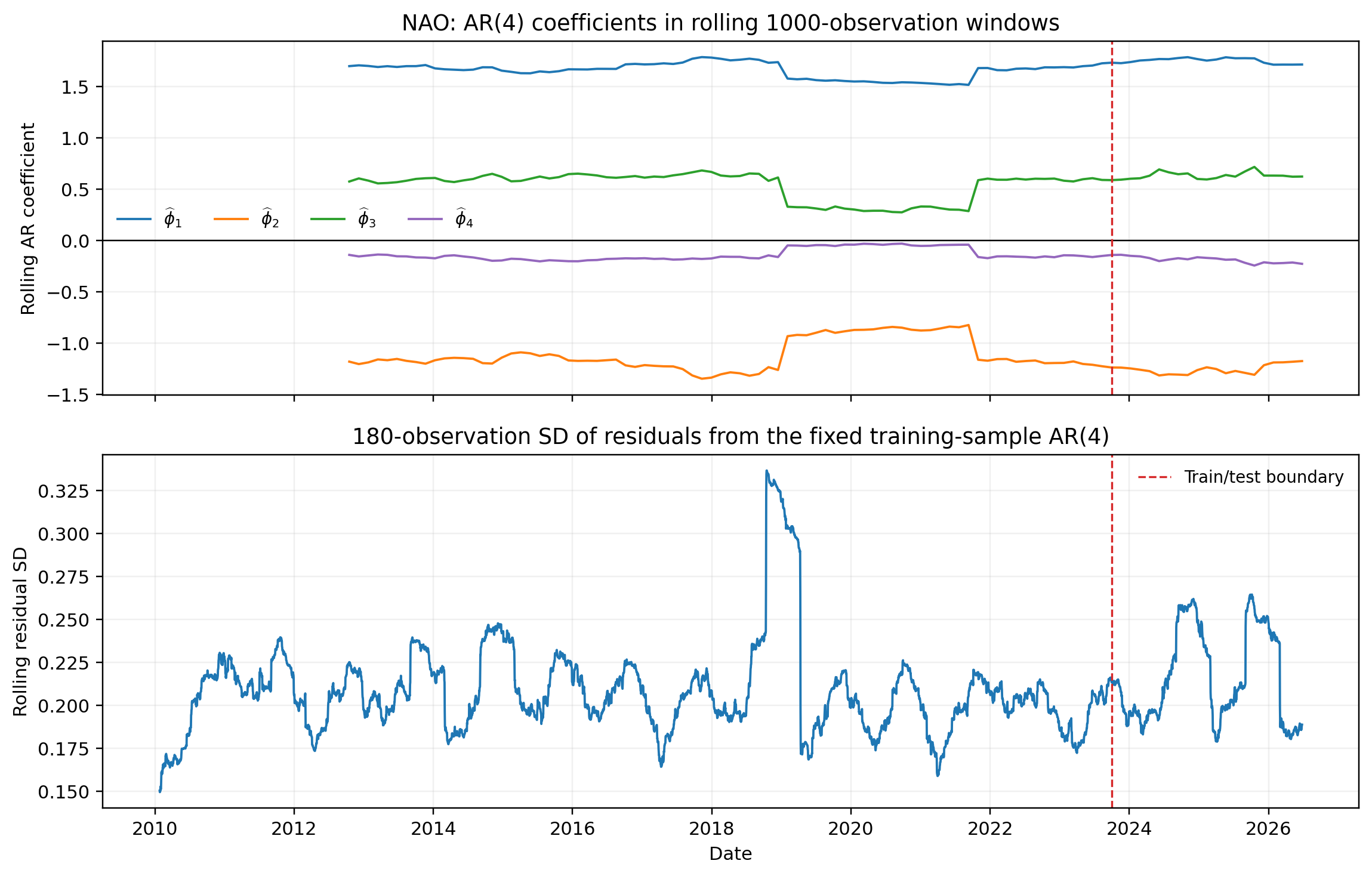}
    \caption[Rolling stability diagnostics for NAO]{Rolling stability diagnostics for NAO. AR$(4)$ coefficient estimates and standard deviation of residuals on rolling windows. The dashed vertical line separates the training block from the test block.}
    \label{fig:nao_stability}
\end{figure}

\subsection{Concluding remarks}\label{concluding}

The three series illustrate both the versatility and some limitations of the method. For the SP500, the OTB/RBE framework is successful both in improving the predictive score against all alternatives other than the location-scale ensemble, and in producing innovations for which standard tests can't detect any remaining serial dependence. For Bitcoin there is also an improvement in CRPS against all alternative models. However, some of the Ljung-Box tests show dependence, and the KS test for uniformity rejects. For the NAO data the method recovers the conditional mean dependence, although it offers no material improvement over an AR$(4)$ model that extracts the same structure, and performs worse than the location-scale ensemble when implemented with two monomial test functions.

The KS and Ljung-Box rejections for Bitcoin suggest that our stationarity assumption may not hold in practice. The autocorrelation patterns that show peaks at daily and weekly lags indicate that there is some periodic component, while the PIT upward bias suggests either a regime change or a structural break between the distributions before and after the train/test cut. Similarly, the rolling diagnostics for NAO suggest some variation across time in the scale of innovations and maybe also in their location. While the theory was developed for stationary processes, the method could be enriched by adding deterministic time covariates, especially periodic functions, such as time of day, day of week or seasonal phase, or some sinusoidal function of time, that would account for the evident time-inhomogeneity in the input data. Our population theory could be extended to cyclostationary processes to account for the inclusion of such covariates. For non periodic time dependence extending the population theory is a harder challenge. 

Another desirable improvement within this experimental setup would be to add a validation stage for model selection, especially for the CRPS comparisons. Currently, parametric choices for the RBE such as the number of channels $k$ and the hidden dimension $d$ were made intuitively based on synthetic results, and the choice for a best alternative was done on the same test sample used to compute the bootstrap significance tests. A cleaner experiment design would add a third, middle validation block to choose both $k$, $d$ and the $\varphi$ variant for the RBE and to choose a best alternative among different KNN$(p,K)$ models, rolling window models of different window sizes or AR$(q)$ models for different $q$. 

\section{An implementation of path simulation}\label{pathsimulation}

We turn now to the problem of path simulation, that is, of generating a new path $x_1',\dots,x_T'$ sampled under the law $P$ that generated the training path $x_1,\dots,x_T$.

Suppose first that $z_t$ is a sufficient statistic for the past and that the
exact transport map $T_z$ and barycentric distribution $\mu$ are known. By
construction, $(T_z)_\#P(\cdot\mid z)=\mu$
and hence
$$(T_z^{-1})_\#\mu=P(\cdot\mid z)$$
If $y_t'$ is then sampled independently from $\mu$ and we set $x_t'=T_{z_t'}^{-1}(y_t')$
then, conditionally on $z_t'=z$, the generated observation has distribution
$P(\cdot\mid z)$. If the state is updated correctly and $x_0'$ is sampled from the appropriate initial distribution, iterating this procedure reproduces the joint law of the process.

For our recurrent barycentric encoder, however, the state used by the inverse map would itself have to be constructed from the simulated path:
$$h_t'=\widehat G(h_{t-1}',x_{t-1}'),\qquad x_t'=\widehat T_{z_t'}^{-1}(y_t'),\qquad y_t'\overset{\mathrm{iid}}{\sim}\widehat\mu$$
That is, during prediction the recurrent state was computed from the observed history, but for path simulation it is computed from observations generated by the model itself. Approximation errors in the inverse map
thus change the subsequent state and, through it, all later conditional
distributions. Controlling the resulting discrepancy between the true and learned path laws would require stability assumptions on both the recurrent state update and the inverse transport maps. This is an interesting challenge that we will explore in further work, but which falls beyond the scope of this thesis.

We therefore did not use the recurrent barycentric encoder for path simulation and, instead, restricted our path simulations to synthetic Markov processes, for which the sufficient statistic for the past is known in advance. We did a run of the OTBP algorithm from \cite{Tabak25} on synthetic data generated by an ARCH$(1)$ model,
$$x_t=\sigma_t\epsilon_t,\qquad \sigma_t^2=\omega+\alpha x_{t-1}^2$$
with $\epsilon_t$ iid Gaussian standard variables and parameters $\omega=0.01$ and $\alpha=0.25$. This is a first-order Markov process, and can be regarded as a special case of the GARCH$(1,1)$ model, as explained in Section \ref{wellknown}.

As with the GARCH model, the exact Monge solution of the barycentric problem is known in closed form
$$y_t=T(x_t,z_t)=\frac{x_t}{\sqrt{\omega+\alpha z_t^2}}\mathbb{E}\left[\sqrt{\omega+\alpha z_t^2}\right]=\mathbb{E}\left[\sqrt{\omega+\alpha z_t^2}\right]\epsilon_t$$
with $z_t=x_{t-1}$.

As a first diagnostic of this alternative OTBP algorithm, Figure \ref{fig:1} plots the learned barycentric $\widehat y_t$ returned by the algorithm against the iid noise $\epsilon_t$ driving $x_t$. As predicted by our theory, the plot follows a straight line, with constant slope approximately equal to $\mathbb{E}[\sigma_t]=\mathbb{E}\left[\sqrt{0.01+0.25x_t^2}\right]$.

\begin{figure}
    \centering
    \includegraphics[width=0.5\linewidth]{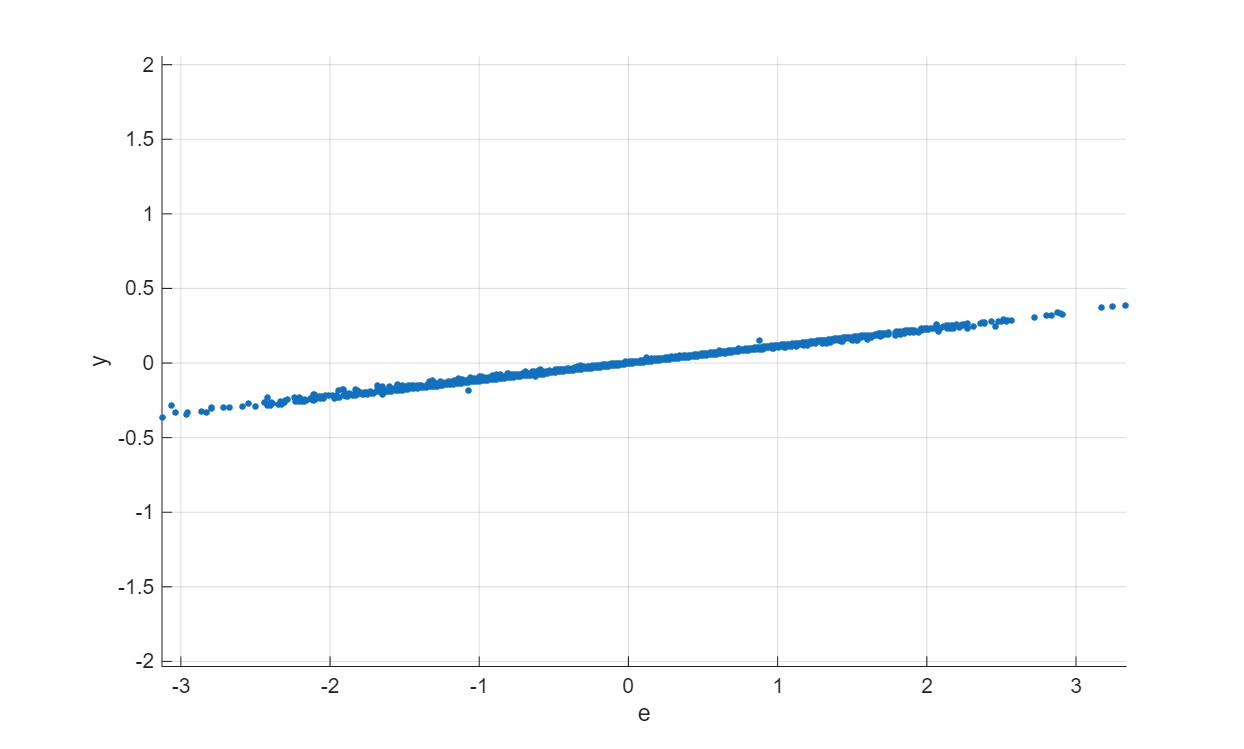}
    \caption[Learned barycentric variables plotted against the driving noise $\epsilon_t$ of an ARCH$(1)$ process]{Learned barycentric variables plotted against the driving noise $\epsilon_t$ of an ARCH$(1)$ process. Parameters  were $\omega=0.01$ and $\alpha=0.25$. The approximately linear relation agrees with the population identity $y_t=\mathbb E[\sigma_t]\epsilon_t$.}
    \label{fig:1}
\end{figure}

We then proceeded to simulate a new path as explained above by drawing $y_1',\ldots,y'_{T_{\text{test}}}$ independently from the empirical barycentric distribution
$$\widehat\mu_T=\frac{1}{T}\sum_{t=1}^T\delta_{\widehat y_t}$$

Starting from an initial
value $x_0'$, the simulated path was then generated recursively according to
$$x_t'=\widehat T_{x_{t-1}'}^{-1}(y_t')$$
Notice that for the exact population-level transport map this recursion reduces to
$$x_t'=\frac{\sigma(x_{t-1}')}{\mathbb E[\sigma_t]}y_t'=\sigma(x_{t-1}')\epsilon_t'$$
which is precisely the original ARCH$(1)$ transition.

Next, in order to verify whether the sampling method using the OTBP inverse transport map can  simulate the original stochastic law we compared its ARCH likelihood with that of a path of equal length generated directly from the ARCH model. With Gaussian innovations the conditional log-likelihood, up to an additive constant, is
$$\ell_{T_{\text{test}}}(\omega,\alpha)=-\frac{1}{2}\sum_{t=1}^{T_{\text{test}}}\left[\log\left(\omega+\alpha x_{t-1}^2\right)+\frac{x_t^2}{\omega+\alpha x_{t-1}^2}\right]$$

The log likelihood offers a direct evaluation of whether the simulated path reproduces the nonlinear conditional relation between $x_t$ and $x_{t-1}$. Figures \ref{fig:2} and \ref{fig:3} show the surfaces and contour lines of the likelihood functions on the parameters $\omega$ and $\alpha$ for one representative pair of paths of the ARCH-model-generated and the OTBP simulated paths.

\begin{figure}
    \centering
    \includegraphics[width=0.5\linewidth]{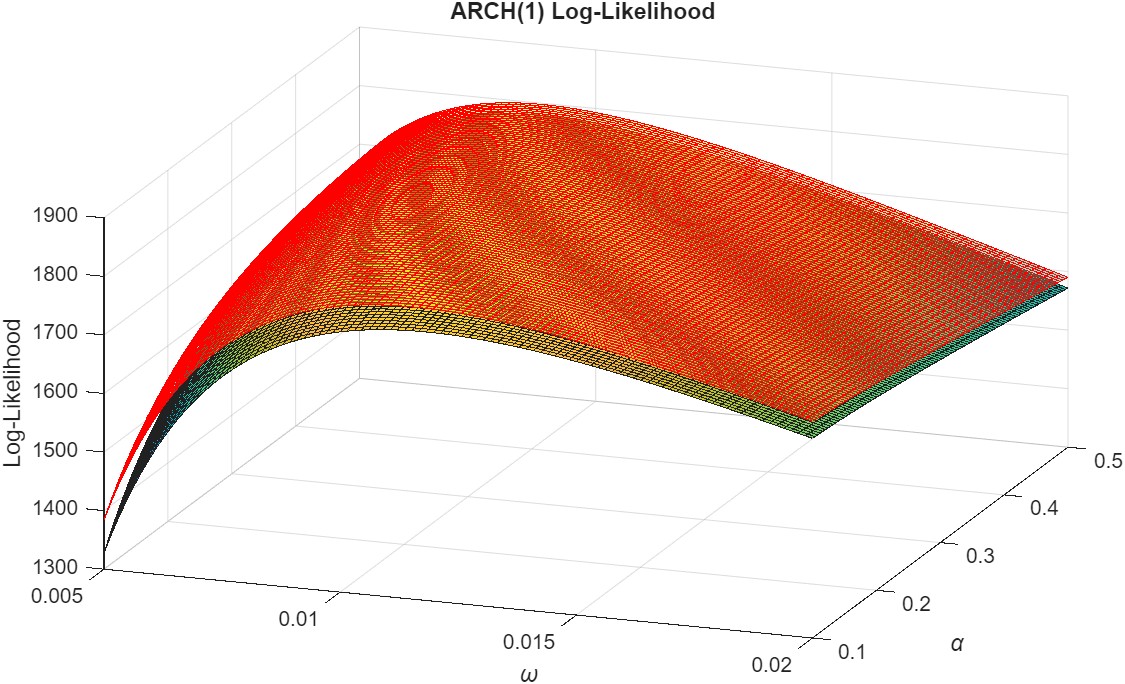}
    \caption[Likelihood surfaces for the OTB simulated path and the ARCH-model generated path]{Likelihood surfaces for the OTB simulated path and the ARCH-model generated path}
    \label{fig:2}
\end{figure}

\begin{figure}
    \centering
    \includegraphics[width=0.5\linewidth]{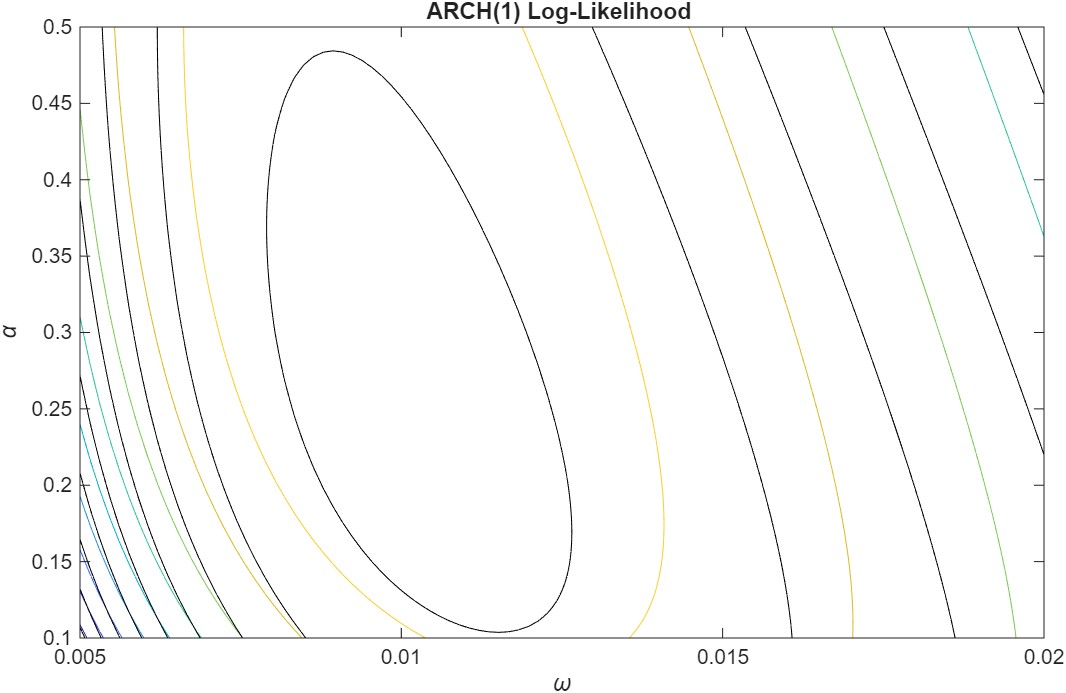}
    \caption[Likelihood contours for the OTB simulated path and the ARCH-model generated path]{Likelihood contours for the OTB simulated path and the ARCH-model generated path}
    \label{fig:3}
\end{figure}

The surfaces and contour lines exhibit a close alignment, suggesting that the method is effective at simulating paths that obey the law of a given ARCH$(1)$ process. To strengthen this conclusion beyond a single run (which could show alignment simply by chance), we ran a more systematic Monte Carlo experiment. In this experiment, the learned transport map and barycentric variables were held fixed, and $50$ independent paths of length $T_{\mathrm{test}}=2400$ were generated using the OTB inverse sampler. For comparison, $50$ paths of the same length were generated directly from the ARCH$(1)$ model, with an initial burn-in discarded sample of $500$ observations. A Gaussian ARCH$(1)$ model was then fitted by maximum likelihood to each path. Table \ref{tabarchMC} shows the results.

\begin{table}[htbp]
\centering
\scriptsize
\resizebox{\textwidth}{!}{%
\begin{tabular}{lcccccc}
\hline
Generator
& Sample mean
& $\widehat\omega$
& $\widehat\alpha$
& Sample variance
& $\widehat\rho_1(x_t^2)$
& $\ell(\widehat\omega,\widehat\alpha)/T_{\mathrm{test}}$ \\
\hline
Theoretical ARCH$(1)$
& $0$
& $0.010000$
& $0.250000$
& $0.013333$
& $0.250000$
& -- \\

Simulated ARCH
& $0.000254\;(0.002651)$
& $0.009961\;(0.000481)$
& $0.252758\;(0.033635)$
& $0.013329\;(0.000528)$
& $0.250612\;(0.048438)$
& $0.764923\;(0.018141)$ \\

OTBP inverse
& $-0.002871\;(0.002426)$
& $0.009633\;(0.000386)$
& $0.263266\;(0.034771)$
& $0.013135\;(0.000553)$
& $0.254624\;(0.031937)$
& $0.776687\;(0.017205)$ \\
\hline 
\end{tabular}%
}
\caption[Comparison of diagnostics for paths generated from the ARCH$(1)$ model against those generated by the OTBP inverse.]{Comparison of diagnostics for paths generated from the ARCH$(1)$ model against those generated by the OTBP inverse. Each entry in the second and third rows is the mean over $50$ paths, with the Monte Carlo standard deviation in parentheses. Each path contains $2400$ retained observations after a burn-in of $500$ observations. The OTBP map was learned once and remained fixed across simulations.}
\label{tabarchMC}
\end{table}

All likelihood optimizations converged. The simulations generated directly by the ARCH$(1)$ model are used as a finite sample benchmark. The paths generated by the OTB inverse have average parameter estimates, variance and squared lag $1$ autocorrelation that remain close to their population values. Their average variance is $0.013135$, compared with the theoretical value $0.013333$, while their average squared first-order autocorrelation is $0.254624$, compared with $0.25$. Thus, the nonlinear conditional variance dependence of the ARCH process is recovered. 

However, the experiment shows small discrepancies: the average estimate of $\omega$ for the OTBP paths is approximately $3.7\%$ below its true value, while the average estimate of $\alpha$ is approximately $5.3\%$ above its true value. The generated paths also have a small negative sample mean.
To put these discrepancies in perspective, we tested the statistical significance of the difference of these estimates with those arising from the paths generated directly by the ARCH$(1)$ model. We found that for the average of $\hat\omega$ and of the sample mean they are indeed statistically significant at the $1\%$ level. For $\alpha$, the variance, and the squared lag-1 autocorrelation they are not significant at the $5\%$ level\footnote{Writing $\bar{\widehat\theta}_{\text{ARCH}}$ and $\bar{\widehat\theta}_{\text{OTBP}}$
for the Monte Carlo means of any of parameter estimates and $s^2_{\theta,\text{OTBP}}$ and $s^2_{\theta,\text{ARCH}}$ for the corresponding Monte Carlo variances and using 
$$\widehat{\text{se}}(\delta_\theta)=\sqrt{(s^2_{\theta,\text{OTBP}}+s^2_{\theta,\text{ARCH}})50^{-1}}$$
for $\delta_\theta=\bar{\widehat\theta}_{\text{ARCH}}-\bar{\widehat\theta}_{\text{OTBP}}$, then 
the resulting absolute test statistics for $\widehat\omega$, $\widehat\alpha$, the sample mean, the variance and the squared lag-1 autocorrelation are 3.76, 1.54, 6.15, 1.79 and 0.49, respectively.}

Still, the positive evidence is that the likelihood surfaces\footnote{Notice that the larger mean log likelihood for the OTBP paths should not be interpreted as evidence that they fit the ARCH model better than paths generated directly from it. The likelihood function is evaluated on different simulated paths, and the absolute level depends on the scale of the samples. In particular, the OTBP paths have a slightly smaller variance.} are centered around the true parameters and that the estimated parameters and diagnostics are close to the true ARCH values, showing that the learned OTB sampler offers a very close approximation to the law of the ARCH$(1)$ model.

\chapter{Conclusion, future work and connections to existing literature}

In this thesis we extended the optimal transport barycenter problem to time-dependent processes via a recurrent adversarial formulation. This formulation allows us to learn both a recurrence $\widehat G$, whose role is to encode the information from the observed past that is relevant to predict $x_{t+1}$, and barycentric variables $\widehat y_t$ that capture the remaining variability. We then showed how these variables, the learned transport map and the learned recurrence can be combined to build the Recurrent Barycentric Encoder, a predictive sampler.

We ran several experiments on the RBE, with both synthetic and historical data, to show that the method can recover the conditional laws of processes with different forms of dependence. The RBE predictive distributions were close to the oracle distributions for the synthetic processes and, among the historical applications, we obtained positive predictive results with SP500 and Bitcoin data. For Bitcoin, the RBE models had an out-of-sample mean CRPS that was detectably lower than every alternative considered, while they were matched only by a recurrent location-scale ensemble in the case of SP500 data.

The NAO results, while not negative, were not as positive. The monomial RBE seemed to have recovered  the same predictive distribution as the AR$(4)$ residual bootstrap model, and had a detectably higher CRPS than the recurrent location-scale ensemble. On the other hand, the neural 3-channel RBE did not have a detectable difference with this RLS ensemble, and performed slightly better than the AR$(4)$. At the same time, both the NAO and the Bitcoin diagnostics showed persistent dependence that suggests the stationarity assumption does not hold. In the Bitcoin dataset in particular, the training and test distributions were different, and some clear daily and weekly periodic scale patterns remained in the barycentric variables. In cases where stationarity fails, observations from different parts of the sample may not follow the same conditional law, and pooling them to estimate a single conditional law is not appropriate.

A natural next step in our research is therefore to extend the population theory and the RBE implementation to include time covariates, especially periodic time covariates, as mentioned in Subsection \ref{concluding}. An extension to non periodic time covariates would likely entail a more challenging adaptation of our stationary population theory. 

The inclusion of periodic time covariates in the algorithmic implementation would also open the way to the inclusion of exogenous covariates, such as external market or atmospheric variables, which may contain predictive information not recoverable from the past of $x_t$. Indeed, extending both the population theory and the RBE construction to include other covariates along the lines of Lemma \ref{lemmaX} is a direction currently pursued in our research\footnote{Unfortunately, this extension remained incomplete at the time of thesis submission}.

A different, simple but important direction for improvement in our research is to split our samples in three to add an intermediate validation stage for model selection in our empirical tests, as explained in subsection \ref{concluding}.

There is also considerable work to do on path simulation. The experiment in Section \ref{pathsimulation} was restricted to a Markov ARCH$(1)$ process to be able to use the non-recurrent OTBP inverse sampler, for which the covariate must be given. Still, the OTBP generated sample had a detectably different maximum-likelihood estimated parameter when compared with directly generated ARCH$(1)$ samples. More generally, making the RBE capable of accurate recurrent path simulation will require studying how errors in the learned recurrence and inverse map accumulate along a simulated path and ensuring we can establish some bounds for them. 

Finally, it behooves us to explain in what sense our method relates to what other recurrent methods for predictive sampling do. As mentioned before, choosing a recurrence $h_{t+1}=G(x_t,h_t)$ to capture a process' history is standard in the literature. In the last decade and a half, different RNN-based architectures that predict whole conditional distributions have been developed. Probabilistic RNNs\cite{Graves14} model the output as a sample from a distribution whose parameters are a function of the RNN's recurrent state $h_t$, so that the prediction $\hat{x}_{t+1}\sim\mathcal{D}(g_\theta(h_{t+1}))$ for some distribution family $\mathcal{D}$, usually a Gaussian or a Gaussian mixture parametrized with a Mixture Density Network.

An influential specialization of this construction aimed at time-series forecasting was introduced as DeepAR in \cite{SalinasEtAl20}. While keeping the main probabilistic RNN structure, DeepAR's main additional goal was to share one model across many related series. Notice that each restart of the recurrent location-scale model that we used to build and ensemble benchmark in Subsection \ref{predictive} is a simple single-series instance of this DeepAR principle.

A different distributional construction was introduced using conditional normalizing flows for time series analysis in \cite{RasulEtAl21}, where the authors model a multivariate observation as the image of a standard Gaussian $z_t$ under an invertible map $T^{-1}_{\theta,h_t}:z_t\mapsto \tilde{x}_t$ that is conditioned on the recurrent state $h_t$. The map and the RNN are fit using change-of-variables likelihood as objective, and then forecasting paths can be simulated by mapping independent innovations through the inverse flow and feeding each generated forecast observation back into the recurrence. This is structurally close to the RBE, although the innovation law is fixed in advance.

In the scalar case, autoregressive quantile flows (AQF)\cite{SiEtAl22}, for their part, model the conditional quantile map
$$
\widehat x_{t+1}=Q_{\theta}(u_{t+1}\mid h_{t+1}),
\qquad u_{t+1}\sim\operatorname{Unif}(0,1),
$$
and train it using integrated quantile loss, which is proportional to CRPS. Notice that unlike likelihood-trained methods such as CNFs, AQFs learn the map in the direction used for sampling, i.e. from the innovation $u$ to $x_t$, and therefore need not compute an inversion of the learned map.

In different ways, all these methods admit a representation
$$
\widehat x_{t+1}=T_{\theta,h_{t+1}}(\varepsilon_{t+1}),\qquad \varepsilon_{t+1}\sim\mu_0
$$
If the recurrent state is sufficient and the conditional law is learned without error, this means that target observations can be expressed as deterministic functions of the past and of innovations independent of the past.

In other words, recurrence, conditional sampling and innovation extraction do not by themselves distinguish our RBE. Rather, the distinctive feature of the RBE is how this representation is arrived at. Instead of fixing $\mu_0$, our abstract population problem uses the Wasserstein barycenter of the conditional laws and the transport maps associated to them. This choice of innovation, together with the adversarial framework used to enforce independence from the past, are the main distinguishing differences between the RBE and these alternatives. Clearly, a natural direction for future work would include a comparison of our methods with CNF and AQF. In any case, the experiments we carried out establish the viability of the RBE as a framework for building predictive distributions, which at the same time has a distinct and clear interpretation in terms of Wasserstein barycenters.

\appendix
\addtocontents{toc}{\protect\contentsline{chapter}{APPENDICES}{\protect\pageref{app:proofs}}{}}
\chapter{Deferred Proofs}\label{app:proofs}
\section{Proof of Lemma \ref{lemmauniquekalman}}

\begin{proof}
    Write $a^*=(1-K\beta)\alpha$. From the previous discussion and (\ref{eqK}) it is clear that $G(x,h)=a^* h+bx$ is a solution for any $b\neq 0$ ($b=0$ gives the stationary solution $h_t=0$ for all $t$). We must show that all solutions are of this form.  For any scalar coefficient $a$ with $|a|<1$ write $h_t^{(a)}=\sum_{i=1}^{\infty}a^{i-1}x_{t-i}$ for the solution to the linear recursion with $b=1$ (where $x_t$ is extended bilaterally if necessary), which is always a Gaussian process, so that $y_t\perp h_t^{(a)}$ implies $y_t{\perp\!\!\!\perp} h_t^{(a)}$ (and viceversa).

Let $G(x,h)=a'h+b'x$ be any solution to (\ref{eqlatent}) with $G$ restricted to linear functions. We want to prove that $a'=a^*$. As we have seen, if $G$ is a solution then $\tilde{G}(x,h)=a'h+x$ is also a solution, which (by our previous derivation of the solution $y_t^*$ to the inner problem) means the squared norm of the orthogonal projection of $x_t$ on the subspace of $h_t^{(a)}$-measurable functions achieves the maximum value, i.e. $$\mathbb{E}\left[\left|P_{h_t^{(a')}}x_t\right|^2\right]=\mathbb{E}\left[\left|P_{h_t^{(a^*)}}x_t\right|^2\right]=\mathbb{E}\left[\left|P_{\mathcal{F}_{t-1}}x_t\right|^2\right]$$
Because $P_{h_t^{(a')}},P_{h_t^{(a^*)}}\leq P_{\mathcal{F}_{t-1}}$, this implies that $P_{h_t^{(a')}}x_t=P_{h_t^{(a^*)}}x_t$. Since $x_t$ and $h_t^{(a)}$ are jointly Gaussian this implies that $P_{h_t^{(a^*)}}x_t$ is a scalar multiple of $h_t^{(a^*)}$ (and the scalar is zero iff they are independent). The same is true for $h_t^{(a')}$ so $h_t^{(a')}$ and $h_t^{(a^*)}$ are scalar multiples of each other. That is, there must be a constant $k$ such that $h_t^{(a')}=kh_t^{(a^*)}$, and
$$Y:=\sum_{i=1}^\infty((a')^{i-1}-k(a^*)^{i-1})x_{t-i}=h_t^{(a')}-kh_t^{(a^*)}=0$$
For the random variable $Y$ defined by this infinite sum to equal $0$ its variance must be $0$. We can write the variance using the spectral decomposition of $x_t$ as follows. Write
$$R(h):=\mathbb{E}[x_tx_{t-h}]=\int_{-\pi}^{\pi} e^{ih\lambda} d\mu $$
where $\mu$ is the measure given by Bochner's theorem. This measure has a density $\mu=fd\lambda$ since, from \eqref{eqHMlin},
$$
z_t=\sum_{r=0}^{\infty}\alpha^r\epsilon_{t-1-r},
\qquad
R(h)=\frac{\beta^2\tau_\epsilon^2}{1-\alpha^2}\alpha^{|h|}
+\tau_\eta^2\mathbf{1}_{\{h=0\}}.
$$
with $R\in\ell^1(\mathbb Z)$ (because $|\alpha|<1$,), which implies that the spectral measure is absolutely continuous with density
$$
f(\lambda)=\frac{1}{2\pi}\left(
\tau_\eta^2+
\frac{\beta^2\tau_\epsilon^2}
{1+\alpha^2-2\alpha\cos\lambda}
\right)>0
$$

Then
\begin{align}
Var(Y)&= \sum_{j,l=0}^\infty ((a')^j-k(a^*)^j)((a')^l-k(a^*)^l)R(|l-j|)\nonumber\\
&= \int_{-\pi}^{\pi} \left[ \sum_{j,l=0}^\infty ((a')^j-k(a^*)^j)((a')^l-k(a^*)^l) e^{ij\lambda} e^{-il\lambda} \right] f(\lambda) d\lambda \nonumber\\
&= \int_{-\pi}^{\pi} \left( \sum_{j=0}^\infty ((a')^j-k(a^*)^j)e^{ij\lambda} \right) \left( \sum_{l=0}^\infty ((a')^l-k(a^*)^l)e^{-il\lambda} \right) f(\lambda) d\lambda
\end{align}

Because the coefficients $((a')^i-k(a^*)^i)$ are real, the two series inside the integral are exact complex conjugates of one another. If we define $C(\lambda) = \sum_{j=0}^\infty ((a')^j-k(a^*)^j)e^{-i\lambda j}$, then the first sum is $\overline{C(\lambda)}$ and the second sum is $C(\lambda)$, yielding
$$Var(Y)= \int_{-\pi}^{\pi} \left|\sum_{j=0}^\infty ((a')^j-k(a^*)^j)e^{-i\lambda j}\right|^2 f(\lambda) d\lambda$$

Because $f(\lambda)>0$ for all $\lambda$, then $Y=0$ implies $\left|\sum_{j=0}^\infty ((a')^j-k(a^*)^j)e^{-i\lambda j}\right|^2=0$ for Lebesgue-a.e. $\lambda$. By uniqueness of Fourier coefficients, this implies $(a')^j-k(a^*)^j=0$ for all $j$, from which $a'=a^*$ (and $k=1$) follows.

\end{proof}

\section{Proof of Lemma \ref{lemmaG}}
\begin{proof}     We will work with bilateral processes, so we first extend  $x_t$ bilaterally to $t\in\mathbb{Z}$ and write $x=(x_t)_{t\in\mathbb{Z}}$. We will also write $\theta$ for the shift operator $(\theta x)_t=x_{t+1}$ (notice that $\theta$ is invertible for bilateral processes). Next, we define a sequence of random variables $z^{(n)}=z^{(n)}(x)$ as follows:
    \begin{align}
        z^{(0)}(x)&=0\nonumber\\
        z^{(n)}(x)&=G(x_{-1},z^{(n-1)}(\theta^{-1}x))
    \end{align}

    Let's see that $z^{(n)}$ is a Cauchy sequence in $L^2$ write $D_n(x):=|z^{(n)}(x)-z^{(n-1)}(x)|$. By condition C1 we have
    $$D_{n+1}(x)\leq \Lambda(x_{-1})\cdot |z^{(n)}(\theta^{-1}x)-z^{(n-1)}(\theta^{-1}x)|=\Lambda(x_{-1}) D_n(\theta^{-1}x)$$
    and iterating this inequality we get
    $$D_{n+1}(x)\leq \prod_{k=1}^n \Lambda(x_{-k}) D_1(\theta^{-n}x)=\prod_{k=1}^n \Lambda(x_{-k}) |G(x_{-n-1},0)|$$
    where in the last equality we use $|G((\theta^{-n}x)_{-1},0)-0|=|G(x_{-n-1},0)|$. From $z^{(0)}=0$ and the finiteness of $a_n$ for all $n\geq 0$ in condition C2 we get that the sequence is in $L^2$. Moreover, we get the $L^2$ bound (for $m>n$)
    \begin{align}\|z^{(m)}-z^{(n)}\|_2&\leq \sum_{k=n}^{m-1}\|z^{(k+1)}-z^{(k)}\|_2=\sum_{k=n}^{m-1}\mathbb{E}\left[|D_{k+1}(x)|^2\right]^{1/2}\nonumber\\
    &\leq \sum_{k=n}^{m-1} \mathbb{E}\left[ \prod_{j=1}^k \Lambda(x_{-j})^2 |G(x_{-k-1},0)|^2\right]^{1/2}= \sum_{k=n}^{m-1} a_k^{1/2}\leq\sum_{k=n}^{m-1} K\rho^{k/2} 
    \end{align}

where in the last inequality $K$ is a constant and $\rho<1$ is given by the limsup condition for the $a_n$ in C2 (and stationarity). Since the geometric series in the RHS goes to 0 as $n\to \infty$ it follows that $z^{(n)}(x)$ is Cauchy. Let $z_*=z_*(x)$ be the limit. We define the solution process $h$ as $h_t=z_*(\theta^t x)$. It is immediate to check that $\mathbb{E}[h_t^2]=\mathbb{E}[z_*(\theta^t x)^2]=\mathbb{E}[z_*(x)^2]<\infty$, and stationarity then follows from $\theta^s x\overset{d}{=}x$.

We must still verify this $(h_t)_{t\in\mathbb{Z}}$ satisfies the recurrence equation. By stationarity it suffices to verify it for $t=1$, which by the definition of $h_1$ is equivalent to showing $G(x_0,z_*(x))=z_*(\theta x)$. Because $z_*$ is a limit in $L^2$ we can take a subsequence $z^{(n_k)}(x)$ that converges to $z_*(x)$ almost surely. Equivalently, the set $A':=\{x:z^{(n_k)}(x)\to z_*(x)\}$ has full measure. By stationarity of $x$, $\theta^{-1} A$ has the same measure as $A$ for any $x$-measurable set $A$, in particular for $A'$. So $\theta^{-1} A'=\{x:z^{(n_k)}(\theta x)\to z_*(\theta x)\}$ has full measure. Equivalently, $z^{(n_k)}(\theta x)$ converges to $z_*(\theta x)$ almost surely. Now, $z^{(n_k-1)}$ converges in $L^2$ so we can take a sub-subsequence (we still call it $z^{(n_k-1)}$) so that $z^{(n_k-1)}$ converges almost surely. Now by construction 
$$z^{(n_k)}(\theta x)=G(x_0,z^{(n_k-1)}(x))$$
We just saw the LHS converges almost surely to $z_*(\theta x)$. To show the RHS converges a.s. to $G(x_0,z_*(x))$, from C1 we get   
$$|G(x_0,z^{(n_k-1)}(x))-G(x_0,z_*(x))|\leq \Lambda(x_0)|z^{(n_k-1)}(x)-z_*(x)|\to 0 \quad\text{ a.s.} $$
This proves that $G(x_0,z_*(x))=z_*(\theta x)$.

To prove uniqueness, write
$$\rho(\Lambda) := \limsup_{n\to\infty} \left( E\left[ \prod_{k=1}^n \Lambda(x_k)^2 \right] \right)^{1/n} < 1$$
  
  Assume there are two stationary, square-integrable solutions $(w_t)_{t\in\mathbb{Z}}$ and $(v_t)_{t\in\mathbb{Z}}$ and  write $\delta_t=|w_t-v_t|$. Notice that $\delta_t$ must be square integrable because so are $w_t$ and $v_t$. Then, iterating backwards, we get

  $$\delta_0\leq \prod_{j=-t}^{-1}\Lambda(x_j)\delta_{-t}$$
  and, writing $A_t:=\prod_{j=-t}^{-1}\Lambda(x_j)$ and using Cauchy-Schwarz, we have
  $$E[\delta_0]\leq E\left[A_t\delta_{-t}\right]\leq E\left[A_t^2\right]^{1/2}\mathbb{E}[\delta^2_{-t}]^{1/2}$$
  The second factor is uniformly bounded, since $\delta_{-t}^2=|w_{-t}-v_{-t}|^2\leq 2|w_{-t}|^2+2|v_{-t}|^2$, and $w_{-t}$ and $v_{-t}$ are square-integrable.
  
  For the first factor, by stationarity, we have
  $$\mathbb{E}[A_t^2]=\mathbb{E}\left[\prod_{j=-t}^{-1}\Lambda(x_j)^2\right]=\mathbb{E}\left[\prod_{j=1}^{t}\Lambda(x_j)^2\right]\leq (\rho(\Lambda)+\eta)^{t}\to 0\quad \text{ as }t \to\infty$$ for $\eta$ any positive real such that $\rho(\Lambda)+\eta<1$.

  Hence, $\mathbb{E}[\delta_0]=0$. Because $\delta_0\geq 0$ this implies that $\delta_0=0$ a.s. Repeating the argument with any $t$ instead of 0 yields $\delta_t=0$ a.s., proving that $w_t=v_t$.

\end{proof}

\section{Proof of Theorem \ref{teoopt}}
\begin{proof}

    The proof works in general when $\mathcal{X}$ is any uncountable standard Borel space, of which $\mathbb R^{d_x}$ is an example, since they are all Borel isomorphic.

    Let $K\subseteq [0,1]$ be the middle third Cantor set and $\phi:\mathcal{X}\to K$ be a Borel isomorphism, which exists because $K$ is an uncountable standard Borel space. Let $F:K\times K\to K$ be the bijection defined as follows on the triadic expansion of two elements $k=\sum_{i=1}^\infty a_i3^{-i}$ and $k'=\sum_{i=1}^\infty a_i'3^{-i}$ ($a_i,a_i'\in\{0,2\}$) in $K$

    $$F(k,k')=\sum_{j=1}^\infty c_j 3^{-j}\quad\text{ where }\begin{cases}
        c_{2i-1}=a_i\\
        c_{2i} = a_i'
    \end{cases}$$
    That is, $F$ interleaves the coefficients of $k$ and $k'$ starting with $k$.
    
    We define $G$ as
    \begin{equation}\label{defG}G(x,z)=\sup_{k\in K}\left(F(\phi(x),k)-\frac{7}{12}|z-k|\right)
    \end{equation}

    The proof has two parts.

    {\em Part 1.} $G$ satisfies $C$. Since $G(x,0)$ is bounded (and, therefore, square-integrable), it suffices to prove that $G$ is Lipschitz with a constant independent of $x$ strictly smaller than 1.
    
    We first prove that for any fixed $x$ $F$ is 7/12-Lipschitz on the second coordinate. Indeed, let $j\geq 1$ be the first coordinate on which $k$ and $k'$ differ, then $|k-k'|\geq 3^{-j}$. Consider first the case where $j>1$. We have
    \begin{align}|F(\phi(x),k)-F(\phi(x),k')|&\leq\sum_{i=1}^\infty|a_i-a_i'|3^{-2i}\leq\sum_{i=j}^\infty|a_i-a_i'|3^{-2i}
    \nonumber\\
    &\leq 2\frac{3^{-2j}}{1-3^{-1}}=3^{-2j+1}\leq 3^{-j-1}\leq 3^{-1}|k-k'|\nonumber
    \end{align}

    For $j=1$ assume without loss of generality that $a_1=2$ and $a_1'=0$. Then, writing $d_i=(a_i-a_i')/2$ we have $|F(\phi(x),k)-F(\phi(x),k')|= 2\cdot 3^{-2}\left|1+\sum_{i=1}^\infty d_{i+1} 3^{-2i}\right|$ and similarly $|k-k'|=2\cdot 3^{-1}\left|1+\sum_{i=1}^\infty d_{i+1} 3^{-i}\right|$. We therefore need to prove that 
    \begin{equation}\label{ineqlip}\frac{1}{3}\frac{\left|1+\sum_{i=1}^\infty d_{i+1} 3^{-2i}\right|}{\left|1+\sum_{i=1}^\infty d_{i+1} 3^{-i}\right|}=3^{-1}\frac{1+\sum_{i=1}^\infty d_{i+1} 3^{-2i}}{1+\sum_{i=1}^\infty d_{i+1} 3^{-i}}\leq \frac{7}{12}
    \end{equation}
    where the absolute value has been removed because, since $d_j\in\{-1,0,1\}$, the series in the denominator is bounded in absolute value by $\frac{3^{-1}}{1-3^{-1}}=1/2$ and the one in the numerator by $\frac{9^{-1}}{1-9^{-1}}=1/8$, so that both numerator and denominator are strictly positive. Write $C_i=(7-4\cdot 3^{-i})$. Then by multiplying both sides of (\ref{ineqlip}) by $12(1+\sum_{i=1}^\infty d_{i+1} 3^{-i})$ and subtracting we see the inequality is equivalent to the inequality $0\leq 3+\sum_{i=1}^\infty d_{i+1}3^{-i}C_i$. Clearly $C_i>0$ for $i\geq 0$ so the RHS of this inequality is minimized when $d_{i+1}=-1$ for all $i$. We must therefore prove $\sum_{i=1}^\infty 3^{-i}C_i\leq 3$. We have
    $$\sum_{i=1}^\infty 3^{-i}C_i=7\sum_{i=1}^\infty 3^{-i}-4\sum_{i=1}^\infty 3^{-2i}=\frac 72-\frac12=3$$
    as wanted. Since $F$ is 7/12-Lipschitz it follows, from the definition $(\ref{defG})$, that $G(x,k)=F(\phi(x),k)$ for all $k\in K$. Now for any two real numbers $u,v$ and any $\epsilon>0$ we take $k\in K$ such that $G(x,u)\leq F(\phi(x),k)-\frac{7}{12}|u-k|+\epsilon$. By definition of the supremum, $ F(\phi(x),k)-\frac{7}{12}|v-k|\leq G(x,v)$. Substracting both inequalities we get
    $$G(x,u)-G(x,v)\leq \frac{7}{12}(|v-k|-|u-k|)+\epsilon\leq \frac{7}{12}|v-u|+\epsilon$$
    Since this was true for any $\epsilon$ we get $G(x,u)-G(x,v)\leq \frac{7}{12}|u-v|$. Doing the same for $k'$ that approximates $G(x,v)$ we get $G(x,v)-G(x,u)\leq \frac{7}{12}|u-v|$, completing the proof of the Lipschitz bound.

    {\em Part 2.} $\sigma(h_t)=\sigma(x_{<t})$ for any stationary solution $h_t$.

    We need to prove $\sigma(h_t)\supseteq \sigma(x_{<t})$, since $\sigma(h_t)\subseteq \sigma(x_{<t})$ is already clear by construction.

    Let $(h_t)_{t\in\mathbb{Z}}$ be a stationary solution. Let's show first that $h_t\in K$ a.s. Define $Y_t:=\text{dist}(h_t,K)$. Let $u_t\in K$ be a projection, i.e. $|u_t-h_t|=Y_t$. Then, since $G(x_t,u_t)\in K$,
    $$Y_{t+1}=\text{dist}(G(x_t,h_t),K)\leq |G(x_t,h_t)-G(x_t,u_t)|\leq\frac{7}{12}|h_t-u_t|=\frac{7}{12}Y_t$$
    But then $0\leq \mathbb{E}[Y_{t+1}]\leq \frac{7}{12}\mathbb{E}[Y_t]=\frac{7}{12}\mathbb{E}[Y_{t+1}]$, by stationarity. So $\mathbb{E}[Y_t]=0$, i.e. $h_t\in K$ almost surely.

    Since $\phi$ is a Borel isomorphism and $F$ is a Lipschitz bijection we have that $G_{\mid \mathcal{X}\times K}$ is a measurable bijection and, by Suslin's theorem, a Borel isomorphism. Since $h_{t+1}$ is a.s. in $K$ it follows that $\sigma(h_{t+1})=\sigma(x_t,h_t)\ (\!\!\!\!\mod 0)$. The claim follows by induction.\end{proof}

\section{Proof of Theorem \ref{teobarycenter1d}}

\begin{proof}
    Let $\mu_Y$ denote the marginal distribution of $y$. Since $y \perp\!\!\!\perp \mathcal{A}$, the conditional distribution of $y$ given $\mathcal{A}$ is almost surely equal to $\mu_Y$.

From Theorem \ref{teogeneralotb}, solving this constrained  square distance minimization over $y$ is equivalent to (\ref{eqwasser}), i.e. to finding the measure $\mu_Y$ that minimizes the expected 2-Wasserstein distance to the conditional distributions $\nu_\omega = P(x \mid \mathcal{A})(\omega)$. In our current notation this is:
$$ \min_{\mu_Y} \mathbb{E} \left[ W_2^2(\nu_\omega, \mu_Y) \right] $$

It is a standard fact about the $p$-Wasserstein distance in dimension one that 
$$ W_{p}(\mu _{1},\mu _{2})=\left(\int _{0}^{1}\left|F_{1}^{-1}(q)-F_{2}^{-1}(q)\right|^{p}\,\mathrm {d} q\right)^{\frac {1}{p}}$$
Thus, the squared 2-Wasserstein distance between two measures equals the $L^2$ distance between their quantile functions. Furthermore, because each conditional probability $\nu_\omega$ is atomless, the optimal coupling that achieves this distance with a fixed $\mu_Y$ is unique and given by a Monge map, which in dimension 1 is (see \cite[Theorem 4.3]{Galichon})
$$ y = T_\omega(x) = F_Y^{-1}( F_{x\mid\mathcal{A}}(x))$$

After substituting the quantile representation of the Wasserstein distance in the minimizing objective we get
$$ \min_{F_Y^{-1}} \mathbb{E} \left[ \int_0^1 \left| F_{x\mid\mathcal{A}}^{-1}(u) - F_Y^{-1}(u) \right|^2 du \right] $$

By Fubini, we can interchange the expectation and the integral, so that minimizing this functional is equivalent to minimizing the inner expectation pointwise for almost every $u \in [0, 1]$:
$$ \min_{F_Y^{-1}(u)} \mathbb{E} \left[ \left| F_{x\mid\mathcal{A}}^{-1}(u) - F_Y^{-1}(u) \right|^2 \right] $$

Since expectation is the minimizer of the expected squared distance to a random variable, the unique optimal quantile function for the barycenter must be
$$ F_Y^{-1}(u) = \mathbb{E} \left[ F_{x|\mathcal{A}}^{-1}(u) \right] $$
 as wanted. Combined with the optimal transport map, this yields the solution
$$y = F_Y^{-1}( F_{x\mid\mathcal{A}}(x) ).$$
\end{proof}

\section{Proof of Corollary \ref{1dcor}}
\begin{proof}
    We apply Theorem \ref{teobarycenter1d} with $\mathcal A=\sigma(z)$. Since $\epsilon$ is independent of $z$, the conditional distribution of $x$ given $z$ is the law of $F(z,\epsilon)$. For fixed $z$, the map $e\mapsto F(z,e)$ is continuous and strictly increasing, so it preserves quantiles. Hence
$$F_{x\mid z}^{-1}(u)=F\left(z,F_\epsilon^{-1}(u)\right)$$
By Theorem \ref{teobarycenter1d}, the barycenter quantile is obtained by averaging the conditional quantiles:
$$F_Y^{-1}(u)=\mathbb E_z\left[F_{x\mid z}^{-1}(u)\right]=\mathbb E_z\left[F\left(z,F_\epsilon^{-1}(u)\right)\right]$$
If $\bar{F}(e):=\mathbb E_z[F(z,e)]$ then
$$F_Y^{-1}(u)=\bar{F}(F_\epsilon^{-1}(u))$$

Define the random variable $y=\bar{F}(\epsilon)$, which is clearly independent of $z$. It suffices to show its quantile function is $F_Y^{-1}$.

Since each $F(z,\cdot)$ is increasing, $\bar{F}$ is increasing. Moreover, by the atomlessness of $\epsilon$, $U:=F_\epsilon(\epsilon)$ is uniformly distributed on $(0,1)$ and $\epsilon=F_\epsilon^{-1}(U)$ a.s. Therefore
$$y=\bar{F}(\epsilon)=\bar{F}(F_\epsilon^{-1}(U))=F_Y^{-1}(U)$$
Since $U$ is uniform, the random variable $F_Y^{-1}(U)$ has quantile function $F_Y^{-1}$, as required.

\end{proof}
\section{Proof of Corollary \ref{1dcor2}}
\begin{proof}
    Case b) is an immediate consequence of Corollary \ref{1dcor} once we notice that $\mathbb{E}[g(z)]+\mathbb{E}[f(z)]\epsilon=\mathbb{E}_z[F(z,\epsilon)]$ for $F(z,e)=g(z)+f(z)e$.

It remains to prove a). For each fixed $z$, the map $e\mapsto g(z)+|f(z)|e$ is continuous and strictly increasing. Applying Corollary \ref{1dcor} with $F(z,e)=g(z)+|f(z)|e$ gives the barycenter quantile
\begin{align}F_Y^{-1}(u)&=\mathbb E\left[g(z)+|f(z)|F_\epsilon^{-1}(u)\right]\nonumber\\
&=\mathbb{E}[g(z)]+\mathbb{E}[|f(z)|]F_\epsilon^{-1}(u)\nonumber
\end{align}

It remains to identify a representative $y$ that arises as a function of $x$ and $z$. Since $\epsilon$ is symmetric, $f(z)\epsilon\stackrel{d}{=}|f(z)|\epsilon$ conditionally on $z$. Therefore the conditional law of $x$ given $z$ is the law of $g(z)+|f(z)|\epsilon$, which implies
$$F_{x\mid z}(x)=F_\epsilon\left(\frac{x-g(z)}{|f(z)|}\right)$$

Now, by Theorem \ref{teobarycenter1d}, the optimal transport from the law of $x$ conditioned on $z$ to the barycenter law is
\begin{align}
    T_z(x)&=F_Y^{-1}(F_{x\mid z}(x))\nonumber\\
    &=F_Y^{-1}\left(
F_\epsilon\left(\frac{x-g(z)}{|f(z)|}\right)\right)\nonumber\\
&=\mathbb{E}[g(z)]+
\mathbb{E}[|f(z)|]\,
F_\epsilon^{-1}\left(
F_\epsilon\left(\frac{x-g(z)}{|f(z)|}\right)\right)\nonumber\\
&=\mathbb{E}[g(z)]+
\mathbb{E}[|f(z)|] \frac{x-g(z)}{|f(z)|}\quad\text{ (atomlessness of $\epsilon$)}\nonumber\\
&=\mathbb{E}[g(z)]+
\mathbb{E}[|f(z)|]\operatorname{sgn}(f(z))\epsilon
\end{align}
 because
$$\frac{x-g(z)}{|f(z)|}=\operatorname{sgn}(f(z))\epsilon$$

Finally, since $\epsilon$ is symmetric and independent of $z$, 
$\operatorname{sgn}(f(z))\epsilon$ is independent of $z$. Hence $y=T_z(x)$ is independent of $z$ and its law is the barycentric law found above with quantile function $F_Y^{-1}$.

\end{proof}

\section{Proof of Lemma \ref{lemmauniquegarch}}

\begin{proof}
For any $G$ in $P$ the induced solution to the recurrence $h_t^{(G)}$ is
\begin{equation}\label{recu1}
h_t^{(G)}=\sum_{i=1}^\infty C^{i-1}(A+B x_{t-i}^2)=\frac{A}{1-C}+B\sum_{i=1}^\infty C^{i-1} x_{t-i}^2
\end{equation}
We know that $\sigma_t^2$ is a solution, and
\begin{equation}\label{recu2}\sigma_t^2=\frac{\omega}{1-\beta}+\alpha\sum_{i=1}^\infty \beta^{i-1} x_{t-i}^2
\end{equation}
If $C=\beta$ then $\sigma_t^2=f(h_t^{(G)})$, where $f(z)=\frac{\omega}{1-\beta}+\alpha B^{-1}\left(z-\frac{A}{1-C}\right)$, so $y_t{\perp\!\!\!\perp} h_t^{(G)}$ iff $y_t{\perp\!\!\!\perp} \sigma_t^2$, which implies $G$ is a solution.

It remains to prove that any solution of the restricted problem is of this form. Suppose then that $G(x,h)=A+Bx^2+Ch$ is a solution and let's prove that $C=\beta$. Suppose $C\neq \beta$. To reach a contradiction it suffices to show that $\mathbb{E}[|x_t-y_t^*|^2]>\mathbb{E}[|x_t-y_t^{(G)}|^2]$, where $y^*_t=\epsilon_t\mathbb{E}[\sigma_t]$ and $y_t^{(G)}$ is the solution to
\begin{equation}\label{eqalt}\min_y\mathbb{E}[|x_t-y_t|^2]\quad\text{ s.t. }\quad y_t{\perp\!\!\!\perp}h_t^{(G)}
\end{equation}
since this contradicts the assumption that $G$ maximizes the objective. It will suffice to show that $y_t^{(G)}\neq y_t^*$, since a barycenter problem like (\ref{eqalt}) has a unique solution and $y_t^*$ satisfies the constraint because $\epsilon_t\mathbb{E}[\sigma_t]\perp\!\!\!\perp h_t^{(G)}$.

Assume then, for the sake of contradiction, that $y_t^{(G)}= y_t^*$ and write $h_t=h_t^{(G)}$. Notice that $x_t$ is atomless conditioned on $h_t^{(G)}$, since $\epsilon_t$ is atomless and independent of the past, and, hence, of $h_t^{(G)}$, making $x_t=\sigma_t\epsilon_t$ conditioned on $h_t^{(G)}$ a mixture of continuous random variables. From Theorem \ref{teobarycenter1d} we then have
$$y_t^{(G)}=F_Y^{-1}\left(F_{x\mid h_t}(x_t)\right)$$
where 
$$F_Y^{-1}(u) = \mathbb{E} \left[ F_{x\mid h_t}^{-1}(u) \right]\quad\text{ is the quantile function of $y_t^{(G)}$}$$
Since we assumed $y_t^{(G)} = \mathbb{E}[\sigma_t] \epsilon_t$, the cdf $F_Y$ is simply the distribution of the scaled noise, $F_Y(y) = F_\epsilon \left( \frac{y}{\mathbb{E}[\sigma_t]} \right)$. Substitute and invert this $F_Y^{-1}$ and our assumed $y_t^{(G)}$ into the optimal transport map to get
$$F_\epsilon(\epsilon_t) = F_{x|h_t}(x_t) = F_{x|h_t}(\sigma_t\epsilon_t)$$
Because $\epsilon_t$ is independent of both $\sigma_t$ and $h_t$ it follows that $F_\epsilon(z) =  F_{x\mid h_t}(\sigma_tz)$ for all $z$\footnote{If follows from Fubini that if $g(X,Y)=0$ a.s. for two independent random variables $X$ and $Y$ with distributions $\mu_X$ and $\mu_Y$ then for $\mu_Y$-a.e. $y$ $g(X,y)=0$ $\mu_X$-a.s. Since $F_\epsilon$ and $F_{x\mid h_t}$ are continuous, a.e. $z$ extends to all $z$.}. Writing $x=\sigma_t z$ we get
$$F_{x\mid h_t}(x) = F_\epsilon \left( \frac{x}{\sigma_t} \right)$$
Now take an interval $(-a,a)$ of 0 such that $F_\epsilon$ is strictly increasing on it. We can take a fixed $0<x_0<a\sqrt{\omega}$ so that $x_0/\sigma_t\in(-a,a)$ a.s., and we can then use the inverse of the cdf to get 
$$\sigma_t =  \frac{x_0}{F_\epsilon^{-1}\left(F_{x\mid h_t}(x_0)\right)} $$
which shows that $\sigma_t$ is $h_t$-measurable, and hence $\sigma_t^2$ is $h_t$-measurable. Let us now prove this implies $C= \beta$, which will complete the proof.

Let $X = (x_{t-1}^2, x_{t-2}^2, x_{t-3}^2, \dots)$ denote the infinite random vector of past squared returns. Notice the support of $X$ lies within the strictly positive cone $\mathbb{R}_+^{\mathbb{N}}$.

Because $A,B,\omega,\alpha$ are constants we can remove them from equations (\ref{recu1}) and (\ref{recu2}) without affecting measurability, and our condition is equivalent to saying $L_0(X)$ is $L(X)$-measurable, where $L_0$ and $L$ are linear functionals defined on the space of trajectories:
$$L_0(X)= \sum_{i=1}^\infty \beta^{i-1} X_i,\quad L(X)= \sum_{i=1}^\infty C^{i-1} X_i$$

Write
$$L(X)=X_1+C X_2+R_C$$
$$L_0(X)=X_1+\beta X_2+R_\beta$$
where $R_C=\sum_{i=3}^{\infty}C^{i-1}X_i$, $R_\beta=\sum_{i=3}^{\infty}\beta^{i-1}X_i$

Condition on the deeper past $\mathcal{P}=\sigma(x_{t-3},x_{t-4},\dots)$. Since the innovations are absolutely continuous and the volatility is a.s. nonzero, the law of $(X_1,X_2)$ conditioned on $\mathcal{P}$ is nondegenerate and has a density on $\mathbb{R}_+^2$.

On this conditional space, $R_C,R_\beta$ are fixed, and 
$$L_0(X)=L(X)+(\beta-C)X_2+(R_\beta-R_C)$$
If $C\neq\beta$, then $L_0(X)$ still varies with $X_2$ along a level set of $L(X)$, because a fixed $L(X)=c$ only fixes $X_1+CX_2$. Since the conditional law of $(X_1,X_2)$ has a density on $\mathbb{R}_+^2$, 
$$\operatorname{Var}(X_2|\mathcal{P},X_1+CX_2)>0$$
on a set of positive probability, implying that either $\beta=C$ or
$$\operatorname{Var}(L_0(X)|\mathcal{P},X_1+CX_2)>0,$$
which contradicts that $L_0(X)$ is measurable with respect to $L(X)$. Therefore $C=\beta$.
\end{proof}

\section{Proof of Theorem \ref{teorankone} and Theorem \ref{teorankk}}

\begin{proof}
Theorem \ref{teorankone} is just Theorem \ref{teorankk} for $k=1$ its proof, so we need only prove the latter.

Let
$U:=F_{X\mid Z}(X)$. By atomlessness of the conditional laws,
$$U\sim\operatorname{Unif}(0,1),\qquad U\perp\!\!\!\perp Z,\quad\text{ and }X=Q_Z(U)\quad\text{a.s.}$$
By the one-dimensional barycenter theorem, the Monge solution is
$Y^*=\bar Q(U)$
and its value is
$V=\mathbb E[|X-Y^*|^2]$

Assume first that for some $(\psi,\varphi)\in\mathcal{A}_k$ there is a representation
$$D(z,u)=\sum_{i=1}^k\psi^{(i)}(z)r_i(\bar Q(u)),
\quad\text{ where }
(\varphi^{(i)})'(y)=2r_i(y),\quad i=1,\ldots,k.$$
For fixed $(z,u)$ set
$$y^*=\bar Q(u)$$
Then
$$Q_z(u)=y^*+\sum_{i=1}^k\psi^{(i)}(z)r_i(y^*)=T_z(y^*)$$
For arbitrary $y$,
\begin{align}
\ell_{\psi,\varphi}(z,u;y)-\ell_{\psi,\varphi}(z,u;y^*)
&=|Q_z(u)-y|^2-|Q_z(u)-y^*|^2+\sum_{i=1}^k\psi^{(i)}(z)\left(\varphi^{(i)}(y)-\varphi^{(i)}(y^*)\right)\nonumber\\
&=2\int_{y^*}^{y}\left(s-Q_z(u)+\sum_{i=1}^k\psi^{(i)}(z)r_i(s)\right)ds\nonumber\\
&=2\int_{y^*}^{y}\left(s-y^*+\sum_{i=1}^k\psi^{(i)}(z)\left(r_i(s)-r_i(y^*)\right)
\right)ds\nonumber\\
&=2\int_{y^*}^{y}\left(T_z(s)-T_z(y^*)\right)ds
\end{align}
Since $T_z$ is nondecreasing, the last integral is nonnegative. Hence
$$\ell_{\psi,\varphi}(z,u;y)\geq \ell_{\psi,\varphi}(z,u;y^*)$$
for all $y$ and for a.e. $(z,u)$.

Thus the pointwise Lagrangian is bounded below by its value at $Y^*$. Since $(\psi,\varphi)$ is admissible, $Y^*\in L^2$, and $Y^*\perp\!\!\!\perp Z$, each product $\psi^{(i)}(Z)\varphi^{(i)}(Y^*)$ is integrable. Therefore
$$\ell_{\psi,\varphi}(Z,U;Y^*)=
|X-Y^*|^2+
\sum_{i=1}^k\psi^{(i)}(Z)\varphi^{(i)}(Y^*)$$
is integrable, and the expectation of $\ell_{\psi,\varphi}$ is well-defined from below as an extended-real functional.

Since $Y^*$ minimizes $\ell_{\psi,\varphi}$ pointwise, it follows that
$$\inf_Y\mathbb E[\ell_{\psi,\varphi}(Z,U;Y)]
=\mathbb E[\ell_{\psi,\varphi}(Z,U;Y^*)]$$
But $Y^*$ is independent of $Z$, and each $\psi^{(i)}(Z)$ is centered. Hence, by admissibility,
$$
\mathbb E\left[\sum_{i=1}^k\psi^{(i)}(Z)\varphi^{(i)}(Y^*)\right]
=\sum_{i=1}^k\mathbb E\left[\psi^{(i)}(Z)\varphi^{(i)}(Y^*)
\right]=0
$$
Therefore
$$\inf_Y\mathbb E[\ell_{\psi,\varphi}(Z,U;Y)]
=\mathbb E[|X-Y^*|^2]=V$$

Now let $(\tilde\psi,\tilde\varphi)\in\mathcal A_k$ be arbitrary. Since $Y^*$ is square-integrable and independent of $Z$, admissibility gives
$$
\mathbb E\left[\sum_{i=1}^k\tilde\psi^{(i)}(Z)\tilde\varphi^{(i)}(Y^*)\right]=0
$$
and therefore
$$
\inf_Y\mathbb E\left[|X-Y|^2+\sum_{i=1}^k\tilde\psi^{(i)}(Z)\tilde\varphi^{(i)}(Y)\right]
\leq
\mathbb E\left[|X-Y^*|^2+\sum_{i=1}^k\tilde\psi^{(i)}(Z)\tilde\varphi^{(i)}(Y^*)\right]=V
$$
Taking the supremum over $(\tilde\psi,\tilde\varphi)\in\mathcal A_k$ gives
$$
\sup_{(\tilde\psi,\tilde\varphi)\in\mathcal A_k}\inf_Y
\mathbb E\left[|X-Y|^2+\sum_{i=1}^k\tilde\psi^{(i)}(Z)\tilde\varphi^{(i)}(Y)\right]\leq V
$$
The pair $(\psi,\varphi)$ constructed above belongs to $\mathcal A_k$ and has inner value equal to $V$, so the reverse inequality also holds. Thus
$$\sup_{(\psi,\varphi)\in\mathcal A_k}\inf_Y
\mathbb E\left[|X-Y|^2+\sum_{i=1}^k\psi^{(i)}(Z)\varphi^{(i)}(Y)\right]=V$$

Conversely, suppose that
$$\sup_{(\psi,\varphi)\in\mathcal A_k}\inf_Y
\mathbb E\left[|X-Y|^2+\sum_{i=1}^k\psi^{(i)}(Z)\varphi^{(i)}(Y)\right]=V$$
and that the supremum is realized by an admissible pair $(\psi,\varphi)$. Since $Y^*$ is independent of $Z$ and $(\psi,\varphi)$ is admissible,
$$\mathbb E\left[|X-Y^*|^2+\sum_{i=1}^k\psi^{(i)}(Z)\varphi^{(i)}(Y^*)\right]=V$$
The realized inner infimum is also equal to $V$, so $Y^*$ minimizes the realized inner problem:
$$Y^*\in\arg\inf_Y
\mathbb E\left[|X-Y|^2+\sum_{i=1}^k\psi^{(i)}(Z)\varphi^{(i)}(Y)\right]$$
Equivalently, in the quantile coordinates $(Z,U)$, the random variable
$Y^*=\bar Q(U)$ minimizes
$$Y\longmapsto \mathbb E[\ell_{\psi,\varphi}(Z,U;Y)]$$
We claim that this implies pointwise minimality:
$$\bar Q(u)\in \operatorname*{arg\,min}_{y\in\mathbb R}
\ell_{\psi,\varphi}(z,u;y)\quad\text{ for a.e. $(z,u)$}$$

Indeed, if this failed, then, since
$y\mapsto \ell_{\psi,\varphi}(z,u;y)$ is continuous, there would exist
$q\in\mathbb Q$, $n\in\mathbb N$, and a measurable set $A$ of positive probability such that
$$\ell_{\psi,\varphi}(Z,U;q)
\leq
\ell_{\psi,\varphi}(Z,U;\bar Q(U))-\frac{1}{n}
\qquad\text{on }A$$
Define
$$\widetilde Y:=q\,{\bf 1}_A+Y^*\,{\bf 1}_{A^c}$$

Then $\widetilde Y\in L^2$. Moreover, on $A$,
$$\ell_{\psi,\varphi}(Z,U;q)=|X-q|^2+ \sum_{i=1}^k\psi^{(i)}(Z)\varphi^{(i)}(q)$$
which is integrable, since $X\in L^2$, each $\psi^{(i)}(Z)\in L^1$, each $\varphi^{(i)}(q)$ is a constant, and $k$ is finite. Moreover,
$$\ell_{\psi,\varphi}(Z,U;\widetilde Y)=\ell_{\psi,\varphi}(Z,U;q){\bf 1}_A+\ell_{\psi,\varphi}(Z,U;Y^*){\bf 1}_{A^c}$$
and both terms are integrable. Hence
\begin{align}
\mathbb E[\ell_{\psi,\varphi}(Z,U;\widetilde Y)]
&\leq\mathbb E[\ell_{\psi,\varphi}(Z,U;Y^*)]
-\frac{\mathbb P(A)}{n}\nonumber\\
&<\mathbb E[\ell_{\psi,\varphi}(Z,U;Y^*)]\nonumber
\end{align}
contradicting the minimality of $Y^*$.

Therefore
$$\bar Q(u)\in \operatorname*{arg\,min}_{y\in\mathbb R}
\ell_{\psi,\varphi}(z,u;y)\quad\text{ for a.e. $(z,u)$ }$$

Since each $\varphi^{(i)}$ is differentiable, the function
$$y\longmapsto|Q_z(u)-y|^2+\sum_{i=1}^k\psi^{(i)}(z)\varphi^{(i)}(y)$$
is differentiable. Therefore its derivative vanishes at $y=\bar Q(u)$.
Thus
$$0=2(\bar Q(u)-Q_z(u))+\sum_{i=1}^k\psi^{(i)}(z)(\varphi^{(i)})'(\bar Q(u))\quad\text{ for a.e. $(z,u)$}$$

Hence
$$D(z,u)=Q_z(u)-\bar Q(u)=\frac12\sum_{i=1}^k\psi^{(i)}(z)(\varphi^{(i)})'(\bar Q(u))\quad\text{ for a.e. $(z,u)$}$$
\end{proof}

\chapter{Implementation details for the historical data experiments}
\label{appimplementation}

\section{Data and preprocessing}

Bitcoin 4-hour prices were retrieved from Binance monthly archive\footnote{\url{https://data.binance.vision/data/spot/monthly/klines/BTCUSDT/4h/}, accessed 4 July 2026} and concatenated into a single .csv file. The SP500 data was downloaded from FRED\footnote{\url{https://fred.stlouisfed.org/series/SP500}, accessed 1 July 2026.}. The NAO data were the daily normalized CDAS 500-hPa index published by the NOAA Climate Prediction Center\footnote{\url{https://ftp.cpc.ncep.noaa.gov/cwlinks/norm.daily.nao.cdas.z500.19500101_current.csv}, accessed 29 June 2026.}.

For Bitcoin and the SP500 we formed log returns

$$r_t=\log P_t-\log P_{t-1}$$

For the NAO index we used the given level values.

For all datasets we used a burn-in period of 500 observations to `warm up' the recurrent state. For Bitcoin and the NAO index the training sample size (after discarding the warmup observations) was $4500$, while the test sample size was $1000$ observations. The SP500 file had only $2511$ observations, so we split it into $2009$ training (including burn-in) and $502$ test observations. The Bitcoin training and test periods run from October 2023 to January 2026 and from January to June 2026, respectively. For the SP500 they run from July 2016 to June 2024 and from June 2024 to June 2026, and for NAO from January 2010 to September 2023 and from October 2023 to June 2026. The Bitcoin and SP500 files end on June 30th 2026, while NAO ends on 26th June 2026.

The normalization for all series was

$$\widetilde x_t=\frac{x_t-\bar x_{\mathrm{tr}}}{s_{\mathrm{tr}}}$$

where $\bar x_{\mathrm{tr}}$ and $s_{\mathrm{tr}}$ are the mean and standard deviation of the training block. The same transformation was applied to the test block and to each of the alternative models.

\section{RBE implementation and alternative models}

The recurrent function $G_\theta$ was implemented as a GRU with hidden dimension $d=4$ and a $k$-dimensional linear readout.

For the choice of neural test functions with $k=3$ each $\varphi^{(i)}$ was implemented as a MLP with two hidden layers of dimension 7 and $\tanh$ activations.

For the choice with $k=2$ monomial test functions notice the first order conditions of the inner minimization yield an affine function of $y$, which allow for $y$ to be cleared and computed exactly, which we did. For the choice with neural $\varphi$ we approximated the solution with LBFG. The outer loop used $200$ Adam updates for $(\lambda,\theta)$, with learning rates $10^{-2}$ for $\theta$ and $10^{-3}$ for $\lambda$. For each dataset we trained five different restarts with seeds $101$, $202$, $303$, $404$ and $505$, and their predictive cdfs were then averaged to form the ensemble predictive cdf

$$\widehat F_t^{\mathrm{ens}}(x)=\frac15\sum_{r=1}^5\widehat F_{t,r}(x)$$

For comparison we used the following alternative distributions:

\begin{itemize}
\item the empirical distribution of the full training block
\item rolling empirical distributions and zero-mean Gaussian distributions with rolling standard deviations, using rolling windows of size $250,500$ and $1000$
\item KNN empirical distributions using lags $p=3,5,10$ and neighborhood sizes $K=50,100,250$
\item a Gaussian GARCH$(1,1)$ model of constant mean for Bitcoin and the SP500
\item AR$(p)$ models with $p=1,4,7,20$ for NAO
\end{itemize}

For the KNN, the predictive cdf was

$$\widehat F_{t,p,K}^{\mathrm{KNN}}(x)=\frac1K\sum_{s\in\mathcal N_{p,K}(t)}\mathbf 1_{{\widetilde x_s\leq x}}$$

where $\mathcal N_{p,K}(t)$ contains the indices $s$ of the $K$ $p$-lag history vectors $H_s=( x_{s-1},\dots, x_{s-p})$ that are nearest to the current vector of $p$ lags $(x'_{t-1},\dots,x'_{t-p})$.

We used quasi MLE to fit the GARCH model, after which 3000 equidistributed quantile points from the standard Normal distribution were used as residuals to form an empirical predictive distribution. We used least squares to fit the AR models, after which their training residuals were added to the predicted conditional mean to form an empirical predictive distribution. All models and empirical distributions were fitted using only the training sample. 

The best alternative reported in Subsection \ref{predictive} is the candidate that had the lowest mean CRPS on the test sample.

For the second comparison against a recurrent location-scale benchmark, we trained the model 

$$
X_{t+1}=a_t+b_t\varepsilon_{t+1},\qquad \varepsilon_{t+1}\sim N(0,1),
$$

with the Gaussian quasi log-likelihood objective
$$
L=\frac{1}{T_{\mathrm{train}}-500}
\sum_{t=500}^{T_{\mathrm{train}}-1}
\left(\log b_t+\frac{(x_{t+1}-a_t)^2}{2b_t^2}\right).
$$
To imitate the RBE setup, we did five restarts with the same five seeds as the RBE and trained each restart for $1000$ Adam iterations with learning rate decreasing from $0.01$ to $0.001$. We then kept the iterate with the lowest $L$ on the training set. As with the RBE models, we built the ensemble as an equally weighted mixture of each of the five restarts' predictive distributions, given as $N(a_{
t,r},b_{t,r}^2)$. The CRPS was then computed analytically assuming Gaussian innovations $\epsilon_t$. We also ran a version of the predictive distribution of the recurrent location-scale ensemble in which we sampled 3000 observations by bootstrapping the training residuals and then computed the finite-sample CRPS. While the obtained CRPS values differed slightly from those obtained analytically assuming Gaussianity, the conclusions were the same: the recurrent location scale ensemble had a statistically significant advantage on the NAO dataset, while the RBE had a statistically distinguishable advantage on Bitcoin data and there was no detectable difference on the SP500 dataset.

\clearpage
\bibliographystyle{acm}
\bibliography{biblio}

@Article{Yang22,
  author={H. Yang and E. Tabak},
  title={Conditional Density Estimation, Latent Variable Discovery, and Optimal Transport},
  journal={Comm. Pure Appl. Math.},
  year={2022},
  volume={75},
  number={3},
  pages={610--663}
}

@Book{Galichon,
  author    = {A. Galichon},
  title     = {Optimal Transport Methods in Economics},
  publisher = {Princeton University Press},
  year      = {2016}
}

@article{Tabak25,
  doi = {10.48550/ARXIV.2507.03669},
  author = {A. Lipnick and E. Tabak and G. Trigila and Y. Wang and X. Ye and W. Zhao},
  title = {The {M}onge optimal transport barycenter problem},
  journal = {arXiv},
  publisher = {Cornell University},
  year = {2025},
  eprint       = {math.CO/2507.03669},
  pages = {2507.03669},
  pubstate      = {\bibstring{prepublished}},
}

@article{Graves14,
  author = {A. Graves},  
  title = {Generating sequences with recurrent neural networks},
  journal = {arXiv},
  year = {2014},
  pages = {1308.0850},  
  copyright = {arXiv.org perpetual, non-exclusive license}
}

@Article{Pass,
  author={B. Pass},
  title={Optimal transportation with infinitely many marginals},
  journal={J. Func. Anal.},
  year={2013},
  volume={264},
  number={4},
  pages={947--963}
  
}

@article{SalinasEtAl20,
  author  = {Salinas, David and Flunkert, Valentin and Gasthaus, Jan
             and Januschowski, Tim},
  title   = {{DeepAR}: Probabilistic Forecasting with Autoregressive
             Recurrent Networks},
  journal = {International Journal of Forecasting},
  volume  = {36},
  number  = {3},
  pages   = {1181--1191},
  year    = {2020},
  doi     = {10.1016/j.ijforecast.2019.07.001}
}

@inproceedings{RasulEtAl21,
  author    = {Rasul, Kashif and Sheikh, Abdul-Saboor and Schuster, Ingmar
               and Bergmann, Urs and Vollgraf, Roland},
  title     = {Multivariate Probabilistic Time Series Forecasting via
               Conditioned Normalizing Flows},
  booktitle = {International Conference on Learning Representations},
  year      = {2021},
  url       = {https://openreview.net/forum?id=WiGQBFuVRv}
}

@inproceedings{SiEtAl22,
  author    = {Si, Phillip and Bishop, Allan and Kuleshov, Volodymyr},
  title     = {Autoregressive Quantile Flows for Predictive Uncertainty
               Estimation},
  booktitle = {International Conference on Learning Representations},
  year      = {2022},
  url       = {https://openreview.net/forum?id=z1-I6rOKv1S}
}

@article{meitz08,
  author  = {Meitz, Mika and Saikkonen, Pentti},
  title   = {Stability of Nonlinear {AR--GARCH} Models},
  journal = {Journal of Time Series Analysis},
  volume  = {29},
  number  = {3},
  pages   = {453--475},
  year    = {2008},
  doi     = {10.1111/j.1467-9892.2007.00562.x}
}

@Article{Tucker59,
  author={H. G. Tucker},
  title={A generalization of the {G}livenko-{C}antelli theorem},
  journal={Ann. Math. Statist.},
  year={1959},
  volume={30},
  issue={3},
  pages={828--830}  
}

@Article{Rosenblatt52,
  author={M. Rosenblatt},
  title={Remarks on a multivariate transform},
  journal={Ann. Math. Statist.},
  year={1952},
  volume={23},
  pages={470--472}  
}

@article{LeGouicLoubes17,
  author  = {Le Gouic, Thibaut and Loubes, Jean-Michel},
  title   = {Existence and Consistency of {Wasserstein} Barycenters},
  journal = {Probability Theory and Related Fields},
  year    = {2017},
  volume  = {168},
  number  = {3--4},
  pages   = {901--917}
}

@article{BigotEtAl17,
  author  = {Bigot, J{\'e}r{\'e}mie and Gouet, Ra{\'u}l and
             Klein, Thierry and L{\'o}pez, Alfredo},
  title   = {Geodesic {PCA} in the {Wasserstein} Space by Convex {PCA}},
  journal = {Annales de l'Institut Henri Poincar{\'e},
             Probabilit{\'e}s et Statistiques},
  year    = {2017},
  volume  = {53},
  number  = {1},
  pages   = {1--26}
}

@book{Santambrogio15,
  author    = {Santambrogio, Filippo},
  title     = {Optimal Transport for Applied Mathematicians:
               Calculus of Variations, {PDE}s, and Modeling},
  series    = {Progress in Nonlinear Differential Equations and
               Their Applications},
  volume    = {87},
  publisher = {Birkh{\"a}user},
  address   = {Cham},
  year      = {2015}
}

@article{Ma26,
  author  = {Ma, Jianyu},
  title   = {Absolute Continuity of {Wasserstein} Barycenters on
             Manifolds with a Lower {Ricci} Curvature Bound},
  journal = {Calculus of Variations and Partial Differential Equations},
  year    = {2026},
  volume  = {65},
  number  = {1},
  pages   = {9},
  note    = {Article 9}
}

@book{Villani03,
  author    = {Villani, C{\'e}dric},
  title     = {Topics in Optimal Transportation},
  series    = {Graduate Studies in Mathematics},
  volume    = {58},
  publisher = {American Mathematical Society},
  address   = {Providence, RI},
  year      = {2003}
}

\end{document}